\documentclass[a4paper,12pt]{amsart}
\usepackage{amsmath,amsfonts,amssymb,amsthm}
\usepackage{graphics}
\usepackage{epsfig}
\usepackage{esint}
\usepackage{shuffle}
\usepackage{istgame}
\usepackage{mathrsfs}
\usepackage{vmargin}
\setpapersize{A4}
\setmargins{2.5cm}{1.5cm}{16.5cm}{23.42cm}{10pt}{1cm}{0pt}{2cm}

\usepackage{enumerate}

\newcommand{ \E }{ \mathbb{E} }

\newcommand{\ignore}[1]{}

\usepackage[hidelinks,hyperindex,breaklinks]{hyperref}

\newtheorem{proposition}{Proposition}
\newtheorem{theorem}{Theorem}
\newtheorem{remark}{Remark}
\newtheorem{lemma}{Lemma}
\newtheorem{corollary}{Corollary}

\newtheorem{approximation}{Approximation}

\title[Critical Drift-Diffusion Equation]{ A Critical Drift-Diffusion Equation: Intermittent Behavior via Geometric Brownian Motion on $ \textbf{SL}(n)$ }
\author[P. Morfe]{Peter S.~Morfe}
\author[F. Otto]{Felix Otto}
\author[C. Wagner]{Christian Wagner}

\begin{document}

\maketitle

\begin{abstract}
This paper concerns the so-called diffusion in the curl of the 2d Gaussian free field, and its generalization to higher dimensions $n \geq 2$, building on the scale-by-scale homogenization approach developed recently by Chatzigeorgiou, Morfe, Otto, and Wang \cite{CMOW}.  It begins by reformulating the approximation scheme of that work in terms of SDEs in the length scale $L$.  This exposes an unexpected connection with a certain geometric Brownian motion on the special linear group $\textbf{SL}(n)$.  The analysis of this process sheds light on the original problem, particularly as it pertains to intermittent behavior exhibited by the (averaged) Lagrangian coordinate.
\end{abstract}

\medskip

This preprint subsumes the previous preprints \cite{MOW24,OW24,MOW24Pt2} of the authors.


\subsection{The drift-diffusion process $X$ for a Gaussian ensemble of divergence-free
drifts $b$, and its expected position $u$}\label{ss:X}
The main result of this text connects two seemingly unrelated objects:
On the one hand, we consider the drift-diffusion process
\begin{align}\label{ao02}
dX=b(X_t)dt+\sqrt{2}dW
\end{align}
in the $ n $-dimensional Euclidean plane,
where $W$ denotes the standard Brownian motion\footnote{The factor of $\sqrt{2}$ in (\ref{ao02})
avoids a factor $\frac{1}{2}$ in (\ref{ao04}).}. We assume that the time-independent drift $b$, a tangent vector field, is divergence-free:
\begin{align}\label{ao24}
\nabla.b=0,
\end{align}
where the dot denotes the natural pairing between the differential $\nabla$,
which we consider a cotangent vector, and the tangent vector\footnote{In view of the upcoming tensor analysis and
elementary representation theory, we are
deliberately careful in distinguishing tangent and cotangent vectors.} $b$. Next to \eqref{ao24}, our main structural assumptions are that $ b $ is an isotropic\footnote{\label{footnoteisotropy}By isotropy we mean that for any rotation $ Q $ of tangent vectors, passing from the variables $ ( \widehat{x}, \widehat{b} ) $ with $ \widehat{b} = \widehat{b} ( \widehat{x} ) $ to the variables $ ( x ,  b ) := ( Q \widehat{ x } , Q \widehat{b} ) $ with $ b = b ( x ) $ does not change the law of the drift field, i.~e.~$ b =_{ \rm law } \widehat{b} $.}, stationary Gaussian field. Incidentally, these assumptions imply centeredness of $ b $.
Under these assumptions we seek to study the scaling critical situation that we are going to explain next.

\medskip

Can the process $ X $ share the parabolic scale invariance in law
of the Wiener process $W$, namely $\lambda W_t=_{law}W_{\lambda^2 t}$?
We note that formally we have\footnote{In this context, ``annealed law'' refers to the expectation
w.~r.~t.~both the environment $b$ and the thermal noise $W$.}
$\lambda X_t=_{annealed\;law}X_{\lambda^2 t}$
provided $b$ transforms in law like a velocity under this parabolic rescaling, meaning that
\begin{align}\label{ao02b}
	b(\lambda x) =_{law} \frac{\lambda}{\lambda^2} b(x) = \lambda^{-1} b(x) .
\end{align}
Since we work with a centered and stationary Gaussian ensemble, which is characterized by its covariance tensor $c(x-y)$ $=\mathbb{E}b(x)\otimes b(y)$ (an element of the tensor product $(\mbox{tangent space}) \otimes (\mbox{tangent space})$
and a function of the tangent vector $z=x-y$), the scaling criticality is conveniently expressed on the level of the positive-semidefinite Fourier transform
${\mathcal F}c(k)=(2\pi)^{-\frac{ n }{ 2 } }  \int dz e^{ -\iota k.z } c(z)$ of $ c $: $ \mathcal{F} c ~ { \rm is } ~ \text{(} 2 - n \text{)-homogeneous} $. Recall that here the wave vector $ k $ naturally acts as a cotangent vector. Thus our isotropy assumption may be expressed as\footnote{Here and henceforth we write $F^{\dagger}$ for the transpose of an endomorphism $F$.}
\begin{align}\label{ao02c}
	\mathcal{F} c ( Q^{ - \dagger } k ) =  ( Q \otimes Q ) \, \mathcal{F} c ( k ) 
	\quad { \rm for ~ any ~ rotation } ~ Q ~ { \rm on ~ tangent ~ space } ,
\end{align}
whilst \eqref{ao24} takes form
\begin{align}\label{ao24repeat}
	\mathcal{F}c(k) . k \otimes \xi = 0
	\quad { \rm for ~ all ~ cotangent ~ vectors } ~ \xi .
\end{align}

\medskip

Let us recall that the scaling assumption \eqref{ao02b} states that zooming in on small scales amounts to a rescaling of the amplitude, which is only compatible with Hölder regularity of order $ - 1 $. This is too rough to make sense of the SDE \eqref{ao02} as can be seen on the level of the generator $  \mathcal{L} u = b.\nabla u + \Delta u $: In view of the transport term $b.\nabla u$, a typical $ \mathcal{L} $-harmonic function $u$ will be H\"older continuous of exponent $ - 1 + 2 = 1 $ (where $2$ is the order of the equation
and $-1$ comes from the scale invariance of $ b $).
However, this regularity is not sufficient to give a sense to
$b.\nabla u$, even when rewriting it as $\nabla. ub$,
since the sum of the orders of regularity of $u$ and $b$ is (borderline) non-positive.
In fact, we establish that $u$ is Lipschitz continuous up to a logarithm,
see the discussion after Theorem \ref{thm:mr02}.

\medskip

To resolve this issue, we impose $ b $ to have a small-scale regularization/ultra-violet (UV) cut-off. W.~l.~o.~g.~we place this cut-off at frequency one. Of course this requirement, which may be phrased as
\begin{align}\label{ao02d}
	\mathcal{F} c ( k ) ~ { \rm is } ~ ( 2 - n ) { \rm \text{-}homogeneous  } ~ { \rm for } ~ | k | \leq 1  ,
\end{align}
is incompatible with the invariance \eqref{ao02b}.

\medskip

Elementary representation theory, see Corollary \ref{corRep01} in Appendix \ref{reptheory}, implies that \eqref{ao02c}, \eqref{ao24repeat} and \eqref{ao02d} determine $ \mathcal{F} c $ up to a constant. We fix this (positive) constant by requiring that\footnote{This definition of $\varepsilon$ ensures that the (approximate) effective diffusivity $ \tilde\lambda_L $ introduced in \eqref{lambda} is independent of the dimension $ n $.}
\begin{align}\label{cov01fo}
	\E | b |^2 = { \rm tr } \, c ( 0 ) =: \varepsilon^2  \frac{ n }{ 4 } ,
\end{align}
where $ 0 < \varepsilon \ll 1 $ plays the role of the Péclet number\footnote{The Péclet number
is the traditional measure of the relative strength of convection w.~r.~t.~diffusion.}: In view of the UV cut-off in \eqref{ao02d} the field $ b $ has acquired a characteristic
length scale, and thus we may associate a characteristic time scale to both diffusion and convection.
In view of the unit diffusivity, and the convection with a drift of order $ \varepsilon $, these characteristic
time scales are of the order $ 1 $ and $ \varepsilon^{-1} $, respectively. Hence the parameter $ \varepsilon $ plays the role of
the Péclet number, which is defined to be the ratio of the two time scales. 

\medskip

Corollary \ref{corRep01} also implies that $\mathcal{F}c(k)$ is explicitly given by\footnote{Here and in the following we use Einstein's summation convention in the strict interpretation that a summation is applied to every pair of upper and lower indices having the same label.}
\begin{align}\label{ao08}
	{\mathcal F}c(k) = { \rm constant } ( n ) \times \varepsilon^2 I(|k|\le 1) \Big( \delta^{ i j } e_{i} \otimes e_{j} - \frac{k^*\otimes k^*}{|k|^2} \Big) \frac{ 1 }{ | k |^{ n - 2 } } ,
\end{align}
where $\{ e_i \}$ denotes the Cartesian basis of tangent space and
$k^*$ is the tangent vector related to the cotangent wave vector $k$
by the Euclidean inner product. The reader may easily check that \eqref{ao08} satisfies \eqref{ao02c}, \eqref{ao24repeat} and \eqref{ao02d}. In particular, $ b $ is automatically $ \textbf{O} ( n ) $-invariant in law.

\medskip

In the same vein as our ensemble of drifts $ b $ is borderline critical from a regularity point of view, it does not fall into the realm of standard homogenization theory: This paper is motivated by previous work on the asymptotics of the mean-square displacement $\mathbb{E} |X_{t}|^{2}$, which confirmed physicists' predictions \cite{Bouchaud,FisherEtAl,KB}, based on renormalization group theory (see \cite[Appendix A]{CMOW} for heuristics in this vain), that it exhibits the superdiffusive scaling
	\begin{equation*}
		\mathbb{E} |X_{t}|^{2} \sim t \sqrt{\ln t}.
	\end{equation*}
The first mathematical results in this direction were contributed by T\'{o}th and Valk\'{o} \cite{TothValko}, who established (suboptimal) superdiffusive upper and lower bounds.  More recently, Cannizzaro, Haunschmid-Sibitz, and Toninelli \cite{CannizzaroHaunschmidSibitzToninelli} obtained the $t \sqrt{\ln t}$ asymptotics up to double logarithmic corrections.  Both these works were based on an intricate analysis of the resolvent of the environment seen by the particle in Fock space.  Our starting point is the work of Chatzigeorgiou, Morfe, Otto, and Wang \cite{CMOW}, which obtained a refined result using a homogenization approach.

\begin{theorem}[{\cite[Theorem 1.1]{CMOW}}] In dimension $ n = 2 $, for any $T \geq 0$,\footnote{Here and in the following we write $ a \approx b $ as $ \varepsilon \ll 1 $ to mean that for every $ 0 < C < \infty $, there exists an $ \varepsilon_0 \leq 1 $, such that  $ ( 1 - C ) b( \varepsilon)  \leq  a ( \varepsilon ) \leq ( 1 + C )  b( \varepsilon ) $ for $ \varepsilon \leq \varepsilon_0 $.  }
	\begin{equation} \label{E: logarithmic correction}
		\mathbb{E} |X_{T}|^{2} \approx  2n  \lambda(T) T
		\quad \text{as} \quad
		\varepsilon \ll 1,
	\end{equation}
where $\lambda(s)$ is defined by\footnote{for the purpose of this theorem, we could use $ \lambda(s)^2 = 1 + \frac{\varepsilon^{2}}{2} \ln (c + s) $ for any $ 0 < c < \infty $}
	\begin{equation} \label{E: definition of lambda}
		\lambda(s) = \sqrt{1 + \frac{\varepsilon^{2}}{2} \ln (1 + s)}.
	\end{equation}
\end{theorem}

The proof was originally carried out using a scale-by-scale homogenization argument in which the length scale $L$ was discrete. In Section \ref{sec:hom} below, we revisit the proof, this time using an SDE in the variable $L$, which is both simpler and also connects our drift-diffusion process to geometric Brownian motion on $\textbf{SL}(n)$. Since \cite{CMOW} first appeared, Armstrong, Bou-Rabee, and Kuusi \cite{ABK24} improved on this line of work, establishing, again via scale-by-scale homogenization, a quenched invariance principle for $X$. As in \cite{ABK24}, our homogenization result, see Section \ref{sec:hom}, applies in any dimension $ n \geq 2 $, thus extending the theorem also to higher dimensions.

\medskip

In this paper, we focus on the expected position $u(x,t)$
(with respect to the thermal noise $W$) of
the solution $X$ of (\ref{ao02}) with $X(t=0)=x$. It is characterized by the
drift-diffusion equation
\begin{align}\label{ao04}
\partial_tu- \mathcal{L} u = 0, ~ u(x, t=0)=x,
\quad \text{where} \quad
\mathcal{L} := b.\nabla u + \Delta u
\end{align}
denotes the generator of $ X $. The paring $ b. \nabla u $ in the definition of $ \mathcal{L} $ is canonical and corresponds to
the directional (or covariant) derivative of $ u $ in direction $ b $.

\medskip

Our main results on the intermittent behavior of $u$ will be proved by exploiting a connection between our scale-by-scale homogenization approach and geometric Brownian motion on the special linear group $\textbf{SL}(n)$. In view of the divergence condition (\ref{ao24}), a connection between
(\ref{ao02}) and a process on $\textbf{SL}(n)$ is not surprising per se. If for
a moment, we disregard diffusion in (\ref{ao02}) and (\ref{ao04}), we obtain
\begin{align}\label{ao30}
d \nabla u ( 0, t ) = \nabla u ( 0 , t ) \nabla b ( X_t ) dt ,
\quad { \rm where } \quad X_{ t } ~ { \rm starts ~ at ~ the ~ origin } .
\end{align}
A remark on our tensorial notation is in order: Like for $b$, we think of
increments of $u$ as tangent vectors; since we consider the differential $\nabla$
as a cotangent vector, $\nabla u$ is an element
of the tensor space $(\mbox{tangent space}) \otimes(\mbox{cotangent space})$,
which is canonically isomorphic to the endomorphism space ${\rm End}(\mbox{cotangent space})$.
Hence the product in (\ref{ao30}) is a product of endomorphisms. According to Jacobi's formula, (\ref{ao30}) implies
\begin{align}\label{ao31}
d{\rm det}\nabla u=({\rm det}\nabla u)({\rm tr}\nabla b)dt\quad\mbox{along}\quad x=X_{t}.
\end{align}
Because of ${\rm tr}\nabla b$ $=\nabla.b$, we learn from (\ref{ao24})
that $\nabla b(X_t)$ is an element of the Lie algebra $\mathfrak{sl}(n)$, the linear space
of trace-free endomorphisms. We then learn from (\ref{ao31})
that ${\rm det}\nabla u=1$ is preserved, i.~e.~$\nabla u(X_t,t)$ is an element of
$\textbf{SL}(n)$, the space of endomorphisms of determinant one.

\subsection{Geometric Brownian Motion $F$ on the Lie Group $\textbf{SL}(n)$.}\label{ss:F}

On the other hand, we consider geometric Brownian motion on the Lie group $\textbf{SL}(n)$, as determined by the Stratonovich equation
\begin{align}\label{ao01bis}
dF=F_{\tau} \circ \! dB,
\end{align}
where $B$ is a certain Brownian motion on the corresponding Lie algebra $\mathfrak{sl}(n)$.  The tensorial
structure of the r.~h.~s.~of (\ref{ao01bis}) is to be
understood as a product of endomorphisms (of cotangent space), and $\circ$ indicates that we interpret the product in the sense of Stratonovich integration. Equation \eqref{ao01bis} motivates our convention to use $\nabla u$ for the transpose of the ``usual definition''
of the derivative of a map as an element of ${\rm End}(\mbox{tangent space})$;
we opted in favor of it because the left multiplication order in (\ref{ao30}) is more standard
on the stochastic equation level (\ref{ao01bis}) compared to the right multiplication $ \circ d B \, F $.

\medskip

Recall that
the Stratonovich form of (\ref{ao01bis}) is compatible with the chain rule,
so that as in (\ref{ao31}) we have $d{\rm det} F =({\rm det}F)\circ \! d{\rm tr} B$ $=0$.
Hence (\ref{ao01bis})
indeed defines a time-stationary Markov process on $\textbf{SL}(n)$. Due to the left-invariance of the evolution \eqref{ao01bis}, there is no loss in generality assuming
\begin{align} \label{E: initial datum}
F_{ \tau = 0 } = { \rm id } .
\end{align}

\medskip

We specify the Brownian motion $B$ using the following result, which is proved in Section \ref{S: diffusion}.

	\begin{lemma} \label{L: properties of F} There is a unique Brownian motion $B$ on $\mathfrak{sl}(n)$ for which the solution $F$ of \eqref{ao01bis} and \eqref{E: initial datum} has the following properties:
		\begin{itemize}
			\item[(i)] $ \textbf{O} ( n ) $-invariance:
            \begin{align*}
                O F O^{-1} =_{\text{law}} F
                \quad \text{for~all} ~ O \in \textbf{O} ( n ) ,
            \end{align*}
			\item[(ii)] the It\^{o} and Stratonovich interpretations of the product \eqref{ao01bis} coincide:
            \begin{align*}
                F_{\tau} \circ \! dB = F_{\tau} dB ,
            \end{align*}
			\item[(iii)] normalization: $\mathbb{E} F_{\tau} F_{\tau}^{*}  = e^{\tau} \text{id} $.
		\end{itemize}
	\end{lemma}
	
While the Stratonovich form \eqref{ao01bis} is convenient to introduce $ F $ as a $ \textbf{SL}( n ) $-valued process, the Itô form (ii) in the above lemma will be used to obtain estimates. Clearly, the point (iii) depends on the choice of norm. Within this paper, we will always use the Frobenius norm $ | F |^2 = { \rm tr } F F^* $ on $ F \in { \rm End } ( { \rm co tangent ~ space } ) $.\footnote{Here and henceforth $F^{*}$ denotes the adjoint of $F$ with respect to the Euclidean inner product.}
	
\medskip

Our main result shows that the increments of $u$ can be approximated using the flow \eqref{ao01bis}.  It is most conveniently stated in terms of the following class of initial value problems.
Denoting by $\tau$ the independent variable,
for $\tau_*,\tau\ge 0$, we define $F_{\tau_*,\tau}$ as
\begin{align}\label{ao29}
[\tau_*,\infty)\ni\tau\mapsto F_{\tau_*,\tau}
\quad\mbox{solves (\ref{ao01bis}) with initial data}\quad F_{\tau_*,\tau=\tau_*}={\rm id},\nonumber\\
\mbox{extended by}\quad F_{\tau_*,\tau}={\rm id}\quad\mbox{for}\;\tau\le \tau_*.
\end{align}
Equipped with this notation, we may formulate the next result, which shows that the effect of the quenched/environmental noise $b$
can be well-represented by a pseudo thermal noise $B$:

\begin{theorem}\label{thm:mr01}
There exists an explicit coupling of $b$ and $B$, cf.~Lemma~\ref{lemR01}, such that for all $(x,T)$ we have\footnote{Here and in the sequel, the statement $A \lesssim  B$ means that there exists a constant $ C = C(n) < \infty $ (which may depend on $ n $) such that $A \leq C B$, and $A \sim B$ means that both $A \lesssim  B$ and $B \lesssim A$ hold. Subscripts on the symbols $ \lesssim $, $ \sim $ and $ \gtrsim $ indicate a dependence of the constant on further variables.}
\begin{align}\label{ao03}
\mathbb{E}\fint_0^Tdt \,
\frac{1}{|x|^2}|u(x,t)-u(0,t)-F_{\tau(|x|^2),\tau(T)}^\dagger x|^2
\lesssim\varepsilon^2\mathbb{E}|F_{0,\tau(T)}|^2 
\quad \text{as} \quad \varepsilon^2 \ll 1 ,
\end{align}
where
\begin{align}\label{ao12}
\tau(s):=\ln \lambda(s).
\end{align}
\end{theorem}

\medskip

Note that the transposition $F^\dagger$ is a collateral damage of
our definition of the differential $ \nabla $, see the discussion in Subsection \ref{ss:X}.
Of course the same coupling holds for all space-time points $(x,t)$.
Letting $|x|\downarrow 0$ in (\ref{ao03}) it follows that\footnote{that is independently established in \cite[Equation (47)]{MOW24}. }
\begin{align}\label{ao07}
\mathbb{E}\fint_0^Tdt|\nabla u(0,t)-F_{0,\tau(T)}|^2
\lesssim\varepsilon^2\mathbb{E}|F_{0,\tau(T)}|^2,
\end{align}
so that $\nabla u(0,t)\approx F_{0,\tau(T)}$ for $\varepsilon\ll 1$ on average in $t\in(0,T)$.

\medskip

Upon passing to the limit $\varepsilon \downarrow 0$, Theorem \ref{thm:mr01} implies the following qualitative statement: To any $s,r\in (1,\infty)$ and $\hat x\in\mathbb{S}^{n-1}$ we associate $(T,x)$ via
\begin{align*}
T=\exp(\frac{1}{\varepsilon^{2}} ( s^{2} - 1) ) \quad\mbox{and}\quad
x=\exp(\frac{1}{\varepsilon^{2}} ( r^{2} - 1 ) )\hat x.
\end{align*}
Then the following distributional limit holds:
\begin{align} \label{distribution}
\lim_{\varepsilon\downarrow 0} \frac{1}{|x|} \fint_{0}^{T} (u(x,t)-u(0,t)) \, dt = F_{\ln r, \ln s}^\dagger\hat x \quad \text{in law.}
\end{align}
In fact, Theorem \ref{thm:mr01} implies convergence of all the finite-dimensional marginals indexed by $(s,r,\hat{x})$.

\medskip

Note that in view of (\ref{ao07}),  
(\ref{ao03}) seems to take the form of a Taylor remainder estimate.
However, because of the subtle dependence on the radial variable $|x|$ via $\tau(|x|^2)$,
this is only true w.~r.~t.~the angular dependence on $\frac{x}{|x|}$,
or on average over scales of a given order. Roughly speaking, this suggests a picture where, at very large scales, $u$ is well-approximated by an affine function, the slope of which, however, varies with the scale.  The same informal description has been posited in \cite{ABK24}; see the discussion preceding Conjecture E therein.

\medskip

Theorem \ref{thm:mr01} is proved in Section \ref{sec:mr} and \ref{sec:mr2} below by combining the results of \cite{CMOW} with the SDE approach developed here and certain fundamental properties of the geometric Brownian motion $F$.

\medskip

The reader may wonder how the process $ F $, or rather its driver $ B $, depends on the base point that we fixed to be the origin $ 0 $.  Because of the stationarity of the $b$ ensemble, for any blow-up sequence $x := \{x_{\varepsilon}\}_{\varepsilon \downarrow 0}$, we can associate
a Brownian motion $B(x) = \lim_{ \varepsilon \downarrow 0 } B(x_{\varepsilon})$.
In the limit $\varepsilon \downarrow 0$, the law of the resulting Gaussian process can be explicitly characterized.
\medskip

For two blow-up sequences $x$ and $ y := \{ y_\varepsilon \}_{ \varepsilon \downarrow 0 }$, if there is an $r \geq 1$ such that 
	\begin{equation} \label{def_of_r}
		\lim_{ \varepsilon \downarrow 0 } \sqrt{1+\frac{\varepsilon^2}{2}\ln(1+|x_\varepsilon-y_\varepsilon|^2)} = r,
	\end{equation}
then the corresponding Brownian motions $B(x)$ and $B(y)$ are related by 
    \begin{align*}
        \mathbb{E} B_{\tau}(x) \otimes B_{\tau}(y) = \max \{ \tau - \ln r, 0 \} \mathbb{E} B_{1}(0) \otimes B_{1}(0) .
    \end{align*}
In view of the Stratonovich SDE $dF=F\circ dB$, this extends to the corresponding geometric 
Brownian motions.

\medskip

Put slightly differently, the Brownian motions $B(x)$ and $B(y)$ are independent until time $\tau = \ln \sqrt{ 1 + d(x,y)^{2} }$ after which their increments are identical.  Here $d(x,y) = r^{2} - 1$ defines a metric on the space of blow-ups.  In fact, from the definition \eqref{def_of_r}, $d$ is an ultrametric, meaning that
$d(x,z)\le\max\{d(x,y),d(y,z)\}$.  This qualitative picture is a tensorial version of the one identified by Dunlap and Gu \cite{DunlapGu} in the context of the stochastic heat equation.

\subsection{Intermittent behavior of $F$}\label{ss:intermit} Theorem \ref{thm:mr01} asserts that the geometric Brownian motion $F$ determined by \eqref{ao01bis} approximates the increments of the expected position $u$.  Thus, properties of $F$ can be lifted to $u$.  Exploiting features of the SDE, we are able to say quite a bit about the intermittent character of $F$. Presently, we carry this out in $ n = 2 $, which is more explicit than the case $ n > 2 $ that will be addressed in a subsequent work.

\medskip

At the level of $F$, non-Gaussian long-time behavior is not surprising: The process $ F $ is the higher-dimensional analogue of a stochastic exponential/geometric Brownian motion in $ n = 1 $. Thus, it is natural to expect that the tensorial version \eqref{ao01bis} in $ n \geq 2 $ shows intermittent behavior.

\begin{lemma} \label{R: intermittent growth of F} In dimension $ n = 2 $, for any $p \geq 1$, we have
	\begin{equation*}
		\mathbb{E} |F_{\tau}|^{2p} \sim_{p} e^{ \frac{1}{2} p (p + 1) \tau } \quad \text{for all} \, \, \tau \geq 0.
	\end{equation*}
\end{lemma}

In a subsequent work we will prove the following extension.

\begin{lemma}[Upcoming work with Şefika Kuzgun \cite{KMOW25}]\label{up01}In dimension $ n \geq 2 $, and any integer $ p \geq 1 $, we have
	\begin{equation*}
		\mathbb{E} |F_{\tau}|^{2p} \sim_{p} e^{ \frac{1}{2} p ( \frac{ 4 }{ n + 2 } p + \frac{ 2n }{ n + 2 }  ) \tau } \quad \text{for all} \, \, \tau \geq 0.
	\end{equation*} 
\end{lemma}

\medskip

	Lemma \ref{R: intermittent growth of F} \& \ref{up01} shows that $ | F_{\tau} |^2 $ is qualitatively similar to a stochastic exponential: In terms of scaling, its moments behave as if $ \ln | F_{\tau} |^2 $ were a Gaussian with mean $ \frac{n}{n+2} \tau $ and variance $ { \textstyle \frac{ 4 }{ n + 2 } } \tau $. In dimension $n = 2$, we make this statement precise in the proof of Lemma \ref{R: intermittent growth of F}, cf.~\eqref{E: Q formula} in Section \ref{S: intermittency}, by showing that the auxiliary process $ R_{ \tau } := \frac{ 1 }{ n } | F_{ \tau } |^2 $ solves an SDE in its own right that is reminiscent of a geometric Brownian motion.

\medskip

By comparing $R$ to one-dimensional geometric Brownian motion, we prove the following result, which will be used to compare the expected position $u$ to the geometric Brownian motion $F$.

\begin{lemma} \label{lem-weakinter} In dimension $ n = 2 $, for any $\tau \geq 0$,
\begin{align}\label{ao79}
\mathbb{E} | F_{ \tau } |^2 I( | F_{ \tau } |^2 \ge \frac{1}{2 \sqrt{2} }\mathbb{E}^\frac{3}{2} | F_{ \tau } |^2 )
\geq\frac{1}{4}\mathbb{E}  | F_{ \tau } |^2 .
\end{align}
\end{lemma}

In a subsequent work we will extend the lemma to\footnote{Clearly the constants in Lemma \ref{lem-weakinter} \& \ref{up02} do not match in two dimensions. This is a consequence of the different proof strategy in \cite{KMOW25}. We expect neither of them to be optimal. }
\begin{lemma}[Upcoming work with Şefika Kuzgun \cite{KMOW25}]\label{up02} In dimension $ n \geq 2 $
\begin{align*} 
\mathbb{E} | F_{ \tau } |^2 I( | F_{ \tau } |^2 \ge \frac{ 1 }{ e }  \, \mathbb{E}^{ \frac{ n + 4 }{ n + 2 } } | F_{ \tau } |^2 )
\geq \frac{ 1 }{ 4 } \mathbb{E}  | F_{ \tau } |^2 
\quad { \rm for } ~ \tau \gg_{ n } 1 .
\end{align*}

\end{lemma}

\medskip

Using Lemmas \ref{lem-weakinter} \& \ref{up02} (the latter of which we do not prove within this work) in conjunction with the approximation of $u$ by $F$, we obtain
\begin{equation*}
{ \rm a ~ version ~ of ~ \eqref{ao79} ~ with } \quad | F_{ \tau }  |^2 \quad \text{replaced by} \quad \int_{ \mathbb{S}^{ n - 1 }  } d \omega \fint_0^Tdt\frac{1}{|x|^2}|u(| x | \omega ,t)-u(0,t)|^2 ,
\end{equation*}
see \eqref{meet01} for the precise statement.
Using this, we deduce the following result about the intermittent behavior of $u$ at the level of increments.


\begin{theorem}\label{thm:mr02}
In dimension $ n \geq 2 $, in the regime
\begin{align}  \label{vregime}
\varepsilon^2 \lambda(|x|^2)\ll 1,
\end{align}
we have, for any $(x,T)$ and every $1 \leq p < \infty$,
\begin{align}
\mathbb{E}\fint_0^Tdt\frac{1}{|x|^2}|u(x,t)-u(0,t)|^2 &\approx
\max\Big\{1,\frac{\lambda(T)}{\lambda(|x|^2)}\Big\} \label{ao06} \\
\mathbb{E} \Big(\fint_0^Tdt\frac{1}{|x|^2}|u(x,t)-u(0,t)|^2 \Big)^p
&\gtrsim_{ p } \max\Big\{1,\frac{\lambda(T)}{\lambda(|x|^2)}\Big\}^{ 1 + \frac{ n + 4 }{ n + 2 } ( p - 1 ) }. \label{ao77}
\end{align}
\end{theorem}

We expect that the exponent in \eqref{ao77} can be improved to match the one in Lemma \ref{R: intermittent growth of F}. Currently, we are only able to transmit the linearized exponent $1 + \frac{ n + 4 }{ n + 2 } ( p - 1 )$ (it is precisely the linearization of the quadratic $\frac{1}{2} p ( \frac{ 4 }{ n + 2 } p + \frac{ 2n }{ n + 2 }  )$ at $p = 1$) since the homogenization results in Section \ref{sec:hom} are limited to $ p = 1 $.

\medskip

Note that in the regime $ 1 \ll \lambda ( | x |^2 ) \ll \lambda ( T ) $ (which implies
$ 1 \ll | x |^2 \ll T $), the r.~h.~s.~of \eqref{ao06} is $ \approx \sqrt{ \frac{ \ln T }{ \ln | x |^2 } } $.
Hence the transformation $ x \mapsto u( t, x ) $ is borderline non-Lipschitz in a
wide range of $ | x | $. This is in line with the discussion in Subsection \ref{ss:X}.

\medskip

Let us now interpret Theorem \ref{thm:mr02} in terms of the process $ X $ in the geometrically intuitive case of dimension $n = 2$:
Since $b$ is divergence-free, we can write $ b $ as the orthogonal
gradient of a stream function $ \psi $, i.~e.~$ b = \nabla^{ \perp } \psi $. Due to
the white-noise character of $ b $, $ \psi $ is a Gaussian free field (GFF), and thus since $ n = 2 $ not stationary
(even in the presence of the UV cut-off).

\medskip

As was already present in \cite{FisherEtAl}, the superdiffusive behavior
\eqref{E: logarithmic correction} can be explained as follows:
Since a realization of the GFF $\psi$ with UV cut-off
has closed stream lines, the process \eqref{ao24} would be trapped in a finite region
if it were not for diffusion. (Incidentally, this
collection of loops is an object of active research
in view of its emergent conformal invariance.)
Diffusion allows the process $X$ to move from one level set to the next.
As $X$ explores larger and larger neighborhoods of its starting point,
it occasionally hits loops of large diameter, which introduces a ballistic element.
In the absence of an infra-red (IR) cut-off, there are loops of larger and larger diameter,
which is at the origin of superdiffusivity.

\medskip

Our Theorem \ref{thm:mr02} describes an intermittent behavior of the process:
Two particles that start at close points and whose motion is driven by the same Brownian
motion start exploring nearby level lines of the GFF.  In rare cases (i.e., with high probability, a small volume fraction of possible starting points in a large ball), the two particles lie in the basin attraction of two long level lines, which rapidly diverge, resulting in the particles being ripped apart with a burst-like character.  This manifests in the intermittent behavior of the increment of the expected position $u$, which has abnormally long, non-Gaussian tails due to these rare events.  Since trajectories only diverge on very large time scales (large enough that the diffusivity is large compared to the distance between particles), this behavior is in contrast (but not in contradiction) with the quenched invariance
principle for $X$ reported in \cite{ABK24}. Theorem \ref{thm:mr02}
seems to capture more of
the fine structure of the level sets of the GFF.

\medskip

Our results suggest that the intermittency exhibited by the drift-diffusion process is more a manifestation of geometry than topology.  On the one hand, at the level of topology, the stream lines of $b$ are closed if and only if $n = 2$, yet we show that the process exhibits intermittency regardless.  On the other hand, Lemma \ref{up01} suggests that intermittency plays less and less of a role as the dimension $n$ grows, which seems to be related to the fact that $\textbf{SL}(n)$ becomes less and less curved.

\medskip

The geometric Brownian motion $F$ provides a simplified description of the gradient $\nabla u$ of the Lagrangian coordinate, in which the spatial noise is replaced by a temporal one.  A conceptually similar reduction occurs in another scaling-critical system, namely, the 2d stochastic heat equation.  Specifically, in the weak coupling regime, Dunlap and Gu \cite{DunlapGu} showed that a certain forward-backward SDE describes the behavior of the solution on logarithmic time scales.

\subsection{Strategy of Proof}\label{S:strategy_of_proof} Recall our interest in analyzing the expected particle coordinates $ u $, which we think of as Lagrangian coordinates. This involves three approximations:

\medskip

As is explained in greater detail in \cite{CMOW}, the first approximation is motivated by a work of \cite{Fannjiang98} and is summarized next.

\begin{approximation}\label{approx01}
We shall replace the evolution of the time variable $ T $ by the evolution of a scale-variable $ L $. Introducing an artificial large-scale (infra-red) cut-off  on scale $ L $ ensemble $ b_L $ with associated generator $ \mathcal{L}_L $, we will argue below that we may approximate for $ t \sim L^2 $
\begin{align*}
	u ( \cdot , t ) \approx u_L ,
	\quad \text{where} ~ u_L ~ \text{solves} \quad
	\mathcal{L}_L u_L = 0 \quad \text{with} \quad \mathbb{E} \nabla u_{L} = \text{id}.
\end{align*}
\end{approximation}

Approximation \ref{approx01} has to be taken with a grain of salt. While it is instructive as a guiding principle for our approach, it is justified a posteriori, i.~e.~after having introduced the next approximation, see \eqref{E: CMOW bound} for the precise statement.

\medskip

Thinking of $ u_L $ as ($\mathcal{L}_L$-)harmonic coordinates, it is reasonable to view them as perturbations of the Euclidean coordinates $ x \mapsto x $ via the so-called corrector $ \phi_L $, namely
\begin{align*}
	u_L ( x ) = x + \phi_L ( x ) .
\end{align*}
Thus, in our geometric interpretation $ \phi_L $ is naturally seen as taking values in tangent space. Hence its components $ \phi_L^i $ will have upper indices in line with $ \phi_L = \phi_L^i e_i $.

\begin{approximation}\label{approx02}
We appeal to the scale-by-scale homogenization technique that was introduced in \cite{CMOW}: We introduce an approximation (for short proxy) $ \tilde\phi_L $ of $ \phi_L $ via an SDE that implements an incremental two-scale expansion, namely
\begin{align*}
	\nabla u_L \approx { \rm id } + \nabla \tilde\phi_L ,
	\quad \text{where} ~ \tilde\phi_L ~ \text{solves} \quad
	d \tilde\phi = ( 1 + \tilde\phi^i_L \partial_i ) \circ \! d \phi ,
	~ \tilde\phi_{ L = 1 } = 0 .
\end{align*}
The crucial ingredient here is the driver $ \nabla d \phi $ is obtained from the increments of $ \nabla \phi_L $ by applying two approximations: homogenization and linearization.
\end{approximation}

We will elaborate on Approximation \ref{approx02} in much greater detail in Section \ref{sec:hom} below.

\medskip

The next step is a reinterpretation of the previous approximation in probabilistic terms.

\begin{approximation}\label{approx03}
We introduce the geometric Brownian motion $ F $ that approximates $ { \rm id } + \nabla \tilde\phi_L  $, i.~e.
\begin{align} \label{E: coupling first instance}
	{ \rm id } + \nabla \tilde\phi_L \approx F_L,
	\quad \text{where} ~
	d F = F_L \circ \! \nabla d \phi , ~ F_{ L = 1 } = { \rm id } ,
\end{align}
where $ \nabla d \phi $ will be identified as a Brownian motion on $ \mathfrak{sl}(n) $.
\end{approximation}

\section{Scale-by-scale homogenization}\label{sec:hom}

We start by exploiting a connection with divergence-form elliptic equations.  Since $ \nabla . b = 0 $, the Poincar\'{e} Lemma implies there exists a stream matrix $ \Psi $, a skew-symmetric element of $ ( { \rm tangent ~ space } ) \otimes ( { \rm tangent ~ space } ) $, such that\footnote{Recall our convention $ \nabla . \Psi = \partial_j \Psi^{ i j } \, e_i $, which is consistent with $ \Psi \nabla u = \Psi^{ i j } \partial_i u \, e_j $ (recall that $ \Psi^{ i j } e_i \otimes e_j $ denotes the canonical isomorphism between cotangent space and tangent space induced by the metric $ \Psi $; thus $ \Psi \nabla u $ denotes the Riemannian $ \Psi $-gradient of $ u $) and leads to $ \nabla . \Psi \nabla u = \Psi . \nabla^2 u + ( \nabla . \Psi ) . \nabla u $.} $ \nabla . \Psi = b $.  Thus we can rewrite the generator of $ X $ via\footnote{We use the convention that the divergence $ \nabla . $ acts on all expressions to the right so that $ \nabla . a \nabla u = \nabla . ( a \nabla u ) $.}
\begin{align*}
	\mathcal{L} u
	= \nabla . a \nabla u ,
	\quad { \rm where } \quad
	a := I  + \Psi , ~ I := \delta^{ i j } \, e_i \otimes e_j  .
\end{align*}
Note that $I$ is the canonical isomorphism between cotangent space and tangent space induced by the Euclidean metric. If $ \Psi $ were stationary, $\mathcal{L}$ would fall into the standard homogenization framework of second-order elliptic operators in divergence form; as a consequence $ X $ would exhibit diffusive behavior.\footnote{Note that the skew-symmetric part $\Psi$ does not affect the ellipticity.}  As a matter of fact, for the critical ensemble \eqref{ao08}, the stream matrix $ \Psi $ is not stationary, which is at the heart of superdiffusive behavior. However, since $ \Psi $ is only borderline non-stationary, which will be made precise in \eqref{estPsi} below, it is reasonable to introduce an artificial small-scale/ infra-red (IR) cut-off on scale $ L $ and study the behavior as $L \to \infty$.

\subsection{Implementation of the infra-red cut-off}\label{ss:IR}The IR cut-off is conveniently implemented on the level of the Fourier transform: We define $ b_L $ as a Schwartz distribution\footnote{Its value on a smooth and compactly
supported 
$(0,\infty)\times\mathbb{R}^n\ni (L,x)\mapsto \zeta (L,x)\in\mathbb{R}^n $ is given by
$b$ applied to the inverse Fourier transform of
$\mathbb{R}^n\ni k\mapsto \int_{|k|^{-1}}^\infty dL{\mathcal F} \zeta (L,k)$, which
is a Schwartz function.} in the variables $ ( L , x ) \in ( 0, \infty ) \times \mathbb{R}^d $ such that
\begin{align}\label{cov04}
{\mathcal F}b_L=I(L^{-1} < |k|){\mathcal F}b
\quad { \rm in ~ the ~ sense ~ of ~ Schwartz ~ distributions }.
\end{align}
In view of the UV cut-off $ b $ it supported on $\{L\ge 1\}$. There are two advantages and one disadvantage
of this sharp cut-off in Fourier space.

\medskip

The first advantage is pathwise and concerns 
the distributional derivative\footnote{Here and in the following $ d $ is a differential (now distributional, later also in the Itô-sense) w.r.t.~the variable $ L $, while $ \partial_{ \cdot } $ and $ \nabla $ are (distributional derivatives) acting on the spatial variables.}  $db$ w.~r.~t.~$L$: It follows 
immediately from the
definition\footnote{Indeed, testing \eqref{cov04} against $ d \zeta $ for a test function $ \zeta = \zeta ( L , k ) $ we obtain
\begin{align*}
	\mathcal{F} db . \zeta
	= - \mathcal{F} b . d \zeta 
	\stackrel{(\ref{cov04})}{=} - \int_0^{ \infty } dL \, I ( L^{-1} < | k | ) \mathcal{F} b . d \zeta ( L , k )
	= - \mathcal{F} b . \zeta ( | k |^{-1} , k )  .
\end{align*}
} 
that
\begin{align}\label{forFelix07}
	{\mathcal F} db ~ { \rm is ~ supported ~ on } ~ \{ L | k | =1 \} .
\end{align}
and thus ${\mathcal F}b_L$ is supported in the closure of $\{L^{-1} <  |k|\le 1\}$.

\medskip

For later purpose we retain
that this implies
\begin{align}\label{fw01}
-\Delta db=L^{-2} db.
\end{align}
The second advantage is the independence of (finite) increments:
\begin{align}\label{cov02}
b_{L_+}-b_L\;\mbox{is independent~of}\;b_L.
\end{align}
Indeed, by the Gaussianity inherited from $b$ we just have to show that 
for any pair of Schwartz vector fields $\xi$ and $\xi'$, the corresponding covariances 
vanishes. By definition (\ref{cov04}) the latter is given by the covariance 
between $b$ applied to ${\mathcal F}^{-1}I(L_+^{-1} < |k| \leq L^{-1}){\mathcal F}\xi$
and $b$ applied to ${\mathcal F}^{-1}I(L^{-1} < |k|){\mathcal F}\xi'$.
As discussed in Subsection \ref{ss:X},
the stationarity of $b$ implies that latter covariance can be written
as the integral operator $( 2 \pi )^{ \frac{n}{2} }   \int dk \, {\mathcal F}c(k)$ applied to the product of 
the Fourier transforms, namely $I(L_+^{-1} < |k| \leq L^{-1}){\mathcal F}\xi$
and $I(L^{-1} < |k|){\mathcal F}\xi'$, which by construction vanishes.

\medskip

The only disadvantage of the sharp cut-off is the limited regularity of $b_L$ in $L$
despite its smoothness in $x$: It is local H\"older continuous of exponent $\alpha$ iff $\alpha<\frac{1}{2}$.\footnote{Indeed, this is easiest seen from the discussion in the next paragraph. From \eqref{cov02} \& \eqref{cov02fo} we learn that the annealed $ L^2 $ modulus of continuity is given by
\begin{align*}
	\E | b_{ L_{+} } - b_{ L } |^2 = \varepsilon^2 \frac{n}{4} \big( \frac{1}{L^2} - \frac{1}{L_{+}^2} \big) \approx  \varepsilon^2 \frac{n}{2 L^3} ( L_{ + } - L )
	\quad { \rm as } ~ L_{ + } \searrow L .
\end{align*}
In view of the Kolomogorov regularity criterion for Gaussian processes, this implies the claimed quenched regularity.}

\medskip

Letting $ L_+ \rightarrow \infty $ in \eqref{cov02}, we find that $ \E | b_L |^2 = \E | b |^2 - \E | b - b_{ L } |^2 $. Thus in order to determine $ \E | b_L |^2 $, it is enough to compute $ \E | b - b_{ L } |^2 $: Notice that in law, $ b - b_L $ may be identified with the ensemble of $ b $'s having UV cut-off at scale $ L $; its covariance is obtained by multiplying the covariance of $ b $ by $  I ( | k | \leq L^{ - 1 } )  $. Therefore, since the law of $ b - b_L $ is determined by its covariance, a scaling argument using the homogeneity assumption shows that $ b - b_L =_{ \rm law } \frac{ 1 }{ L } b ( \frac{ 1 }{ L } \cdot ) $. Thus we learn that
\begin{align}\label{cov02fo}
	\E | b_L |^2
	= \Big( 1 - \frac{ 1 }{ L^2 } \Big) \E | b |^2
	\stackrel{ (\ref{cov01fo}) }{ = }  \varepsilon^2  \frac{ n }{ 4 } \Big( 1 - \frac{ 1 }{ L^2 } \Big) .
\end{align}

\subsection{The stream tensor $ \Psi $} Clearly there is a choice of gauge to be made since the equation 
	\begin{equation}\label{psi3}
		 \nabla . \Psi_L = b_L
	\end{equation}
with $ \Psi_L $ skew does not determine $ \Psi_L $. We will pick the well-established choice\footnote{Note that in Fourier space, our convention $ \nabla . \Psi = \partial_j \Psi^{ i j } \, e_i $ amounts to $ \mathcal{F} \nabla . \Psi_L = \iota k_j \Psi_L^{ i j } e_i $, while the identity $ \nabla . b_{L} = 0 $ takes the form $k. \mathcal{F} b_{L} = 0$, hence the definition \eqref{gaugePsi} yields $ \mathcal{F} \nabla . \Psi_L = \mathcal{F} b_L $ and thus $ \nabla . \Psi_L = b_L $.}
\begin{align}\label{gaugePsi}
\mathcal{F} \Psi_L = \frac{ \iota }{ | k |^2 }  ( k^{ *  } \otimes \mathcal{F} b_L -  \mathcal{F} b_L \otimes k^* ) ,
\end{align}
see e.~g.~\cite[Section 11]{KLO}. In view of the discussion around \eqref{cov04} makes sense in a pathwise way: Since $ \mathcal{F} b_L $ is supported on the annulus $ \{ L^{-1} < | k | \leq 1 \} $ the products with $ k^{ *  } $ on the r.~h.~s.~of \eqref{gaugePsi} are well-defined Schwartz distributions. For $ n = 3 $ this corresponds to the Coulomb gauge. From \eqref{gaugePsi} and the homogeneity of \eqref{ao02d} we learn that the Fourier transform of the covariance tensor of $ \Psi $ is $ -n $-homogeneous on $ \{ L^{-1} < | k | \leq 1 \} $. Thus the normalization \eqref{cov02fo} implies
\begin{align}\label{estPsi}
	\E | \Psi_L |^2 \sim \varepsilon^2 \ln L .
\end{align}
The logarithmic divergence as $ L \rightarrow \infty $ shows the borderline non-stationarity of the (formal) limit. As mentioned already at the beginning of this section, this divergence is at the heart of superdiffusive behavior: That divergence-free drifts with stationary, $L_{\E}^{2}$ stream matrices (the so-called ``$H^{-1}$ condition") lead to diffusive asymptotics was originally discovered by Oelschl\"{a}ger \cite{Oelschlaeger88}.  For more on the diffusive setting, the reader can see the textbook of Komorowski, Olla, and Landim \cite{KLO} or the more recent papers of Kozma and T\'{o}th \cite{KozmaToth17}, Fehrman \cite{Fehrman23}, and T\'{o}th \cite{Toth18} for the state of the art.

\medskip

In the two-dimensional case the space of skew tensors $ ( { \rm tangent ~ space } ) \otimes ( { \rm tangent ~ space } ) $ is one dimensional; thus $ \Psi_{L} $ may be identified with a scalar (the stream function) by writing $ \Psi_{L} = \psi_{L} J $, where $ J = e_1 \otimes e_2 - e_2 \otimes e_1 $. Hence $ b_{L} = \nabla . \Psi_{L} = \partial_2 \psi_{L} \, e_1 - \partial_1 \psi_{L} \, e_2 = - \nabla^{ \perp } \psi_{L} $ as in \cite{CMOW}.  As already discussed after Theorem \ref{thm:mr02}, see Section \ref{ss:intermit}, $\psi_{L}$ is the GFF with IR and UV cut-offs.

\subsection{A primer on homogenization theory} \label{S: primer} Recall from Subsection \ref{S:strategy_of_proof} that our analysis begins with the harmonic coordinates $ u_L $ solving $ \mathcal{L}_L u_L = 0 $ with $\mathbb{E} \nabla u_{L} = \text{id}$. Thinking of $ u_L $ as a perturbation of the Euclidean coordinates, we write $ u_L ( x ) = x + \phi_L ( x ) $; the function $ \phi_L $ is the corrector taking values in tangent space. As in the discussion at the very beginning of this section, the equation for $ u_L $ may be recast into
\begin{align} \label{defn-a}
	\nabla . a_L(e^i+\nabla\phi_L^i) = 0, \quad \text{where} \quad a_{L} = I + \Psi_{L}, ~ \mathbb{E} \nabla \phi_{L}^{i} = 0,
\end{align}
or its more expanded form, the Helmholtz-type decomposition
\begin{align}\label{tao01}
	a_L ( e^i + \nabla\phi_L^i ) = \bar a_L e^i + \nabla.\sigma_L^i  
	\quad { \rm with } \quad \bar{a}_{L} ~ { \rm constant,} ~
	\sigma_L^i ~ { \rm skew } , ~ \mathbb{E} \nabla\sigma_{L}^{i} = 0.
\end{align}
The first equation can be recovered by applying the divergence $ \nabla . $ to the latter.  We do not discuss existence and uniqueness of the solution of \eqref{defn-a} in this paper since we never need it; however, the interested reader can see the discussion in \cite[Section 3]{CMOW}.

\medskip

The Helmholtz decomposition \eqref{tao01} gives rise to an intertwining between $ \mathcal{L}_L $ and a constant-coefficient operator: 
\begin{align}\label{intertwin}
	\nabla . a_L \nabla ( 1 + \phi_L^i \partial_i ) \bar u 
	= \nabla . \bar a_L \nabla  \bar u 
	+ \nabla . ( \phi_L^i a_L + \sigma_L^i ) \nabla \partial_i \bar u .
\end{align}
Indeed, expanding the term under the divergence by using Leibniz' rule and invoking the Helmholtz decomposition \eqref{tao01} yields
\begin{align*}
	a_L \nabla ( 1 + \phi_L^i \partial_i ) \bar u 
	= \partial_ i \bar u \, a_L ( e^i + \nabla \phi_L^i )  + \phi_L^i a_{L} \nabla \partial_i \bar u 
	\stackrel{ \eqref{tao01} }{ = } \bar a_L \nabla \bar u + \partial_i \bar u \, \nabla . \sigma_L^i  +  \phi_L^i a_L \nabla \partial_i \bar u.
\end{align*}
Using Leibniz' rule once more, this time in form of%
\footnote{\label{leibniz}The following means in terms of the $i$-th component
$\partial_j \zeta \sigma^{ij}$ $=\zeta \partial_j\sigma^{ij}+\sigma^{ij}\partial_j \zeta$, where the latter summand is equal to the contraction $ \sigma^{ * } \nabla \zeta = -  \sigma \nabla \zeta $.}
$\nabla.  \zeta \sigma = \zeta\nabla.\sigma - \sigma\nabla \zeta$
for scalar $ \zeta $ and skew $\sigma$ for $( \zeta, \sigma) = ( \tilde\phi_L^i, \sigma_L^i )$, this identity becomes
\begin{align*}
	a_L \nabla ( 1 + \phi_L^i \partial_i ) \bar u 
	= \bar a_L \nabla \bar u + \nabla . \partial_i \bar u \, \sigma_L^i  +  ( \phi_L^i a_L + \sigma_L^i ) \nabla \partial_i \bar u.
\end{align*}
Applying the divergence, and noting that $ \nabla .  \partial_i \bar u \, \sigma_L^i  $ is divergence-free since $ \partial_i \bar u \, \sigma_L^i  $ is skew, we obtain \eqref{intertwin}.

\medskip

The intertwining \eqref{intertwin} is based on what is known as the two-scale expansion, i.~e.~$ ( 1 + \phi_L^i \partial_i ) \bar u  $, and the merit of the skew-symmetric field $ \sigma_L $ is that it brings the error term in divergence form, which is convenient for energy estimates. Based on the ideas of \cite{CMOW}, we will work with proxies $ ( \tilde a_L, \tilde\phi_L^i , \tilde\sigma_L^i ) $ of $ ( \bar a_L , \phi_L^i , \sigma_L^i ) $ that implements the idea of an incremental, or~scale-by-scale, homogenization procedure, to the effect that \eqref{tao01} is only true up to a residuum.

\subsection{Heuristic argument for Approximation \ref{approx02}: scale-by-scale homogenization} \label{S: heuristics}

We will construct the proxies $ ( \tilde a_L , \tilde\phi_L^i , \tilde\sigma_L^i ) $ incrementally, and thus apply the derivative w.~r.~t.~$L$ to \eqref{tao01}. In view of the independence of increments \eqref{cov02} and the logarithmic behavior \eqref{estPsi}, $ da = d \Psi $ has a white-noise character in $ \ln L $. Thus the product rule within the framework of Itô calculus gives rise to the (infinitesimal) quadratic variation $d[a ( e^i + \nabla\phi^i ) ] = d[a\nabla\phi^i]$, so that we obtain
\begin{align*}
	d a \, ( e^i + \nabla \phi_L^i )
	+ a_L \nabla d \phi^i 
	+ d [ a \nabla \phi^i ] = d \bar a e^i + \nabla . d \sigma_L^i .
\end{align*}
On this (formal) identity, we will perform two approximations:
\begin{itemize}
	\item[\tiny $ \bullet $] homogenization, i.~e.~replacing $a_L\nabla d\phi^i$ by
$\bar a_L \nabla d\phi^i$, where $ \bar a_L $ is constant,
	\item[\tiny $ \bullet $] linearization, i.~e.~neglecting $da\nabla\phi_L^i$, which\footnote{As noted in \cite{CMOW}.} is formally a quadratic term in $ \varepsilon $.
\end{itemize}
Consequently, we arrive at the approximation 
\begin{align*}
	d a e^i
	+ \bar a_L \nabla d \phi^i 
	+ d [ a \nabla \phi^i ]
	\approx d \bar a e^i + \nabla . d \sigma_L^i .
\end{align*}
Our strategy is predicated on using this approximation as the definition of a first approximation $ ( \tilde a_L , d \phi^i , d \sigma^i ) $, which is not yet the proxy $ (\tilde\phi_L^i , \tilde\sigma_L^i ) $ announced earlier. For simplicity, we will refrain from introducing a new variable for the approximation $ ( d \phi^i , d \sigma^i ) $ as (infinitesimal) increments of the corrector will not be relevant in the subsequent steps. That said, we \emph{define} $ ( \tilde a_L, d \phi^i , d \sigma^i ) $ via 
\begin{align*}
da e^i + \tilde a_L\nabla d\phi^i+d[a \nabla\phi^i]= d \tilde a e^i+\nabla.d\sigma^i ,
\end{align*}
which, upon splitting into martingale and finite-variation parts becomes
\begin{align}\label{tao17}
dae^i+\tilde a_L \nabla d\phi^i=\nabla.d\sigma^i 
\quad { \rm with } \quad d \sigma^i ~ { \rm skew },
\quad\mbox{and}\quad
d\tilde a e^i=d[a \nabla\phi^i ].
\end{align}
Since there is no drift when $L = 1$, we know that $\bar{a}_{L = 1} = a_{L = 1} = I$, and thus \eqref{tao01} implies that $\nabla \phi^{i}_{L = 1} = 0$ and $\nabla.\sigma^{i}_{L = 1} = 0$.  This motivates us to supplement \eqref{tao17} with the initial conditions
\begin{equation} \label{E: initial conditions proxies}
	\phi^{i}_{L = 1} = 0, \quad \sigma^{i}_{L = 1} = 0, \quad \tilde{a}_{L = 1} = I.
\end{equation}

\medskip

The equations above uniquely determine\footnote{Here we abuse notation: Henceforth $\phi^{i}_{L}$ is not the actual corrector, i.e., the solution of \eqref{tao01}, but rather the integral of the increment $d\phi^{i}$ from \eqref{tao17}.} $\phi^{i}_{L}$ and $\tilde{a}_{L}$ in the classes of centered, stationary Gaussians with independent increments and constant-in-space, deterministic matrices, respectively, in a manner that is independent of the choice of gauge of $\Psi_{L}$. Since $ a $ and $ \phi $ have independent increments, see~\eqref{gaugePsi}, we obtain for their covariation $ [ \phi^i \, \nabla . a ]_{L} = \E \phi_{L}^i \nabla . a_{L} $ and $[a \nabla \phi^{i}]_{L} = \mathbb{E} a_{L} \nabla \phi^{i}_{L}$. Since by stationarity, we can integrate by parts to obtain $\mathbb{E}\phi_{L}^i \nabla . a_{L} = \mathbb{E} \phi_{L}^{i} \nabla.\Psi_{L} = \mathbb{E} \Psi_{L} \nabla \phi^{i}_{L} $ $= \mathbb{E} a_{L} \nabla \phi^{i}_{L}$, the second item of \eqref{tao17} becomes
\begin{align}\label{tao17y}
d\tilde a e^i = d[ \phi^i \, \nabla . a ]
\stackrel{ (\ref{defn-a}) }{ = }  d[ \phi^i \, b ] .
\end{align}
Moreover, upon postulating that $\tilde{a}_{L} = \tilde{\lambda}_{L} I$ for some deterministic constant $\tilde{\lambda}_{L}$ and applying the divergence to the first item of \eqref{tao17}, we obtain the equation
\begin{align}\label{tao17x}
- \tilde{\lambda}_{L} \Delta d \phi^{i} = - \nabla . \tilde a_L \nabla d\phi^i 
\stackrel{ \eqref{tao17} }{ = } \nabla . dae^i
\stackrel{ (\ref{defn-a}), (\ref{psi3}) }{ = } d b^i,
\end{align}
which by \eqref{fw01} has the solution
\begin{align}\label{E: eqn for dphi}
	d \phi^i = \tilde{\lambda}_{L}^{-1} L^{2} d b^{i} ,
	\quad \phi^i_{ L = 1 } = 0 .
\end{align}
in the class of centered, stationary Gaussians. Both \eqref{tao17x} and \eqref{tao17y}, resp.~\eqref{E: eqn for dphi}, solely depend on $ b_{L}$, and are thus independent of the gauge of $ \Psi_{L}$.

\begin{lemma} \label{L: construction of proxies} There is a unique centered, stationary Gaussian field $\phi^{i}_{L}$ with independent increments and a unique constant-in-space, deterministic isotropic matrix $\tilde{a}_{L} = \tilde{\lambda}_{L} I$ such that \eqref{tao17x}, \eqref{tao17y}, and \eqref{E: initial conditions proxies} hold.  The latter is determined by the initial value problem
	\begin{align}\label{ODElambda}
		d \tilde\lambda_L = \varepsilon^2 \frac{ d \ln L }{ 2 \tilde\lambda_L },
		~ \tilde\lambda_{ L = 1 } = 1,
	\end{align}
with the explicit solution
	\begin{align}\label{lambda}
		\tilde\lambda_L^2 = 1 + \varepsilon^2 \ln L .
	\end{align}
Further, $\phi^{i}_{L}$ and $\tilde{a}_{L}$ do not depend on the choice of gauge for $\Psi_{L}$.
\end{lemma}

\begin{proof}
To determine $\tilde{\lambda}$, we contract \eqref{tao17y} along $ \delta_{ i j } e^j $, which amounts to taking $ { \rm tr } \, \tilde a = n \, \tilde\lambda $, from which we learn that
\begin{align*}
	n \, d \tilde\lambda
	= \delta_{ i j } \, d [ \phi^i b^j ]  .
\end{align*}
After substituting \eqref{E: eqn for dphi}, this identity becomes
\begin{align*}
	n \, d \tilde\lambda = \delta_{ i j } \frac{ L^2 }{ \tilde\lambda_L } \, d [ b^i b^j ]  
	=  \frac{ L^2 }{ \tilde\lambda_L } \, d [ b \cdot b ] .
\end{align*}
Since $ b $ is a continuous martingale with independent increments we know that $ [ b \cdot b ]_{ L } = \E | b_L |^2  $. Further, since $ b $ is a stationary field, $  [ b \cdot b ]_L $ is constant in space. Taken together with \eqref{cov02fo}, this implies 
\begin{align}\label{forFelix01}
	d [ b \cdot b ] \stackrel{ (\ref{cov02fo}) }{ = } \varepsilon^2 \frac{ n }{ 2 } \frac{ d L }{ L^3 } = \varepsilon^2 \frac{ n }{ 2 } \frac{ d \ln L }{ L^2 } ,
\end{align}
and thus \eqref{ODElambda}.  Integration of this ODE yields the formula \eqref{lambda}. 
\end{proof}

Henceforth we take $\tilde{a}_{L}$ and $\phi^{i}_{L}$ to be the fields determined by \eqref{tao17x}, \eqref{tao17y}, and \eqref{E: initial conditions proxies}.  We define the skew-symmetric matrix field $\sigma^{i}_{L}$ using the explicit choice of the Coulomb gauge so that \eqref{tao17} holds.  As in the case of $\Psi_{L}$, we define $\sigma_{L}^{i}$ via the Coulomb gauge, given in Fourier space by the formula\footnote{This choice is consistent with \eqref{tao17}: Since $ \tilde a_L $ is homogeneous (in space) and isotropic we have $ k^* \otimes \mathcal{F} ( \tilde a_L \nabla d \phi^i ) - \mathcal{F} ( \tilde a_L \nabla d \phi^i ) \otimes k^* = 0 $. Therefore, by taking the derivative of \eqref{tao17gauge} and recalling that $ d a = d \Psi $, we recover
\begin{align*}
	\mathcal{F} d \sigma_{L}^i = \frac{ \iota }{ | k |^2 } \big( k^* \otimes \mathcal{F} ( d a e^i + \tilde a_L \nabla d \phi^i ) - \mathcal{F} ( d a e^i + \tilde a_L \nabla d \phi^i ) \otimes k^* \big).
\end{align*}
As in \eqref{gaugePsi}, the above equation and \eqref{tao17x} imply \eqref{tao17}.
}
	\begin{align} \label{tao17gauge}
	\mathcal{F} \sigma_{L}^i = \frac{ \iota }{ | k |^2 } \big( k^* \otimes \mathcal{F}  \Psi_{L} e^i  - (\mathcal{F}  \Psi_{L} e^i ) \otimes k^* \big).
	\end{align}
This defines a centered, stationary Gaussian field with independent increments. 

\medskip

Our heuristics also covers the case of a (bare) diffusion that is anisotropic
w.~r.~t.~the rotational symmetry (\ref{ao02c}) of the ensemble of drifts $b$.
The upcoming Remark \ref{rmk02} shows that in this situation
of conflicting symmetries the heuristics predicts that the drift prevails: Asymptotically, 
i.~e.~for large $L$ or rather $\tilde\lambda_L$, 
our proxy $\tilde a_L$ for the effective diffusivity becomes isotropic. 
This indicates a certain universality of our anomalous-diffusion effect, 
in the sense that $\tilde a_L$ 
becomes independent of the nature of the bare diffusion, just relying on its presence.

\begin{remark}[Universality]\label{rmk02}
We consider the generator ${\mathcal L}=b.\nabla+\nabla.\bar{a}\nabla$ 
for some deterministic and constant positive-definite symmetric bilinear form $\bar{a}$ 
(on cotangent space),
which means that the last item in (\ref{E: initial conditions proxies}) reads $\tilde a_{L=1}=\bar{a}$. 
Then (\ref{tao17}) yields a deterministic evolution for 
(the deterministic and constant positive-definite symmetric bilinear form) $\tilde a_L$
with the asymptotic property
\begin{align*}
\tilde a_L\approx\tilde\lambda_L { I } \quad\mbox{as soon as}\;\tilde\lambda_L\gg_{ \bar{a} } 1.
\end{align*}
In fact, this convergence is exponential in $\tau=\ln\tilde\lambda_L$
with a rate that only depends on $n$.
\end{remark}

\medskip

By definition, the driver $ d \phi^i $ may be seen as a homogenization approximation of the differential of the genuine corrector defined in \eqref{defn-a}. However, as in classical homogenization, we cannot expect $\nabla d \phi^{i}$ to be a good approximation of the differential of the gradient of the genuine corrector.
Instead, as in the intertwining property \eqref{intertwin} we will need to work with the two-scale expansion of $ d \phi^i $, 
which is also conveniently implemented in an incremental manner: Instead of integrating $ d \phi^i $, we apply the two-scale expansion, but this time computed w.~r.~t.~a proxy corrector $ \tilde\phi_L = \tilde\phi_L^i e_i $ defined by the evolution
\begin{align}\label{tao05b}
	d \tilde\phi^i = ( 1 + \tilde\phi_L^j \partial_j ) \circ \! d \phi^i ,
	\quad \tilde\phi^i_{ L = 1 } = 0 ,
\end{align}
which will serve as our ultimate proxy for the genuine corrector. Since the Stratonovich product is stable with respect to smooth approximation, we use it here rather than the It\^{o} one. This would become relevant if we used a smooth large-scale cut-off in the definition of $b_{L}$ instead of a sharp one, in which case $b_{L}$ would no longer have independent increments but we would expect our approach to apply nonetheless. As it turns out, in our current set-up, see Subsection \ref{S: rigorous proxy analysis}, the driver $ d \phi^i $ has the property that the Stratonovich interpretation agrees with the Itô interpretation. The latter, cf.~\eqref{phi-tilde}, is convenient to obtain well-posedness of \eqref{tao05b} since the nonlinearity\footnote{The reader may take two points of view here: On the one hand, evaluating the SDE at a single point $ x $, Lipschitz continuity of the nonlinearity implies well-posedness. On the other hand, the nonlinearity is Lipschitz continuous w.~r.~t.~an operator topology on suitable Sobolev spaces, which implies well-posedness when viewing $ \tilde\phi_L $ as a random field. The latter is convenient since later on we will take spatial derivatives of $ \tilde\phi_L $.} $ ( 1 + \tilde\phi_L^j \partial_j ) $ is an affine function of the unknown $ \tilde\phi_L $.

\medskip

\medskip

In the mathematical community,
the idea of a scale-by-scale homogenization for drift-diffusion processes with random drifts
seems to originate in \cite{BricmontKupiainen91}, a formal, renormalization-theoretic approach that was later made rigorous in \cite{SznitmanZeitouni06}.  
In the setting of anomalous diffusion for processes exhibiting superdiffusive asymptotics, 
this strategy was first implemented in \cite{CMOW} for the problem at hand.
Shortly afterwards, scale-by-scale homogenization has also been employed in the construction
of (time-dependent) drifts that produce anomalous dissipation 
\cite{ArmstrongVicol23,BurczakSzekelyhidiWu23}, and very recently has been refined
in \cite{ABK24} to a quenched invariance principle for the problem at hand.

\medskip

A continuum variance decomposition, i.~e.~a continuum decomposition of 
an underlying Gaussian field, here the tensorial Gaussian free field $\Psi$,
according to spatial scales $L$ is also used in quantum field theory
(where the underlying noise is thermal instead of environmental).
For instance, in case of the $\Phi^4_d$ model for $d=2,3$, \cite{BG} use it to control (exponential) moments of the
field $\phi$ with the help of suitable martingales in $\ln L$, analogous to $\tilde\phi$.
The variance decomposition also induces a
Hamilton-Jacobi-type evolution equation for the effective Hamiltonian in $\ln L$,
known as the flow equation. In \cite[Section 2.3]{GM}, to cite another recent
work, proxys for the solution of the flow equation are constructed by truncation
of an expansion (analogous to an expansion in $\varepsilon$ here),
and the evolution of the residuum is monitored (like $ \tilde f $ later on in Proposition \ref{prop01}).

\medskip

As already alluded to in the introduction, the approach used here contrasts with the one previously used in this context in \cite{CannizzaroHaunschmidSibitzToninelli}.  The strategy used there is motivated by a operator fixed-point equation that arises in the Wiener chaos expansion-based analysis of the resolvent.  It has since been refined in the very recent work \cite{cannizzaro_moulard_toninelli}
 on the 2d anisotropic stochastic Burgers equation, where not only the logarithmic growth of the diffusivity but also the Gaussian fluctuations of the solution are derived.  In both works, the proofs center around the construction of an ad hoc approximation of the operator-valued fixed point (which is not known to exist).  Similar strategies were employed in earlier works on the 2d anisotropic KPZ equation \cite{cannizzaro_erhard_toninelli}, the extension of \cite{CannizzaroHaunschmidSibitzToninelli} to certain time-dependent drifts \cite{DeLimaFeltesWeber24}, and the weak coupling limit of self-repelling Browian polymers \cite{cannizzaro_giles}.
All the aforementioned works use the fact that the Markov process in question (in our setting, the environment seen by the particle) has an explicitly known, tractable Gaussian stationary measure.  This structural condition has also been used in the study of subcritical singular SPDEs, such as the probabilistic approach to the 1d stochastic Burgers equation \cite{gubinelli-perkowski}.

\subsection{Approximation \ref{approx03}: Identification of $\nabla d \phi$ with a Brownian motion on $ \mathfrak{sl}(n) $}

As we already hinted at above, to pass from Approximation \ref{approx02} to \ref{approx03}, we approximate the evolution for $ { \rm id } + \nabla \tilde\phi_L $, which is obtained after applying $ \nabla $ to \eqref{tao05b}, by neglecting the term $ \tilde\phi_L ^j \nabla \partial_j \circ \! d \phi $. Therefore, the bridge from scale-by-scale homogenization to geometric Brownian motion on $ \textbf{SL} ( n ) $ is built upon the observation that --- up to a change of variables --- $ \nabla d \phi $ acts as a Brownian motion on $ \mathfrak{sl}( n ) $, which is characterized by its invariance properties and two extrinsic normalizations.

\begin{lemma}\label{lemR01}
	The process
	\begin{align}\label{forFelix06}
		B_{ \tau } :=  \int_{ 1 }^{ L } \frac{ \ell^2 }{ \tilde\lambda_{\ell} } \nabla d b ( 0 ),
		\quad \text{with the change of variables} \quad \tau = \ln \tilde\lambda_L
	\end{align}
	is the unique (in law) Brownian motion on $ \mathfrak{sl} ( n ) $ satisfying
	\begin{itemize}
		\item[(i)] $ O B O^{ - 1 } =_{\text{law}} B$ for all $O \in \textbf{O}( n )$,
		\item[(ii)] $\mathbb{E} B_{\tau} B_{ \tau } = 0$, and
		\item[(iii)] $\mathbb{E} B_{\tau} B_{\tau}^{*} = \tau \, { \rm id}$. 
    
	\end{itemize}
	Moreover, we have the explicit coupling between $ B $ and $ \nabla d \phi $ given by
	\begin{align}\label{E: coupling second instance}
		B_{ \tau } = \int_1^{ L } \nabla d \phi ( 0 ) 
	\end{align}
	and points (ii) and (iii) imply that 
	\begin{align}
		d [ \nabla \phi \, \nabla \phi ] &= 0 , \label{forFelix03} \\
		d [ \nabla \phi \, ( \nabla \phi )^* ] &= { \rm id } \, d \tau . \label{wr65bis}
	\end{align}
\end{lemma}

\begin{proof}
	First, let us note that \eqref{ao24} shows that $ B $ takes values in $ \mathfrak{sl}(n) $. Furthermore, by \eqref{cov02} the centered Gaussian process $ L \mapsto B_L := \int_1^{L } \frac{ \ell^2 }{ \tilde\lambda_{\ell} } \nabla d b ( 0 ) $ has independent increments. Moreover, applying $ \nabla $ to \eqref{E: eqn for dphi} yields the representation \eqref{E: coupling second instance} from \eqref{forFelix06}. The rest of the proof is devoted to arguing that the change of variables in \eqref{forFelix06} implies the further claims.

	\medskip

Next, let us establish the points (i), (ii), and (iii).  From the $ \textbf{SO} (n) $-invariance of $ b $, which implies its $ \textbf{O} ( n ) $-invariance by point (i) of Lemma \ref{lemRep02} in Appendix \ref{reptheory}, it is evident by \eqref{forFelix06} that $ \nabla \phi $ is $ \textbf{O} ( n ) $-invariant in law under the natural action of $ \textbf{O} ( n ) $. Furthermore, since both maps $ ( E , F ) \mapsto E F $ and  $ ( E , F ) \mapsto E F^* $ commute with the action of $ \textbf{O}(n) $ we learn that
\begin{align*}
	\E B_{ \tau } B_{ \tau } \quad { \rm and } \quad \E B_{ \tau }  B_{ \tau }^*
	\quad { \rm are ~ isotropic ~ endomorphisms ~ of  } ~ { \rm co  tangent ~ space } ,
\end{align*}
and therefore a multiple of the identity, see point (ii) in Lemma \ref{lemRep02} in Appendix \ref{reptheory}. To identify these constants we apply the trace and claim that
\begin{align}\label{rep01next}
	{ \rm tr } \, \E B_{ \tau } B_{ \tau } = 0
	\quad { \rm and } \quad
	{ \rm tr } \, \E B_{ \tau }  B_{ \tau }^* = n \, \tau ,
\end{align}
which translates into properties (ii) and (iii) respectively. 

\medskip

Since $ B $ is a continuous martingale with independent increments, we have $ \E B^2 = [ B B ] $ and $ \E B B^*= [ B B^* ] $, which by virtue of definition \eqref{E: coupling second instance} implies that
\begin{align*}
	d \E B B
	= d [ B B ]
	= d [ \nabla \phi \, \nabla \phi ] 
	\quad { \rm and } \quad
	 d \E B B^* 
	=d [ B B^* ]
	= d [ \nabla \phi \, ( \nabla \phi )^* ] .
\end{align*}
Thus \eqref{rep01next} is equivalent to
\begin{align}\label{forFelix05}
	{ \rm tr } \, d [ \nabla \phi \, \nabla \phi ]  = 0 \, d \tau 
	\quad { \rm and } \quad
	{ \rm tr } \, d [ \nabla \phi \, ( \nabla \phi )^* ]  = n \, d \tau ,
\end{align}
and \eqref{forFelix03} and \eqref{wr65bis} hold once \eqref{forFelix05} is proven.

\medskip

For the first item in \eqref{forFelix05} we appeal to stationarity, which, as in the argument for \eqref{tao17y}, allows for an integration by parts:
\begin{align*}
	{ \rm tr } \, d  [ \nabla \phi \, \nabla \phi ]
	\stackrel{ (\ref{E: eqn for dphi}) }{ = } \frac{ L^4 }{ \tilde\lambda_L^2 } \delta_{ \ell }^{ i } d [ \partial_i b^j \, \partial_j b^\ell ]
	= \frac{ L^4 }{ \tilde\lambda_L^2 } \delta_{ \ell }^{ i } d [ \partial_j b^j \, \partial_i b^\ell ]
	= \frac{ L^4 }{ \tilde\lambda_L^2 } d [ \nabla . b  \,  \nabla . b ] .
\end{align*}
The latter vanishes by \eqref{ao24}. By the same logic, we can compute
\begin{align*}
	{ \rm tr } \, d  [ \nabla \phi \, ( \nabla \phi )^* ]
	= \frac{ L^4 }{ \tilde\lambda_L^2 } \delta^{ i j }  d [ \partial_i b  \cdot \partial_j b ]
	= \frac{ L^4 }{ \tilde\lambda_L^2 } \delta^{ i j }  d [ - \partial_i \partial_j b  \cdot b ]
	=  \frac{ L^4 }{ \tilde\lambda_L^2 } d [ ( - \Delta ) b  \cdot b ] .
\end{align*}
Inserting \eqref{fw01} in this identity, the above equation turns into
\begin{align}\label{dertrick}
	{ \rm tr } \,  d [ \nabla \phi \, ( \nabla \phi )^* ]
	= \frac{ L^2 }{ \tilde\lambda_L^2 } \, d  [ b \cdot b ] 
	\stackrel{ (\ref{forFelix01}) }{ = } \varepsilon^2 \frac{ n }{ 2 } \frac{ d \ln L }{ \tilde\lambda_L^2 } .
\end{align}
Thus, to conclude the second item in \eqref{forFelix05} (and hence \eqref{rep01next}), we need to choose $ \tau $ such that
\begin{align}\label{forFelix04}
	d \tau = \varepsilon^2 \frac{ d \ln L }{ 2 \tilde\lambda_L^2 } .
\end{align}
In view of \eqref{ODElambda}, this equation is solved by $ \tau = \ln \tilde\lambda $.

\medskip

Finally we argue that after the change of variables $ L \mapsto \tau $ the process $ B $ becomes a Brownian motion. Indeed, by Lemma \ref{L: characterize covariance} below, Gaussianity and the formerly established points (i), (ii) and (iii) imply that there exists a Gaussian random variable $ X $ such that $ B_{ \tau } =_{ \rm law } \tau^{ \frac{1}{2} } X $. Together with the independence of increments, which is preserved by the change of variables, this scaling property shows that $ B $ is a Brownian motion. En passant we learn that the law of $ B $ is unique: By Lemma \ref{L: characterize covariance}, $ X $ and thus, by independence of increments, also $ B $ is uniquely determined in law.
\end{proof}

We now argue that the Fourier symbol of the (tensorial) covariance characterizing
the driving (Gaussian) martingale $ \int_1^{\cdot} \nabla d\phi $ scales like $ ( |k|^n \ln \frac{1}{|k|} )^{-1} $ on large scales, which corresponds to a $  \log \log $-correlated field. Indeed, since the Fourier transform of $ d b $ is supported on the sphere $ \{ L | k | = 1 \} $, see \eqref{forFelix07}, we may use \eqref{fw01} and \eqref{lambda} to compute
\begin{align*}
	\mathcal{F} \nabla d \phi
	\stackrel{\eqref{E: eqn for dphi}}{=} \mathcal{F} d b \otimes \iota \tilde\lambda_L^{-1} L^2 k 
	\stackrel{\eqref{forFelix07} \& \eqref{lambda}}{ = } \mathcal{F} d b \otimes \frac{ \iota k }{ | k |^2 { \sqrt{ 1 + \varepsilon^2 \ln \frac{1}{ | k | } } }  } .
\end{align*}
The right hand side only depends on $ L $ through the driver $ db $. Thus by integrating the expression we learn that the covariance tensor of $ \int_1^{\cdot} \nabla d\phi $ has Fourier transform given by
\begin{align*}
	\frac{ \mathcal{F} c^{ i j } }{ | k |^4  ( 1 + \varepsilon^2 \ln \frac{1}{ | k | } ) }  \, ( e_i \otimes k )  \otimes ( e_j  \otimes k ),
	~ ~ { \rm where } ~
	c = c^{ i j } \, e_i \otimes e_j 
	~ { \rm is ~ the ~ covariance ~ tensor ~ of } ~ b .
\end{align*}
We can now use \eqref{ao08} and \eqref{cov01fo} to read off the scaling: since $\mathcal{F}c$ scales like $ \sim \varepsilon^{2} |k|^{2-n}$, the covariance of $\nabla \phi_{L = \infty}$ scales like $ \sim \frac{ 1 }{ | k |^n} $, as in the case of a logarithmically correlated field, up to the logarithmic discount $ \sim \varepsilon^2 ( 1 + \varepsilon^2 \ln \frac{ 1 }{ | k | } )^{ - 1 } $.
For very large scales, i.~e.~$\varepsilon^2\ln\frac{1}{|k|}\gg 1$, 
the prefactor is $\approx ( \ln \frac{1}{|k|} )^{-1}$
and thus independent of $\varepsilon$; in this range,
the overall scaling is $ ( |k|^n \ln \frac{1}{|k|} )^{-1} $, corresponding to a $ \log \log $-correlated field. Random fields that behave like the exponential of $ \log \log $-correlated fields also occur in the very critical regime of directed polymers, see \cite{CaravennaSunZygouras}. 

\subsection{Rigorous implementation of Approximation \ref{approx02}: scale-by-scale homogenization} \label{S: rigorous proxy analysis}Recall from Subsection \ref{S: heuristics} that $ ( \tilde a_L , d \phi^i , d \sigma^i ) $ are defined in such a way that\footnote{Recall that we write $ I = \delta^{ i j } e_i \otimes e_j $.}
\begin{align}\label{inchelm}
	da e^i+\tilde a_L\nabla d\phi^i=\nabla.d\sigma^i,
	\quad { \rm with } \quad
	\tilde a_L = \tilde\lambda_L \, I ,
	~ \tilde\lambda_L = \sqrt{ 1 + \varepsilon^2 \ln L } .
\end{align}

\medskip

Next, following the discussion in Subsection \ref{S: heuristics}, we define $ \tilde\phi^i $ as the solution to the Stratonovich SDE
\begin{align}\label{phi-tilde-s}
d\tilde\phi^i=(1+\tilde\phi_L^j\partial_j) \circ \! d\phi^i \quad \mbox{and} \quad \tilde\phi_{L=1}^i=0 .
\end{align}
As already mentioned in that discussion, as a corollary of \eqref{forFelix03} in Lemma \ref{lemR01}, we find that the Itô-Stratonvich correction vanishes. Indeed, note that
\begin{align*}
	\tilde\phi^j \partial_j \circ \! d \phi^i -\tilde\phi^j \partial_j d \phi^i 
	= \frac{1}{2} d [ \tilde\phi^j \, \partial_j \phi^i ]
	\stackrel{ (\ref{phi-tilde-s} ) }{ = } \frac{1}{2} d [ \phi^j \, \partial_j \phi^i ] + \frac{1}{2} \tilde\phi^k d [ \partial_k \phi^j \, \partial_j \phi^i ] .
\end{align*}
The first summand on the r.~h.~s.~vanishes by an integration by parts since $ \nabla . \phi = 0 $, see proof of Lemma \ref{lemR01}. The second summand vanishes by \eqref{forFelix03}.
Hence \eqref{phi-tilde} may also be formulated in the Itô sense
\begin{align}\label{phi-tilde}
	d \tilde\phi^i = ( 1 + \tilde\phi^j \partial_j ) d \phi^i 
	\quad \mbox{and} \quad
	\tilde\phi^i_{ L = 1 } = 0 ,
\end{align}
which we will exclusively use later on.

\medskip

The upgrade of $ d\sigma^i $ is less obvious, and motivated by Lemma \ref{lem01} below. In Stratonovich form it is given by
\begin{align}\label{sigma-tilde-s}
	d \tilde\sigma^i = d \sigma^i 
	+ (\partial_j \circ \! d \phi^i) \, \tilde\sigma_L^j
	+ \tilde\phi_L^i \circ \! d \Psi
	\quad \mbox{and} \quad
	\tilde\sigma_{L=1}^i=0 .
\end{align}
Since we will never use the Stratonovich form \eqref{sigma-tilde-s}, let us state without proof the Itô form 
\begin{align}\label{sigma-tilde}
	d \tilde\sigma^i = d \sigma^i 
	+ (\partial_j d \phi^i) \, \tilde\sigma_L^j
	+ \tilde\phi_L^i d \Psi
	- \tilde\phi_L ^j  d [ ( \partial_j \phi^i ) \, \Psi ] 
	\quad \mbox{and} \quad
	\tilde\sigma_{L=1}^i=0 ,
\end{align}
which we will use below. Note that \eqref{sigma-tilde} preserves the skewness of $ \tilde\sigma^i $.
Indeed, along the lines of the above argument, one may check that $ (\partial_j \circ \! d \phi^i) \, \tilde\sigma_L^j - (\partial_j d \phi^i) \, \tilde\sigma_L^j = \frac{ 1 }{ 2 } \tilde\phi_L^j d [ ( \partial_j  \phi^i ) \, \Psi ] = \tilde\phi_L^i \circ \! d \Psi - \tilde\phi_L^i d \Psi $.

\medskip

Finally, we monitor the deviation from \eqref{tao01} for $( \tilde a_L,\tilde\phi_L,\tilde\sigma_L)$
defined through \eqref{inchelm}, \eqref{phi-tilde}, and \eqref{sigma-tilde} by analyzing
the residuum, the tangent vector field $ \tilde f_L^i$ defined through
\begin{align}\label{helm-tilde}
	a_L (e^i + \nabla \tilde\phi_L^i) = \tilde a_L e^i + \nabla. \tilde\sigma_L^i + \tilde f_L^i.
\end{align}
Definition \eqref{helm-tilde} naturally makes $ \tilde f_L^i $ an element of tangent space, and we thus regard $ \tilde f_L $ as an element of $ { \rm Hom } ( { \rm co tangent ~ space } , { \rm tangent ~ space } ) $. From the equations for $ ( \tilde a_L , \tilde\phi_L , \tilde\sigma_L ) $ we can obtain a characterization of $ \tilde f_L $ in terms of an Itô SDE, which is analogous to the finite-difference equation obtained in \cite{CMOW}.

\begin{lemma}[{Continuum version of \cite[(50)]{CMOW}}]\label{lem03}
	We have
	\begin{align}\label{wr03}
		d \tilde f =  \tilde f_L \nabla d\phi + ( \tilde\phi_L ^i a + \tilde\sigma_L ^i )\nabla\partial_i d\phi
		+ \tilde\phi_L \otimes d b \quad { \rm and } \quad \tilde f_{ L = 1 }=0 .
	\end{align}
\end{lemma}

The r.~h.~s.~contribution $ ( \tilde\phi_L^i a + \tilde\sigma_L^i ) \nabla \partial_i d \phi $ arises from homogenization, see \eqref{intertwin}, while $ \tilde\phi_L \otimes d b $ comes from linearization. Furthermore, note that \eqref{wr03} shows that the martingale $ \tilde f $ has vanishing expectation, 
which was part of its definition in \cite[(48)]{CMOW}. This implies the exact consistency
$\mathbb{E}a_L ({\rm id}+\nabla\tilde\phi_L )$ $=\tilde\lambda_L I $,
which in \cite{CMOW} only holds in the limit $M \searrow 1$, see \cite[(51)]{CMOW}.

\begin{proof}[Proof of Lemma \ref{lem03}]
	\phantom{.}
	
	\medskip
	
	\textit{Step 1 (Main argument).} The claimed identity \eqref{wr03} in $ { \rm Hom } ( { \rm cotangent ~ space } , { \rm  tangent ~ space } ) $ amounts in every coordinate direction $ e^i $ to the identity in tangent space
\begin{align}\label{ttao09bis}
	d \tilde f^{ i } 
	= \tilde f_L \nabla d \phi^{ i } 
	+ ( \tilde\phi_L ^j a_L + \tilde\sigma_L ^j ) \nabla \partial_j d \phi^{ i }
	+ \tilde\phi_L ^{ i } d b ,
\end{align}
which we will derive from \eqref{helm-tilde}.

\medskip

By Leibniz' rule within Itô's calculus, we obtain
\begin{align*}
d a ( e^{ i } + \nabla\tilde\phi^{ i } ) 
= da e^i+ da \nabla\tilde\phi_L^i
+a_L\nabla d\tilde\phi^i+d[a\nabla\tilde\phi^i], 
\end{align*}
into which we insert $ d a = d \Psi $ (which is immediate from \eqref{defn-a}) and \eqref{inchelm} to the effect of
\begin{align}\label{ttao12}
d a ( e^{ i } + \nabla\tilde\phi^{ i } ) 
=- \tilde a_L \nabla d\phi^i
+ \nabla . d\sigma^i 
+ d \Psi \nabla\tilde\phi_L^i
+ a_L\nabla d\tilde\phi^i
+d[ \Psi \, \nabla\tilde\phi^i].
\end{align}
In view of the two last terms in \eqref{ttao12}, we apply $ \nabla$ to \eqref{phi-tilde}, which yields by Leibniz' rule
\begin{align}\label{ttao10}
\nabla d\tilde\phi^i
= \partial_jd\phi^i \, (e^j+\nabla\tilde\phi_L^j)
+ \tilde\phi_L^j \nabla\partial_jd\phi^i.
\end{align}
Below, we will argue that
\begin{align}\label{parity}
	d [ \Psi \otimes \nabla^2 \phi ] = 0 
\end{align}
as a consequence of parity. Using \eqref{parity}, we obtain from (\ref{ttao10}) that
the last term in (\ref{ttao12}) can be written as
\begin{align*}
d[ \Psi \, \nabla\tilde\phi^i]
=d[ \partial_j\phi^i  \, \Psi ] (e^j+\nabla\tilde\phi_L^j)
= d[ \Psi \, \nabla \phi^i ]
+ d[\partial_j\phi^i \, \Psi ] \, \nabla\tilde\phi_L^j.
\end{align*}
Thus, using the first item in \eqref{tao17}, we learn from \eqref{ttao12} that
\begin{align*}
	d ( a ( e^{ i } + \nabla\tilde\phi^{ i } ) - \tilde a e^i ) 
	=- \tilde a_L \nabla d\phi^i
	+ \nabla . d\sigma^i 
	+ d \Psi \nabla\tilde\phi_L^i
	+ a_L\nabla d\tilde\phi^i
	+ d[\partial_j\phi^i \, \Psi ] \, \nabla\tilde\phi_L^j.
\end{align*}
Moreover, applying $a_L$ to identity (\ref{ttao10}), and using the definition of $ \tilde f_L $ 
to substitute $a_L(e^j+\nabla\tilde\phi_L^j)$, we obtain that the penultimate term of the previous equation is
\begin{align*}
a_L\nabla d\tilde\phi^i
= \partial_j d\phi^i \, ( \tilde a_L e^j +\nabla.\tilde\sigma_L^j + \tilde f_L^j)
+\tilde\phi_L^j a_L\nabla\partial_jd\phi^i.
\end{align*}
Combining the last two equations and using $ \partial_j d\phi^i \, \tilde a_L e^j = \tilde a_L \nabla d \phi^i $ yields
\begin{align*}
d ( a ( e^{ i } + \nabla\tilde\phi^{ i } ) - \tilde a e^i )
& =  \partial_j d\phi^i \, \tilde f_L^j
+ \tilde\phi_L^j a_L\nabla\partial_jd\phi^i \\
& \qquad + \nabla . d\sigma^i 
+ d \Psi \nabla\tilde\phi_L^i
+ \partial_j d\phi^i \,  \nabla.\tilde\sigma_L^j
+ d[\partial_j\phi^i \, \Psi ] \, \nabla\tilde\phi_L^j  \\
&= \partial_j d\phi^i \, \tilde f_L^j 
+ ( \tilde\phi_L^j a_L +  \tilde\sigma_L^j ) \nabla \partial_j d\phi^i
+ \tilde\phi_L^i \nabla . d \Psi \\
& \qquad + \nabla . ( d \sigma^i + \partial_j d\phi^i  \, \tilde\sigma_L^j  - \tilde\phi_L^i d \Psi  - \tilde\phi_L^j d[\partial_j\phi^i \, \Psi ]), 
\end{align*}
where in the last equality we used the Leibniz' rule from Footnote \ref{leibniz} for $( \zeta ,\sigma) = ( \tilde\phi_L^i, d\Psi )$, $ (\zeta,\sigma) =(\tilde\phi_L^j,d[\partial_j\phi^i \, \Psi ]) $ (for which $ \nabla \sigma $ vanishes) and $ ( \zeta , \sigma ) = (\partial_j \tilde\phi_L^i, \tilde\sigma_L^j ) $. As advertised earlier, the last equation serves as our motivation for definition \eqref{sigma-tilde}. Inserting this equation we thus learn the claim in its component-wise form \eqref{ttao09bis}.

\medskip

\textit{Step 2 (Argument for \eqref{parity}).}  To argue that \eqref{parity} holds we will establish that
\begin{align}\label{parity2}
	\Psi_L ( 0 ) ~ { \rm and } ~ ( \phi_L , \nabla^2 \phi_L ) ( 0 )
	~ { \rm have ~ opposite ~  parity ~ (in ~ law) ~ under ~ the ~ inversion } ~
	x \mapsto - x . 
\end{align}
Using the independence of increments to evaluate the covariation $ d [ \Psi \otimes  \nabla^2 \phi ]  = d \E \Psi \otimes \nabla^2 \phi $ and appealing to stationarity, we obtain \eqref{parity}.

\medskip

Here comes the argument for \eqref{parity2}: $ \Psi_L $ and $ \phi_L $ are defined in terms of $ b $ through divergence-form second-order differential operators, which by virtue of \eqref{fw01} take the form\footnotemark
\begin{align}\label{parity3}
	\frac{ \tilde\lambda_L  }{ L^2 } d \phi \stackrel{ \eqref{E: eqn for dphi} }{ = }  d b , \quad { \rm and } \quad
	\frac{ 1 }{ L^2 } d \Psi \stackrel{ \eqref{gaugePsi} }{ = }  \nabla^* \otimes d b - d b \otimes \nabla^*  .
\end{align}%
\footnotetext{Here, and only here, we use the symbol $ \nabla^* $ to denote the gradient vector $ \delta^{ i k } \partial_k \, e_i $; and, a bit formal, $ \nabla^* \otimes d b - d b \otimes \nabla^* $ denotes the tensor $ ( \delta^{ i p } \partial_{ p } d b^j - \delta^{ j q } \partial_{ q } d b^i ) e_i \otimes e_j $.}%
In particular, $ d \phi $, resp. $ d \Psi $, transforms like the r.~h.~s.~expression $ db $, resp.~$  \nabla^* \otimes d b - d b \otimes \nabla^*  $, under the change of variables $ x \mapsto - x $. Since the latter have opposite parity (in law) this implies \eqref{parity2}.
\end{proof}

\medskip

In order for \eqref{helm-tilde} to provide an effective approximation of \eqref{tao01} it is important that the residuum $ \tilde f_L $ is small compared to the size of the effective diffusion matrix $ \tilde a_L $. This control was already proven in the previous work \cite{CMOW} using the finite-difference analogue of \eqref{wr03}.

\begin{proposition}[{see \cite[(79)]{CMOW}}]\label{prop01}
	For $\varepsilon^2\ll 1$ it holds
	\begin{align*}
		\E | \tilde f_L |^2 \lesssim \varepsilon^2 \tilde\lambda_L .
	\end{align*}
\end{proposition}

Here, we equip the space $ { \rm Hom } ( { \rm co tangent ~ space } , { \rm tangent ~ space } )  $, to which $ \tilde f_L $ belongs, with the Frobenius norm $ | \tilde f_L |^2 = { \rm tr } \, \tilde f_L^* \tilde f_L $, where $ \tilde f_L^* \in { \rm Hom } ( { \rm  tangent ~ space } , { \rm co tangent ~ space } ) $ denotes the adjoint of $ \tilde f_L $.

\medskip

Within the continuum approach to scale-by-scale homogenization, we can provide a transparent proof of Proposition \ref{prop01} using the ODE for $ \E | \tilde f_L |^2 $, which is reminiscent of the iteration formulas in \cite{CMOW}, and follows as an elementary application of Itô calculus to equation \eqref{wr03}. We summarize this step in the next lemma.

\begin{lemma}[{Continuum version of \cite[(79)]{CMOW}}]\label{lem01}
	We have
	\begin{equation}\label{add07}
		d\mathbb{E}| \tilde f |^2 - \mathbb{E} | \tilde f_L |^2 d \tau 
		\lesssim  \E|\tilde\phi_L \otimes a_L  + \tilde\sigma_L |^2 \frac{ d \tau }{ L^2 } 
		+ \tilde\lambda_L^2 \E| \tilde\phi_L |^2 \frac{ d \tau }{L^2}.
	\end{equation}
\end{lemma}

Here is an important structural remark regarding the differential inequality \eqref{add07}: The exponential structure in $ d \tilde f = \tilde f \nabla d \phi + \hdots $ in \eqref{wr03} results in the exponential structure $ d\mathbb{E}| \tilde f |^2 \leq \mathbb{E} | \tilde f_L |^2 d \tau + \hdots $, which implies exponential growth $ e^{ \tau } = \tilde\lambda_L $. Thus, in the following proof it is important to capture the constant in front of $ \mathbb{E} | \tilde f_L |^2 d \tau $ by carefully treating all cross variations of \eqref{wr03} involving $ \tilde f_L \nabla d \phi $, while coarse estimates on the other terms are sufficient. In particular, the norm used on the term on the r.~h.~s.~of \eqref{add07} does not matter. The small prefactor in Proposition \ref{prop01} comes from the smallness of the r.~h.~s.~in \eqref{add07}, which is the subject of a later lemma.

\begin{proof}[Proof of Lemma \ref{lem01}]
Since by \eqref{wr03}, $ \tilde f $ is a martingale vanishing at zero, its quadratic variation determines the evolution of 
$\mathbb{E}| \tilde f |^2$, namely $ \E | \tilde f |^2 = \E [ { \rm tr } \tilde f^* \tilde f ] $. Turning to (\ref{wr03}), we note that the parity argument
used for \eqref{parity} yields
\begin{align}\label{eqn04prev}
	d [\nabla \phi \, (\nabla\partial_i \phi)^*]=0
	\quad\mbox{and}\quad
	d [\nabla \phi \, ( e_i \otimes b )^*]=0.
\end{align}
Hence the cross variations between the first and second and first and last terms on the r.~h.~s.~of \eqref{wr03} vanish. Thus, by applying Cauchy-Schwarz (for the inner product $ \E [ { \rm tr } \, f^* g ] $) to the cross variation of the second and third term, we learn that
\begin{align}
\E | \tilde f_{ \cdot } |^2
& \leq \E [ { \rm tr } \, ( { \textstyle \int_1^{ \cdot } } \tilde f_L \nabla d \phi )^* ( { \textstyle \int_1^{ \cdot } } \tilde f_L \nabla d \phi ) ] \nonumber \\
& + 2 \E [ { \rm tr } \, ( { \textstyle \int_1^{ \cdot } } (\tilde\phi_L^i a_L + \tilde\sigma_L^i) \nabla \partial_i d \phi )^* ( { \textstyle \int_1^{ \cdot } } (\tilde\phi_L^{j} a_L + \tilde\sigma_L^{j} ) \nabla \partial_{j} d \phi ) ] \nonumber \\
& + 2 \E [ { \rm tr } \, ( { \textstyle \int_1^{ \cdot } } \tilde\phi_L \otimes d b )^* ( { \textstyle \int_1^{ \cdot } } \tilde\phi_L \otimes db ) ]. \label{eqn04}
\end{align}
Recall that the factor one in front of the term $ \E [ { \rm tr } (  { \textstyle \int_1^{ \cdot } }  \tilde f_L \nabla d \phi )^* (  { \textstyle \int_1^{ \cdot } }  \tilde f_L \nabla d \phi ) ] $ is due to its vanishing cross variations, and is crucial to determine the exponent of $ \tilde\lambda_L $ in Proposition \ref{prop01}.

\medskip

Due to the two algebraic identities $  {\rm tr} \, (f g )^* f g ={\rm tr} \, f^*f g g ^* $ and $ { \rm tr } ( \dot x \otimes \xi )^* ( \dot x \otimes \xi ) = | \dot x |^2 | \xi |^2 $ inequality \eqref{eqn04}  may be rewritten as
\begin{align*}
d \E | \tilde f |^2
& \leq \E { \rm tr } \tilde f_L ^* \tilde f_L d  [ \nabla \phi ( \nabla \phi )^*  ] \\
& + 2 \E { \rm tr }  (\tilde\phi_L^i a_L + \tilde\sigma_L^i)^* (\tilde\phi_L^{j} a_L + \tilde\sigma_L^{j} )  d [ \nabla \partial_i \phi ( \nabla \partial_{j} \phi )^* ]
+ 2 \E | \tilde\phi_L |^2 d [ b  \cdot b ].
\end{align*}
As alluded to earlier, the constants in front of the last two terms are irrelevant, thus by another application of Cauchy-Schwarz, we obtain
\begin{align*}
d \E | \tilde f |^2 - \E { \rm tr } \tilde f_L ^* \tilde f_L d [ \nabla \phi ( \nabla \phi )^*  ]
\lesssim \E | \tilde\phi_L \otimes a_L + \tilde\sigma_L |^2 d [ \delta^{ i j } { \rm tr } \,  \nabla \partial_i \phi ( \nabla \partial_{j} \phi )^* ]
+ \E | \tilde\phi_L |^2 d [ b  \cdot b ].
\end{align*}
Thus to conclude \eqref{add07}, it remains to identify three quadratic variations.

\medskip

First, let us observe that by \eqref{wr65bis} in Lemma \ref{lemR01} the tensor $ d [\nabla \phi(\nabla \phi )^*] $ is equal to $ { \rm id } \, d \tau $ identifying the l.~h.~s.~of \eqref{add07}.

\medskip

Next, we identify the term $ [ \delta^{ i j } { \rm tr } \,  \nabla \partial_i \phi ( \nabla \partial_{j} \phi )^* ] $. To this end, we argue as in \eqref{dertrick}: Since the quadratic variation is deterministic and thus by stationarity constant in space, we can appeal to a partial integration (or Leibniz' rule) showing that $ [ \delta^{ i j } { \rm tr } \,  \nabla \partial_i \phi ( \nabla \partial_{j} \phi )^* ] = [ { \rm tr } \,  \nabla ( - \Delta ) \phi ( \nabla \phi )^* ]  $. By definition \eqref{E: eqn for dphi} and \eqref{fw01} the operator $  - \Delta  $ acts as $ L^{ - 2 } $ on the increment $ d \phi $, which shows
\begin{align}\label{dertrick2}
	d [ \delta^{ i j } { \rm tr } \, \nabla \partial_j \phi (\nabla \partial_i \phi )^* ]
	= d  [ { \rm tr } \, \nabla \phi (\nabla ( - \Delta ) \phi )^* ]
	= \frac{ 1 }{ L^2 } d  [ { \rm tr } \, \nabla \phi (\nabla \phi )^* ] 
	\stackrel{ (\ref{wr65bis}) }{ = } \frac{ { \rm tr } \, { \rm id } }{ L^2 } d \tau .
\end{align}
This identifies the first term on the r.~h.~s.~of \eqref{add07}.

\medskip

Finally, the term $ d [ b  \cdot b ] $ may be evaluated using \eqref{forFelix01} and recalling the ODE \eqref{forFelix04} satisfied by the variable $ \tau $.
\end{proof}

\medskip

In order to deduce Proposition \ref{prop01} from Lemma \ref{lem01}, we need to estimate the r.~h.~s.~of \eqref{add07}, which boils down to computing second moments of $ \tilde\sigma_L $ and fourth moments of $ \tilde\phi_L $. (The fourth moment becomes convenient when estimating the product $ \tilde\phi_L \otimes a_L $ coming from the homogenization error.) As in \cite{CMOW}, we have

\begin{lemma}[{Continuum version of \cite[(69) \& (70)]{CMOW}}]\label{lem02}
	It holds
	\begin{align*}
		\tilde\lambda_L \, \E^{ \frac{ 1 }{ 4 } } | \tilde\phi_L |^4 + \E^{ \frac{ 1 }{ 2 } } | \tilde\sigma_L |^2
		\lesssim \varepsilon L .
	\end{align*}
\end{lemma}

Note that the lemma states that the approximation of the corrector $ \tilde\phi_L $ is $ \mathcal{O} ( \varepsilon \tilde\lambda_L^{ - 1 } ) $ times smaller than linear growth, so that in combination with Proposition \ref{prop01}, this approach amounts to a quantitative stochastic homogenization result of order $ \varepsilon \tilde\lambda_L^{ - \frac{1}{2} } $. The proof of Lemma \ref{lem02} follows the same spirit as the one of Proposition \ref{prop01}, and is based on the derivation of ODEs for the moments using Itô calculus. 

\begin{proof}[Proof of Lemma \ref{lem02}] \phantom{.}

\medskip

\textit{Step 1 (Estimate on $ \E | \tilde\phi_L |^4 $).} In this step, we derive\footnote{As opposed to the residuum $ \tilde f $ the exact constants do not matter here.}
\begin{equation}\label{eqn4m}
d \E | \tilde\phi |^4
\lesssim \E | \tilde\phi_L |^4 d \tau
+ \E | \tilde\phi_L |^2 L^2 d \tau ,
\end{equation}
and show that this differential inequality implies the claimed growth of $ \E | \tilde\phi_L |^4 $. We start with the latter.

\medskip

Dividing \eqref{eqn4m} by $ e^{ c \tau } $ yields
\begin{align*}
	\frac{ d }{ d \tau } \frac{ \E | \tilde\phi_L |^4 }{ e^{ c \tau } }
	=  \frac{ 1 }{ e^ { c \tau } } \Big( \frac{ d }{ d \tau } \E | \tilde\phi_L |^4 - c \E | \tilde\phi_L |^4 \Big) 
	\leq c \frac{ L^2 }{ e^{ c \tau } } \E | \tilde \phi_L |^2 ,
\end{align*}
for the implicit constant $ c $ in \eqref{eqn4m}, and thus
\begin{align*}
	\E | \tilde\phi_L |^4
	\lesssim e^{ c \tau ( L ) } \Big( \int_{ \tau ( 1 ) } ^{ \tau ( L ) } d \tau \frac{ L^2( \tau )  }{ e^{ c \tau } } \Big) \sup_{ \ell \leq L } \E | \tilde\phi_{ \ell } |^2 
	\leq e^{ c \tau ( L ) } \Big( \int_{ \tau ( 1 ) } ^{ \tau ( L ) } d \tau \frac{ L^2( \tau ) }{ e^{ c \tau } } \Big) \sup_{ \ell \leq L } \E^{ \frac{1}{2} } | \tilde\phi_{ \ell } |^4 ,
\end{align*}
where we used Jensen's inequality in the second step. To conclude, we observe that the integral times $ e^{ c \tau } $ is monotone in $ \tau $, so that
\begin{align*}
	\E | \tilde\phi_{ \ell } |^4
	\lesssim e^{ c \tau ( L ) } \Big( \int_{ 0 } ^{ \tau ( L ) } d \tau \frac{ L^2 ( \tau ) }{ e^{ c \tau } } \Big) \sup_{ \ell \leq L } \E^{ \frac{1}{2} } | \tilde\phi_{ \ell } |^4 
	\quad { \rm provided } \quad
	\ell \leq L .
\end{align*}
Therefore, we can absorb the supremum on the r.~h.~s.~and compute the remaining integral (which we defer to Appendix \ref{AppIntegrals}) to obtain 
\begin{align*}
	\sup_{ \ell \leq L } \E^{ \frac{ 1 }{ 2 } } | \tilde\phi_{ \ell } |^4
	\lesssim e^{ c \tau ( L ) } \int_{ 0 } ^{ \tau ( L ) } d\tau \frac{ L^2(\tau) }{ e^{ c \tau } } 
	\stackrel{ (\ref{int01}) }{\lesssim} \Big( \varepsilon \frac{ L }{ e^{ \tau (L )  } } \Big)^2 ,
\end{align*}
which by virtue of $ \tau( L ) = \ln \tilde\lambda_L $ translates into the claimed bound on $ \E | \tilde\phi_L |^4 $.

\medskip

Now comes the argument for \eqref{eqn4m}: Applying Itô's formula to the $ 4 $-linear expression $ | \tilde\phi_L |^4 = ( \tilde\phi_L^* . \tilde\phi_L ) (  \tilde\phi_L^* . \tilde\phi_L ) $, which is symmetric in the first two and last two indices, we learn that
\begin{align*}
	d | \tilde\phi |^4
	= 4 | \tilde\phi _L |^2 \tilde\phi_L \cdot d \tilde\phi
	+ 2 \big(  | \tilde\phi_L |^2 I^* + 2 \, \tilde\phi_L^{ * } \otimes \tilde\phi_L^{ * }  \big) . d [ \tilde \phi \otimes \tilde\phi ] .
\end{align*}
The first term $ 4 | \tilde\phi_L |^2 \tilde\phi_L \cdot d \tilde\phi $ has mean zero and thus is not relevant for us. The quadratic variation is contracted against the second derivative of a $ 4 $-linear expression, i.~e.~a $ 2 $-linear expression in $ \tilde\phi $.  We may use the same parity argument as \eqref{parity} to further expand the quadratic variation
\begin{align*}
d [ \tilde \phi \otimes \tilde\phi ] 
= d [ \phi \otimes \phi ] 
+ \tilde\phi^i \tilde\phi^j d [ \partial_i \phi \otimes \partial_j \phi ] . 
\end{align*}
Thus the SDE for $ \E | \tilde\phi |^4 $ consists of $ 2 $-linear terms in $ \tilde\phi $ driven by $ d [ \phi \otimes \phi ] $ and $ 4 $-linear terms in $ \tilde\phi $ driven by $ [ \nabla \phi \otimes \nabla \phi ] $. Since we can control the entries of (the tensor valued-measures) $ d [ \phi \otimes \phi ] $ and $ d [ \nabla \phi \otimes \nabla \phi ] $ by $ d [ \phi \cdot \phi ] $ and $ d [ { \rm tr } \nabla \phi ( \nabla \phi )^* ] $, we thus obtain 
\begin{align*}
	d \E | \tilde\phi |^4
	\lesssim \E | \tilde\phi |^2 d [ \phi \cdot \phi ] 
	+ \E | \tilde\phi |^4 d [ { \rm tr }  \, \nabla \phi ( \nabla \phi )^* ]  .
\end{align*}
To conclude it thus remains to insert the two estimates
\begin{align*}
	d [ \phi \cdot \phi ] \lesssim L^2 d \tau 
	\quad { \rm and } \quad
	d [ { \rm tr } \, \nabla \phi ( \nabla \phi )^* ] \stackrel{ (\ref{wr65bis}) }{ \lesssim } d \tau  .
\end{align*}
The latter follows from applying \eqref{wr65bis} in Lemma \ref{lemR01}. The first estimate follows from the second by arguing as in \eqref{dertrick}.

\medskip

\textit{Step 2 (Estimate on $ { \rm tr } \, d [ \Psi^* \Psi ] $).} As an auxiliary step, we need to argue that
\begin{align}\label{estPsi2}
	d [  { \rm tr } \, \Psi^* \Psi ]
	\lesssim \tilde\lambda_L^2 d \tau .
\end{align}

\medskip

To derive \eqref{estPsi2} we utilize the Coulomb-gauge definition of $ \Psi $ in form of \eqref{parity3} followed by an integration by parts to obtain
\begin{align*}
	d [ { \rm tr } \, \Psi^* \Psi ] 
	\stackrel{ \eqref{parity3} }{ \lesssim } L^4 \delta^{ i j } d [ \partial_i b \cdot \partial_j b ]
	= L^4 \delta^{ i j } d [ (-\partial_i \partial_{j}) b \cdot b ]
	=  L^4 d [ ( - \Delta )  b \cdot b ] \stackrel{ \eqref{fw01} }{ = } L^{2} d [ b \cdot b ].
\end{align*}
To conclude we proceed as in the proof of \eqref{dertrick}: We invoke \eqref{forFelix01} and \eqref{forFelix04} to obtain \eqref{estPsi2}.

\medskip

\textit{Step 3 (Estimate on $ \E | \tilde\sigma_L |^2 $).} We claim that
\begin{equation}\label{adda04}
d \E | \tilde\sigma |^2
\lesssim \E | \tilde\sigma_L |^2 d \tau
+ \tilde\lambda_L^2 L^2 d \tau
+ \tilde\lambda_L^2 \E | \tilde\phi_L |^2 d \tau .
\end{equation}
This equation implies the desired growth of $ \E | \tilde\sigma |^2 $. Indeed, due to the estimate on $ \E | \tilde \phi_L |^4 $ and Jensen's inequality we can absorb the last term of \eqref{adda04} in the penultimate one; thus \eqref{adda04} becomes
\begin{align*}
	d \E | \tilde\sigma |^2
	\lesssim \E | \tilde\sigma_L |^2 d \tau +  ( \tilde\lambda_L^2 L^2 + \varepsilon^2 L^2 ) d \tau
	\lesssim \E | \tilde\sigma_L |^2 d \tau + \tilde\lambda_L^2 L^2  d \tau ,
\end{align*}
since by \eqref{lambda} we have $ \varepsilon^2 \ll 1 \leq \tilde\lambda^2 $.
By virtue of \eqref{int01} in Appendix \ref{AppIntegrals} the same argument as for $ \E | \tilde\phi_L |^4 $ implies the claimed estimate on $  \E | \tilde\sigma_L |^2  $. 

\medskip

We now turn to the proof of \eqref{adda04}: By equivalence of norms we may prove \eqref{adda04} by estimating the single components $ d \E | \tilde\sigma^i |^2 $. Applying Itô's formula to the quadratic expression $ | \tilde\sigma_L^i |^2 = { \rm tr } \, ( \tilde\sigma_L^{ i } )^{ * } \tilde\sigma_L^i $, and then appealing to \eqref{sigma-tilde} yields
\begin{align*}
	d  | \tilde\sigma^i |^2 &= 2 { \rm tr } \, ( \tilde\sigma_L^i ) ^* d \tilde\sigma^i +  d [ { \rm tr } \,  ( \tilde\sigma^i )^* \tilde\sigma^i ] \\
	&= { \rm mean ~ zero } \,
	- 2 \tilde\phi_L ^k { \rm tr } \, ( \tilde\sigma_L^i )^* d [ ( \partial_k \phi^{ i } ) \, \Psi ] 
	+ d [ { \rm tr } \, ( \tilde\sigma^i )^* \tilde\sigma^i ] .
\end{align*}
Taking the expectation implies
\begin{align}\label{fosw02prev}
	d \E | \tilde\sigma^i |^2
	= -  2 \E \tilde\phi_L ^k { \rm tr } \, ( \tilde\sigma_L^i )^*  d [ ( \partial_k \phi^{ i } ) \, \Psi ] 
	+ \E d [ { \rm tr } \, ( \tilde\sigma^i )^* \tilde\sigma^i ].
\end{align}
Since constants in \eqref{adda04} we estimate all cross terms by using Kunita-Watanabe inequality, see e.~g.~\cite[Proposition 4.18]{LeGall}, and Young's inequality. On the first term in \eqref{fosw02prev} we obtain
\begin{align*} 
	- 2 \E \tilde\phi_L ^k { \rm tr } \, ( \tilde\sigma_L^i )^*  d [ ( \partial_k \phi^{ i } ) \, \Psi ] 
	\lesssim \E | \tilde\phi_L |^2 d [  { \rm tr } \, \Psi^* \Psi ] +  \E | \tilde\sigma_L |^2  d [ { \rm tr } \, ( \nabla \phi )^* \nabla \phi ]  ,
\end{align*}
while on the second term (using \eqref{sigma-tilde}) in \eqref{fosw02prev} we have
\begin{align*}
\E d [ { \rm tr } \,  ( \tilde\sigma^i )^* \tilde\sigma^i ]
&\lesssim d  [ { \rm tr } \, ( \sigma^i )^*  \sigma^i  ] 
+ \E | \tilde\phi_L |^2 d [ { \rm tr } \,  \Psi^* \Psi ]
+ \E  | \tilde\sigma_L |^2  d [ { \rm tr } \,  ( \nabla \phi )^* \nabla \phi  ] .
\end{align*}
Inserting the last two estimates into \eqref{fosw02prev} yields
\begin{align*}
d \E | \tilde\sigma^i |^2
&\lesssim d [ { \rm tr } \, ( \sigma^i )^*  \sigma^i  ] 
+ \E | \tilde\phi_L |^2 d [ { \rm tr } \, \Psi^* \Psi ]
+ \E  | \tilde\sigma_L |^2 d [  { \rm tr } \, ( \nabla \phi )^* \nabla \phi  ] .
\end{align*}
Below we will argue that the quadratic variation for $ \sigma $ is estimated by
\begin{equation}\label{fosw02}
	d [ { \rm tr } ( \sigma^i )^* \sigma^i ] \lesssim  \tilde\lambda_L^2 L^2 d \tau ,
\end{equation} 
which, in combination with \eqref{wr65bis} and \eqref{estPsi2}, implies \eqref{adda04}.

\medskip

Here comes the argument for \eqref{fosw02}: Observe that the Coulomb gauge \eqref{tao17gauge} for $ \sigma $ has the same structure as \eqref{gaugePsi}. Thus, arguing as for \eqref{estPsi2}, we obtain
\begin{align*}
	d [  { \rm tr } \,  ( \sigma^i )^* \, \sigma^i ]  
	\lesssim L^2  d [  { \rm tr } \, a^* a ]
	\stackrel{ (\ref{defn-a}) }{ = }  L^2  d [  { \rm tr } \, \Psi^* \Psi ]
	\stackrel{ \eqref{estPsi2} }{  \lesssim } \tilde\lambda_L^2 L^2 d \tau ,
\end{align*}
as claimed.
\end{proof}

Equipped with the preceding lemmas, we are now in the position to give a proof of Proposition \ref{prop01}.

\begin{proof}[Proof of Proposition \ref{prop01}]
	Before we apply Lemma \ref{lem02}, we split the homogenization error by using the triangle and Cauchy-Schwarz inequalities
	\begin{align*}
		\E | \tilde\phi_L \otimes a_L + \tilde\sigma_L |^2
		\lesssim \E | \tilde\phi_L |^2 + \E |  \tilde\phi_L \otimes \Psi_L |^2  + \E | \tilde\sigma_L |^2
		\lesssim ( 1 + \E^{ \frac{1}{2} } | \Psi_L |^4 ) \E^{ \frac{1}{2} } | \tilde\phi_L |^4 + \E | \tilde\sigma_L |^2 .
	\end{align*}
	The potential $ \Psi_L $ may be estimated as follows. By Gaussianity, we have $ \E^{ \frac{1}{2} } | \Psi_L |^4 \lesssim \E | \Psi_L |^2 $ by \eqref{estPsi} in combination with \eqref{lambda}, we learn that $ \E | \Psi_L |^2 \lesssim \tilde\lambda_L^2 $. Inserting this estimate into \eqref{add07} yields in combination with Lemma \ref{lem02} the differential inequality
	\begin{align*}
		d\mathbb{E} | \tilde f |^2 - \mathbb{E} | \tilde f_L |^2 d \tau
		\lesssim \varepsilon^2 \Big( \frac{ L^2 }{ \tilde\lambda_L^2 } ( 1 + \tilde\lambda_L ^2 ) + L^2  \Big)  \frac{ d \tau }{ L^2 } 
		+ \varepsilon^2 L^2  \frac{ d \tau }{ L^2}
		\lesssim \varepsilon^2 d \tau ,
	\end{align*}
	which may be recast as
	\begin{align*}
		e^{ \tau } \frac{ d }{ d \tau  } \frac{ \E | \tilde f_L |^2 }{ e^{ \tau } } \lesssim \varepsilon^2 .
	\end{align*}
	The claim follows from integrating this equation w.~r.~t.~$ \tau $, since $ e^{ - \tau } $ is integrable.
\end{proof}

 
 \section{Geometric Brownian Motion on $\textbf{SL}(n)$} \label{S: diffusion}
 
 \subsection{Properties of the Process} \label{S: diffusion definition} Recall from the introduction that we study what we call, by analogy with the scalar case, geometric Brownian motion on $\textbf{SL}(n)$, specifically the process $F$ determined by the SDE
 	\begin{equation} \label{E: sde again}
		dF = F_{\tau} \circ \! dB.
	\end{equation}
Here $B$ is a certain Brownian motion on the Lie algebra $\mathfrak{sl}(n)$ (interpreted as a subspace of the space of endomorphisms of cotangent space), which is characterized via the properties of $F$ in Lemma \ref{L: properties of F}, proved below in Section \ref{S: characterization}.

\medskip

It is possible to show that $F$ is a bonafide Brownian motion on $\textbf{SL}(n)$, see \cite[Theorem 7.2.2]{LeJan} \& \cite[Corollary 1]{Urakawa} for the details.  As a Markov process, its generator is equal to the Laplace-Beltrami operator associated to the left-invariant metric on $\textbf{SL}(n)$ induced by the covariance of $B$, though we will not need this fact in what follows.

\medskip

In fact, the Brownian character of $F$ is also encapsulated in the following multiplicative version of independent increments, which follows readily from the corresponding property of $B$: If $F_{\tau_{*},\tau}$ is defined as in the introduction (see \eqref{ao29}), then, for any $\tau_{*} < \tau_{**} < \tau$, we have $ F_{\tau_{*},\tau} = F_{\tau_{*},\tau_{**}} F_{\tau_{**},\tau} $ and the increments $F_{\tau_{*},\tau_{**}}$ and $F_{\tau_{**},\tau}$ are independent of each other.  This shows that $F$ is analogous to a product of i.i.d.\ random matrices, or, more generally, the matrix cocycles studied in ergodic theory, as in the seminal work of Furstenberg and Kesten \cite{furstenberg-kesten}.  Indeed, the main results of multiplicative ergodic theory apply to $F$, particularly Oseledets' Theorem \cite{Oseledets}.  Although we will not need this theorem later, we state it next due to its intrinsic interest.



	\begin{theorem}[\cite{Oseledets}]\label{Oseledets} There is a measurable flag variety $V_{1,\tau} \supsetneq V_{2,\tau} \supsetneq \cdots \supsetneq V_{n,\tau} \supsetneq V_{n + 1,\tau} = \{0\}$ consisting of linear subspaces $V_{1,\tau},\dots,V_{n,\tau}$ of tangent space and a deterministic sequence $\lambda_{1} > \lambda_{2} > \dots > \lambda_{n}$ of Lyapunov exponents such that, for any $\tau_{*} < \tau$ and $j \in \{1,2,\dots,n\}$,
		\begin{align*}
			F^{\dagger}_{\tau_{*},\tau} V_{j,\tau_{*}} = V_{j,\tau}
		\end{align*}
	and, for any $v \in V_{j,\tau_{*}} \setminus V_{j + 1,\tau_{*}}$,
		\begin{align*}
			\lim_{\tau \to \infty} \frac{1}{\tau} \ln |F^{\dagger}_{\tau_{*},\tau}v| = \lambda_{j}.
		\end{align*}
	\end{theorem}

When $n = 2$, we show below that $\lambda_{1} = \frac{1}{4} = -\lambda_{2}$; see Remark \ref{R: lyapunov exponent} for the details.

	\begin{proof}[Proof of Theorem \ref{Oseledets}] Consider the shift transformation $\sigma$ sending the Brownian motion $B$ to $\sigma(B)_{\tau} = B_{\tau + 1} - B_{1}$.  This is an ergodic measure-preserving transformation on the canonical Wiener space.  By the discussion proceeding this proof, the process $\{F_{k , k + 1}\}_{n \in \mathbb{N}}$ forms a matrix cocycle, and Oseledets' Theorem applies (see \cite[Theorem 5.7]{Krengel} or \cite[Theorem 3.4.1]{Arnold}).  The fact that nothing is lost in passing from discrete time to continuous time is readily deduced from the continuity and stationarity of the increments $\{F_{\tau_{*},\tau}\}$.  
	
	\medskip
	
	What is remarkable here, and still requires justification, is the fact that there are exactly $n$ Lyapunov exponents.  As explained in \cite[Example 9.3]{bougerol_lacroix}, this follows from the fact that $B$ is a nondegenerate Brownian motion in $\mathfrak{sl}(n)$ (see the proof of Lemma \ref{L: properties of F} below). \end{proof}


\subsection{Characterization of the Brownian motion on the Lie algebra $ \mathfrak{sl}(n) $} \label{S: characterization} We characterize $B$ by first rephrasing the properties of $F$ in infinitesimal form.  Since $B$ is a Brownian motion, it suffices to characterize its statistics at a single time, say, $\tau = 1$, and thus to characterize a certain $\mathfrak{sl}(n)$-valued Gaussian endomorphism $X$.

	\begin{lemma} \label{L: characterize covariance} There is a unique $\mathfrak{sl}(n)$-valued Gaussian random endomorphism $X$ such that
		\begin{itemize}
			\item[(i)] $ O X O^{ - 1 }  =_{\text{law}} X$ for all $O \in \textbf{O}(n)$,
			\item[(ii)] $\mathbb{E}[ X X ] = 0$,
			\item[(iii)] $\mathbb{E}[ X X^{*} ] = \text{id}$. 
		\end{itemize} 
	\end{lemma}
	
	\begin{proof} Since $X$ is trace-free, (i) implies that $\mathbb{E} X$ is trace-free and isotropic, hence it vanishes.  It thus only remains to characterize the covariance of $X$.  Since $X$ takes values in endomorphisms of cotangent space, its covariance is a bilinear form on the dual space, namely, the endomorphisms of tangent space.  Let $\bullet$ be this covariance tensor, that is, for any two endomorphisms of tangent space $G,\tilde{G}$, we have
			\begin{equation*}
				\mathbb{E} (G.X) (\tilde{G}.X) = G \bullet \tilde{G}.
			\end{equation*}
		Assumption (i) implies that $\bullet$ is rotationally invariant:
			\begin{equation} \label{E: invariance}
				O G O^{-1} \bullet O \tilde{G} O^{-1} = G \bullet \tilde{G} \quad \text{for any} \, \, O \in \textbf{O}(n), \, \, G,\tilde{G} \in {\rm End}({\rm tangent ~ space}).
			\end{equation}
        Thus, we can appeal to the representation theory of $\mathbf{O}(n)$, as in Appendix \ref{reptheory}.

        \medskip
			
		 Since $X$ takes values in $\mathfrak{sl}(n)$, the identity $\text{id}$ is in the kernel of $\bullet$, that is, $G \bullet \text{id} = 0$ for any $G$.  Therefore, Lemma \ref{lemRep03} (i) in Appendix \ref{reptheory}, \eqref{E: invariance} implies that there are symmetric bilinear forms $\bullet_{\text{sym}}$ and $\bullet_{\text{skew}}$, respectively, such that 
		\begin{equation} \label{E: cov structure}
			G \bullet \tilde{G} = G_{\text{sym}} \bullet_{\text{sym}} \tilde{G}_{\text{sym}} + G_{\text{skew}} \bullet_{\text{skew}} \tilde{G}_{\text{skew}} \quad \text{for any} \, \, G,\tilde{G} \in \text{End}(\text{tangent ~ space}),
			\end{equation}
        where $G_{\text{sym}} = \frac{1}{2} ( G + G^{*} ) - n^{-1} ( {\rm tr} \, G ) \text{id}$ and $G_{\text{skew}} = \frac{1}{2} ( G - G^{*} )$.
		
		\medskip
		
	The factorization \eqref{E: cov structure} implies that $X$ can be decomposed in the form $X = X_{\text{sym}} + X_{\text{skew}}$, where $(X_{\text{sym}},X_{\text{skew}})$ is a pair of independent Gaussian endomorphisms, $X_{\text{sym}}$ taking values in the subspace $\text{Sym}_{0}$ of trace-free, symmetric endomorphisms with covariance $\bullet_{\text{sym}}$ and $X_{\text{skew}}$ taking values in the subspace $\text{Skew}$ of skew-symmetric endomorphisms with covariance $\bullet_{\text{skew}}$.  In particular, because of independence, assumptions (ii) and (iii) become
		\begin{equation*}
			0 = \mathbb{E}[ X^{2}_{\text{sym}}] + \mathbb{E}[ X^{2}_{\text{skew}} ], \quad \text{id} = \mathbb{E}[ X^{2}_{\text{sym}} ] - \mathbb{E}[ X^{2}_{\text{skew}} ],
		\end{equation*}
	and, therefore,
		\begin{equation} \label{E: determining}
			\mathbb{E}[ X_{\text{sym}}^{2} ] = - \mathbb{E}[ X_{\text{skew}}^{2} ] = \frac{1}{2}  \, \text{id}.
		\end{equation}
	
	\medskip
	
	Finally, note that $\bullet_{\text{sym}}$ and $\bullet_{\text{skew}}$ are both invariant under the action of $\textbf{O}(n)$ by conjugation.  That is, \eqref{E: invariance} still holds if we replace $\bullet$ by $\bullet_{\text{sym}}$ and $\bullet_{\text{skew}}$.  According to representation theory for the group $\textbf{O}(n)$, see Lemma \ref{lemRep03} (ii) \& (iii) in Appendix \ref{reptheory}, the symmetric bilinear forms on $\text{Sym}_{0}$ (respectively, $\text{Skew}$) invariant under this action form a one-dimensional linear space.  Thus, equation \eqref{E: determining}, being a nontrivial linear equation in $\bullet_{\text{sym}}$, uniquely determines $\bullet_{\text{sym}}$.  By the same logic, \eqref{E: determining} also completely determines $\bullet_{\text{skew}}$.
    
    \medskip
    
    Concerning existence, fix positive-definite, $\mathbf{O}(n)$-invariant bilinear forms $\bullet_{\text{sym}}$ and $\bullet_{\text{skew}}$, defined respectively on the spaces of trace-free symmetric and skew-symmetric endomorphisms (e.g., the restriction of the Frobenius inner product $(G,G') \mapsto {\rm tr} \, G^{*} G'$ to the respective spaces).  By $\mathbf{O}(n)$-invariance, if $X_{\text{sym}}$ and $X_{\text{skew}}$ are centered Gaussian random variables with covariances $\bullet_{\text{sym}}$ and $\bullet_{\text{skew}}$, respectively, then $\mathbb{E} X^{2}_{\text{sym}} = c_{\text{sym}} \text{id}$ and $-\mathbb{E} X^{2}_{\text{skew}} = c_{\text{skew}} \text{id}$ for some constants $c_{\text{sym}}, c_{\text{skew}} > 0$.  Up to multiplication by constants, there is no loss of generality assuming that $c_{\text{sym}} = c_{\text{skew}} = \frac{1}{2}$.  The desired random variable $X$ can then be taken to be the sum $X = X_{\text{sym}} + X_{\text{skew}}$ of independent realizations of $X_{\text{sym}}$ and $X_{\text{skew}}$, or, equivalently, by defining $X$ to be the centered Gaussian with covariance determined by \eqref{E: cov structure}. \end{proof}
	
Next, we use Lemma \ref{L: characterize covariance} and the construction in Lemma 7  to determine the Brownian motion $B$ driving the SDE $dF = F \circ \! dB$.

	\begin{proof}[Proof of Lemma \ref{L: properties of F}] Since (ii) implies that $F$ is a martingale, which implies that $\mathbb{E} [F F^{*}]_{\tau} = \mathbb{E} F_{\tau} F_{\tau}^{*}$, we can reformulate conditions (i)-(iii) in the following equivalent way:
    \begin{itemize}
        \item[(i)] $O F O^{-1} \overset{\text{law}}{=} F$ for any $O \in \textbf{O}(n)$.
        \item[(ii)] $F \circ \! dB = F dB$.
        \item[(iii)] $\frac{d}{d \tau} \mathbb{E}[F F^{*}]_{\tau} = \mathbb{E} F_{\tau} F_{\tau}^{*}$.
    \end{itemize}
    By Lemma 7, which itself is a consequence of Lemma 11, there is a unique Brownian motion $B$ on $\mathfrak{sl}(n)$ such that 
		\begin{itemize}
			\item[(i')] $O B O^{-1} \overset{\text{law}}{=} B$ for each $O \in \textbf{O}(n)$,
			\item[(ii')] $\mathbb{E} [ B_{\tau}^{2} ] = 0$, 
			\item[(iii')] $\mathbb{E} [ B_{\tau} B_{\tau}^{*} ] = \tau \, \text{id}$.
		\end{itemize}
	We will prove that each of the properties (i)--(iii) above are equivalent to the corresponding property of $B$ in (i')--(iii').
	
	\medskip
	
	We start with the equivalence of (i) and (i').  Given an $O \in \textbf{O}(n)$, observe that 
	\begin{align} \label{E: rot inv}
		d(O F O^{-1}) &= O dF O^{-1} = O (F \circ \! dB) O^{-1} = (O F O^{-1}) \circ \! d(O B O^{-1}).
	\end{align}
Thus, if $O B O^{-1} =_{\text{law}} B$, then $O F O^{-1} =_{\text{law}} F$, and hence (i') implies (i).

\medskip

Next, we prove the converse. To this end, we apply Itô's product rule to the bilinear map $ F \mapsto F \otimes F $: To identify the quadratic variation, let us note that for endomorphisms $ F $, $ B $, we have $ ( F \otimes F ) ( B \otimes B ) = F B \otimes F B $.   Since the martingale part of $dF$ equals $F_{\tau} dB$, we use this identity in conjunction with the Riemann sum approximation of $[F \otimes F]$ to find
\begin{align} \label{E: tensorial approach}
        d [ F \otimes F ]
        = ( F_{\tau} \otimes F_{\tau} ) d [ B \otimes B ] ,
\end{align}
which, since $ F_{ \tau = 0 } \otimes F_{ \tau = 0 } $ is equal to the identity and, by the properties of Brownian motion,
	\begin{equation} \label{E: covariance trick brownian motion}
		d [ B \otimes B ] = \mathbb{E} B_{\tau = 1} \otimes B_{\tau = 1} d \tau,
	\end{equation}
yields
\begin{align*}
        \frac{ d }{ d \tau } \Big|_{ \tau = 0 } \E F_{\tau} \otimes F_{\tau}
        = \E B_{ \tau = 1 }  \otimes B_{ \tau = 1 } .
\end{align*}
On the other hand, in view of \eqref{E: rot inv}, we can apply the same argument to $ O F O^{ -1 } $. Since $ d [ O F O^{ - 1} \otimes O F O^{ - 1} ] = ( O F_{\tau} O^{-1} \otimes O F_{\tau} O^{-1}) d [ O B O^{ - 1 } \otimes O B O^{ - 1 } ] $, this yields
\begin{align*}
        \frac{ d }{ d \tau } \Big|_{ \tau = 0 } \E O F_{\tau} O^{ -1 }  \otimes O F_{\tau} O^{ - 1 }
        = \mathbb{E} O B_{ \tau = 1 }  O^{ - 1 } \otimes O B_{ \tau = 1 } O^{ - 1 } .
\end{align*}
If $ O F O^{-1} =_{\text{law}} F $, then the last two displays are equal; thus showing that the covariance of $ B_{\tau = 1} $ is invariant under the action of conjugation by $O$.  By \eqref{E: covariance trick brownian motion} and the independent increments property of Brownian motion, it follows that the full covariance of $B$ is invariant under conjugation by $O$.  Therefore, (iii) implies (iii').
\medskip

Next, we prove the equivalence of (ii) and (ii').  Regarding (ii), from the asymptotic midpoint formula for the Stratonovich integral, we have $F_{\tau} \circ \! dB = F_{\tau} dB + \frac{1}{2} d[F B]$.  The martingale part of $dF$ is then certainly $F_{\tau}dB$, hence $d[FB] = F_{\tau} d[B B]$ by a noncommutative extension of standard manipulations.  We therefore arrive at
			\begin{equation*}
				F_{\tau} \circ \! dB = F_{\tau} dB + \frac{1}{2} F_{\tau} d [ B B ].
			\end{equation*}
At the same time, \eqref{E: covariance trick brownian motion} implies $d[ B B ] = \mathbb{E} B_{\tau =1}^{2}  d\tau$ and thus, $F_{\tau} \circ \! dB = F_{\tau} dB$ if and only if $\mathbb{E} B_{\tau = 1}^{2} = 0$.  Therefore, by another application of \eqref{E: covariance trick brownian motion}, (ii) and (ii') are equivalent.

\medskip

Next, we turn to property (iii).  
Applying the contraction $F \otimes \tilde{F} \mapsto F \tilde{F}^{*}$ for endomorphisms $F$ and $\tilde{F}$ to \eqref{E: tensorial approach}, 
we find
$d[F F^{*}] = F_{\tau} d [B B^{*}] F_{\tau}^{*}$,
which yields, by \eqref{E: covariance trick brownian motion},
	\begin{equation} \label{E: need this quadratic variation}
		d [F F^{*}] = F_{\tau} \left( \mathbb{E} B_{\tau = 1} B^{*}_{\tau = 1} \right) F_{\tau}^{*} d \tau.
	\end{equation}  Hence (iii') implies (iii).

\medskip

Conversely, if (iii) holds, then we invoke \eqref{E: need this quadratic variation} to find
	\begin{align*}
		\text{id} = \mathbb{E} F_{\tau=0} F_{\tau=0}^{*} = \frac{d}{d\tau} \Big|_{\tau = 0} \mathbb{E} [F F^{*}]_{\tau} = \mathbb{E} F_{\tau =0} \left( \mathbb{E} B_{\tau= 1} B_{\tau = 1}^{*} \right) F_{\tau =0}^{*} = \mathbb{E} B_{\tau = 1} B^{*}_{\tau = 1},
	\end{align*}
which gives (iii').
\end{proof}

\begin{remark} From the properties of $B_{\tau = 1}$ established in the proof of Lemma \ref{L: characterize covariance}, we deduce that $B$ can be decomposed in the form $B = B_{\text{sym}} + B_{\text{skew}}$, where the symmetric part $B_{\text{sym}}$ and the skew part $B_{\text{skew}}$ are independent Brownian motions, which are $\textbf{O}(n)$-invariant in law and determined by the normalization $\mathbb{E} B_{\text{sym},\tau}^{2} = - \mathbb{E} B_{\text{skew},\tau}^{2} = \frac{1}{2} \tau \text{id}$. \end{remark}


\subsection{The Symmetric Part in $n = 2$} We close this section with a characterization of the symmetric part of the covariance of $B$ in dimension $n = 2$, which is an essential ingredient in our analysis of intermittency in the next section.

	\begin{lemma} \label{L: symmetric part} If $B$ is the Brownian motion of Lemma \ref{L: properties of F} and $n = 2$, then, for any symmetric endomorphism $G \in \text{End}(\text{tangent ~ space})$, 
		\begin{equation*}
			\mathbb{E} (G.B_{\tau})^{2}  = \frac{1}{4} \tau \big( ( { \rm tr } \, G)^{2} - 4 \det \, G \big).
		\end{equation*}
	\end{lemma}
	
	\begin{proof} Since $B$ is a Brownian motion, we only need to check the case $\tau = 1$. First, we recall that the covariance of the symmetric part of $B$ is invariant under the action of $\textbf{O}(2)$.  Therefore, since the space of symmetric $\textbf{O}(2)$-invariant bilinear forms on the trace-free symmetric endomorphisms $\text{Sym}_{0}$ is one-dimensional (Lemma \ref{lemRep03} (ii) in Appendix \ref{reptheory}), there is a constant $\lambda > 0$ such that, for any $G_{0}, \tilde{G}_{0} \in \text{Sym}_{0}$, 
		\begin{equation} \label{E: def of lambda}
			\mathbb{E} (G_{0}.B_{\tau = 1}) (\tilde{G}_{0}.B_{\tau = 1})  = \lambda { \rm tr } \, G_{0} \tilde{G}_{0} .
		\end{equation}
	If instead $G, \tilde{G} \in \text{Sym}$ (possibly with nonzero trace), then 
            \begin{align*}
                \mathbb{E} (G.B_{\tau = 1}) (\tilde{G}.B_{\tau = 1})  &= \mathbb{E}  ( G_{\text{sym}}.B_{\tau = 1} ) ( G_{\text{sym}}.B_{\tau = 1} )  = \lambda {\rm tr} \, G_{\text{sym}} \tilde{G}_{\text{sym}},
            \end{align*} 
        where as earlier we write $G_{\text{sym}} = G - \frac{1}{2} ( {\rm tr} \, G ) \text{id}$.
	
	\medskip
	
Next, we compute $\lambda$ using the identity \eqref{E: determining}.  Namely, if $B_{\text{sym}}$ is the symmetric part of $B$, then taking the trace of \eqref{E: determining} yields
	\begin{equation*}
		\mathbb{E}[ |B_{\text{sym},\tau = 1}|^{2}] = 1.
	\end{equation*}
On the other hand, after summing \eqref{E: def of lambda} over an orthonormal basis of $\text{Sym}_{0}$, we find
	\begin{equation*}
		\mathbb{E}[ |B_{\text{sym},\tau = 1}|^{2} ] =  2 \lambda, \quad \text{hence} \, \, \lambda = \frac{1}{2}.
	\end{equation*}
	
	Next, we can define a symmetric bilinear form $\tilde{\bullet}$ on $\text{Sym}$ by the rule
		\begin{equation*}
			G \tilde{\bullet} \tilde{G} = ({\rm tr} \, G)({\rm tr} \, \tilde{G}) - 2 \det ( e^{1} \otimes G e_{1} + e^{2} \otimes \tilde{G} e_{2}) - 2 \det (e^{1} \otimes \tilde{G}e_{1} + e^{2} \otimes G e_{2}).
		\end{equation*}
	Note that the associated quadratic form is $G \tilde{\bullet} G = (\text{tr} \, G)^{2} - 4 \det \, G$.  Further, $\tilde{\bullet}$ is invariant under the action of $\textbf{O}(2)$ on $\text{Sym}$.  Thus, there is a symmetric linear map $\mathbf{T}$ on $\text{Sym}$ such that $O ( \mathbf{T} G ) O^{-1} = \mathbf{T} ( O G O^{-1} )$ and $G \tilde{\bullet} G' = {\rm tr} \, G' (\mathbf{T} G)^{*}$ for each $G,G' \in \text{Sym}$ and $O \in \textbf{O}(n)$. Applying Lemma \ref{L: symmetric irreducibility} to $\mathbf{T}$, we deduce that there is a $\kappa \in \mathbb{R}$ and a $\mu \in \mathbb{R}$ such that
		\begin{equation*}
			G \tilde{\bullet} \tilde{G} = \kappa {\rm tr} \, G \, {\rm tr} \, \tilde{G} + \mu {\rm tr} \, G_{\text{sym}} \tilde{G}_{\text{sym}} \quad \text{for any} \, \, G,\tilde{G} \in \text{Sym}.
		\end{equation*}
	Substituting $G = \tilde{G} = \text{id}$ yields $\kappa = 0$, and the choice $G = \tilde{G} = e^{1} \otimes e_{1} - e^{2} \otimes e_{2}$ yields $\mu = 2$.  Therefore,
		\begin{equation*}
			\frac{1}{4} G \tilde{\bullet} G = \frac{1}{2} {\rm tr} \, G_{\text{sym}} G_{\text{sym}} = \mathbb{E} (G.B_{\tau = 1})^{2}. 
		\end{equation*}
	\end{proof}
 
\section{Intermittent behavior of the Bessel-type process $R$ in $n = 2$} \label{S: intermittency}

The remainder of our study of the geometric Brownian motion $F$ focuses on intermittency.  Here we exploit structure specific to dimension $n = 2$.  In particular, we will deduce information about $F$ by studying the process $R$ given by $R_{\tau} = \frac{1}{2} |F_{\tau}|^{2}$, where we recall $|\cdot|$ is the Frobenius norm.

\medskip

The analysis begins with the observation that $R$ satisfies an SDE in its own right.

	\begin{lemma} \label{L: R SDE and lyapunov exponent} If $R$ is the process $R_{\tau} = \frac{1}{2} |F_{\tau}|^{2}$, then $\mathbb{P}\{ R_{\tau} > 1 \, \, \text{for each} \, \, \tau > 0\} = 1$ and there is a one-dimensional Brownian motion $w$ adapted to $B$ such that 
		\begin{gather} 
			dR= \frac{1}{2} R_{\tau} d\tau + \sqrt{R^{2}_{\tau} - 1} \circ \! dw. \label{E: R SDE}
		\end{gather}
	\end{lemma}

\begin{proof}[Sketch of proof] To motivate the appearance of the SDE \eqref{E: R SDE}, notice that since $dF^{*} = dB^{*} F^{*}$ and ${\rm tr} \, F^{*} F dB^{*} = {\rm tr} \, dB F^{*} F = {\rm tr} \, F^{*} F dB$, an application of It\^{o}'s lemma to the product $F F^{*}$ yields
    \begin{equation} \label{E: decomposition}
        dR = {\rm tr} \, F_{\tau}^{*} F_{\tau} dB + \frac{1}{2} {\rm tr} \, F_{\tau} d [ B B^{*} ] F_{\tau}^{*} = {\rm tr} \, F_{\tau}^{*} F_{\tau} dB + R d \tau.
    \end{equation}
At the same time, applying Lemma \ref{L: symmetric part} with $G_{\tau} = F_{\tau} F_{\tau}^{*}$, writing $B'$ for an independent copy of $B$ with law $\mathbb{E}'$, and using the fact that $G$ is adapted to $B$, we see that the martingale $M$ with $dM = {\rm tr} \, F_{\tau}^{*} F_{\tau} dB = G_{\tau}^{\dagger}.dB$ has quadratic variation 
    \begin{align} \label{E: R quadratic variation}
        d [ M M ] = \mathbb{E}' (G_{\tau}^{\dagger}.B'_{\tau = 1})^{2} d \tau = \frac{1}{4} ( ( {\rm tr} \, G_{\tau} )^{2} - 4 \det G_{\tau} ) d \tau = ( R_{\tau}^{2} - 1 ) d \tau.
     \end{align}
This shows that $R$ solves the following SDE in law:
    \begin{equation} \label{E: R SDE Ito form}
        dR = R_{\tau} d \tau + \sqrt{R_{\tau}^{2} - 1} dw.
    \end{equation}
Since the solution of this SDE has $d[ \sqrt{R^{2} - 1} \, w] = \frac{R_{\tau}}{\sqrt{R_{\tau}^{2} - 1}} d [R w] = R_{\tau} d \tau$, the Stratonovich formulation of the equation is precisely \eqref{E: R SDE}.

\medskip

To fully justify that $R$ solves the SDE above, there is just one technicality, namely, the possibility that $R_{\tau} = 1$.  Indeed, since $F$ takes values in $\textbf{SL}(2)$, $R \geq 1$ holds globally, and $R_{\tau} = 1$ if and only if $F_{\tau} \in \textbf{SO}(2)$.  Yet $\textbf{SO}(2)$ is a compact, codimension-two submanifold of $\textbf{SL}(2)$, thus of vanishing capacity, hence the diffusion $F$ cannot hit it.  In Appendix \ref{A: stopping}, we give a more elementary, but less transparent, proof based on the Feller test, along with the argument for the existence of the Brownian motion $w$ in Lemma \ref{L: R SDE and lyapunov exponent}. \end{proof}

It turns out that the SDE solved by $R$ can be reformulated in terms of the largest eigenvalue $S$ of $F^{*} F$.  From the definition of $R$ and the fact that $\det F = 1$, the two processes are related by the equation
	\begin{equation} \label{E: S equation}
		R_{\tau} = \frac{1}{2} ( S_{\tau} + S_{\tau}^{-1} ).
	\end{equation}
This equation has the explicit solution $S = \exp \text{arccosh} \, R$, from which we deduce that
	\begin{equation} \label{E: eigenvalue eqn}
		dS = \frac{1}{2} S_{\tau} \frac{ S_{\tau}^{2} + 1 }{ S_{\tau}^{2} - 1 } d \tau + S_{\tau} \circ \! dw.
	\end{equation}
Here there is no ambiguity because $R$ (hence also $S$) never hits one.

\medskip

The SDE \eqref{E: eigenvalue eqn} is reminiscent of one-dimensional geometric Brownian motion.  Specifically, the process $Q$ defined by
\begin{equation} \label{E: Q formula}
Q_{\tau} = e^{  \frac{1}{2} \tau + w_{\tau} }
\end{equation}
solves the SDE
	\begin{align} \label{E: geometric brownian motion}
		dQ = \frac{1}{2} Q_{\tau} d \tau + Q_{\tau} \circ \! dw.
	\end{align}
We will see that the process $R$ is qualitatively similar to $Q$.  For instance, since $\frac{S^{2} + 1}{S^{2} - 1} > 1$ and $R \geq 2^{-1} S$ by definition, the comparison principle for Stratonovich SDEs implies that 
	\begin{equation} \label{E: compare with geometric BM}
		\mathbb{P} \{ 2 R_{\tau} \geq S_{\tau} \geq Q_{\tau} \, \, \text{for each} \, \, \tau \geq 0\} = 1.
	\end{equation}
Next, we show that a matching upper bound holds at the level of moments.

\begin{proof}[Proof of Lemma \ref{R: intermittent growth of F}] Applying It\^{o}'s formula to the SDE \eqref{E: R SDE Ito form}, we find, for any $p \geq 1$,
			\begin{align*}
				d\mathbb{E} R^{p}
				= \frac{1}{2} p(p + 1) \mathbb{E} R^{p}_{\tau} d \tau - \frac{1}{2} p (p - 1) \mathbb{E} R_{\tau}^{p - 2} d \tau
				\leq \frac{1}{2} p(p + 1) \mathbb{E} R_{\tau}^{p} d \tau.
			\end{align*}
		By the comparison principle, the inequality $\mathbb{E} R_{\tau}^{p} \leq e^{  \frac{1}{2} p (p + 1) \tau } $ follows.
		
	\medskip
	
		At the same time, from the formula \eqref{E: Q formula} and the classical formula for the moment generating function of a Gaussian, we know that $\mathbb{E} Q_{\tau}^{p} = e^{  \frac{1}{2} p (p + 1) \tau } $.  Thus, by \eqref{E: compare with geometric BM}, we have
			\begin{equation*}
				\mathbb{E} R_{\tau}^{p} \geq 2^{-p} e^{  \frac{1}{2} p(p + 1) \tau } .
			\end{equation*} 
	\end{proof}
	
The lower bound \eqref{E: compare with geometric BM} can also be used to compute the top Lyapunov exponent of $F$, as explained in the next remark.

\begin{remark} \label{R: lyapunov exponent} In the limit $\tau \to \infty$, we have
    \begin{equation} \label{E: top lyapunov exponent via R}
        \lim_{\tau \to \infty} \frac{1}{\tau} \ln R_{\tau} = \frac{1}{2} \quad \text{almost surely.}
    \end{equation}
In particular, the top Lyapunov exponent of $F$ equals $\frac{1}{4}$. \end{remark}

\begin{proof}[Proof of \eqref{E: top lyapunov exponent via R}] First, notice that \eqref{E: compare with geometric BM} implies that  $S_{\tau} \to \infty$ as $\tau \to \infty$ almost surely.  From this, we can use the SDEs satisfied by $\ln S$ and $\ln Q$ (readily computed from \eqref{E: eigenvalue eqn} and \eqref{E: geometric brownian motion}, respectively) to deduce that
	\begin{equation*}
		\lim_{\tau \to \infty} \frac{1}{\tau} ( \ln S_{\tau} - \ln Q_{\tau} ) = \lim_{\tau \to \infty} \frac{1}{2 \tau} \int_{0}^{\tau} d \tau' \big( \frac{ S^{2}_{\tau'} + 1 }{ S_{\tau'}^{2} - 1} - 1 \big)  = 0.
	\end{equation*}
Thus, since $R_{\tau} = \frac{1}{2} ( S_{\tau} + S_{\tau}^{-1} )$, we conclude
	\begin{equation*}
		\lim_{\tau \to \infty} \frac{1}{\tau} \ln R_{\tau} \overset{\eqref{E: S equation}}{=} \lim_{\tau \to \infty} \frac{1}{\tau} \ln S_{\tau} = \lim_{\tau \to \infty} \frac{1}{\tau} \ln Q_{\tau} \overset{\eqref{E: Q formula}}{=} \frac{1}{2} \quad \text{almost surely.}
	\end{equation*}
\end{proof}

We conclude this section with the proof of Lemma \ref{lem-weakinter}.

\begin{proof}[Proof of Lemma \ref{lem-weakinter}] We claim that it suffices to observe that 
		\begin{equation} \label{E: scaling identity for geometric BM}
			\mathbb{E} Q_{\tau} I(Q_{\tau} \geq \mathbb{E}^{\frac{3}{2}} Q_{\tau} ) = \frac{1}{2} \mathbb{E} Q_{\tau} .
		\end{equation}
	Indeed, since $R_{\tau} \geq \frac{1}{2} Q_{\tau}$, once we have established \eqref{E: scaling identity for geometric BM}, we can use the identity $\mathbb{E} R_{\tau} = e^{\tau} = \mathbb{E} Q_{\tau} $ (which follows directly from \eqref{E: R SDE Ito form}) to find the desired estimate
		\begin{align*}
			\mathbb{E} R_{\tau} I(R_{\tau} \geq \frac{1}{2} \mathbb{E}^{\frac{3}{2}} R_{\tau} )  \geq \frac{1}{4} \mathbb{E} R_{\tau} .
		\end{align*} 
		
	\medskip
		
	It only remains to prove \eqref{E: scaling identity for geometric BM}.  Since $Q_{\tau}$ is given by \eqref{E: Q formula} and $\mathbb{E} Q_{\tau} = e^{\tau}$, this identity is equivalent to
		\begin{equation*}
			e^{-\frac{\tau}{2}} \mathbb{E} e^{w_{\tau}} I(w_{\tau} \geq \tau) = \frac{1}{2}.
		\end{equation*}
	Since $w_{\tau}$ is a Gaussian with variance $\tau$, this in turn is equivalent to the elementary identity $\int_{\tau}^{\infty} dw e^{-\frac{\tau}{2} + w - \frac{w^{2}}{2\tau}} = \frac{1}{2} \int_{0}^{\infty} dw' e^{-\frac{w'^{2}}{2\tau}}$.
	\end{proof}


\section{Proofs of the Main Results}\label{sec:mr}

In this section we connect our results on the geometric Brownian motion $F$ to the intermittent behavior of the Lagrangian coordinates $u$: Theorem \ref{thm:mr01} and its consequence Theorem \ref{thm:mr02}.

\medskip

In line with $ u_L ( x ) = x + \phi_L ( x ) $ we will write 
\begin{equation} \label{E: definition of tilde u}
\tilde u_L ( x ) := x + \tilde\phi_L ( x ).
\end{equation}

In establishing Theorem \ref{thm:mr01}, we use the coupling of $b$ and $B$ provided by Lemma \ref{lemR01} (see also \eqref{E: coupling first instance}), which, by virtue of the equivalence of Itô and Stratonovich interpretation, cf.\ the discussion at the beginning of Subsection \ref{S: rigorous proxy analysis}, turns \eqref{ao01bis} \& \eqref{ao29} into
\begin{equation} \label{E: coupling equation for F}
    dF_{\tau_{*},\cdot} = F_{\tau_{*},\tau} \nabla d \phi(0) \, \, \text{for} \, \, \tau > \tau_{*}, \quad F_{\tau_{*},\tau} = \text{id} \, \, \text{if} \, \, \tau \leq \tau_{*}.
\end{equation}
We stress that the $\tau$ and $L$ variables are related by $\tau = \ln \tilde{\lambda}_{L}$.  The following lemma summarizes the key estimate towards Theorem \ref{thm:mr01}.

\begin{lemma}\label{lem:mr03}
	With $ \tau_* := \tau( | x|^2 ) $, $L_{*} = \sqrt{ |x|^{2} + 1}$, and $ \tilde\lambda_* := \tilde\lambda_{ L_{*} } $, we have, for any $L \geq 1$,
	\begin{align} \label{pmr03}
		\frac{1}{|x|^{2}} \E | \tilde u_L ( x ) - \tilde u_L ( 0 ) - F_{ \tau_*, \tau }^{ \dagger } x  |^2
		\lesssim \varepsilon^2 \max \{ 1 , \frac{ \tilde\lambda_L }{ \tilde\lambda_* } \} .
	\end{align}
\end{lemma}

We will prove Lemma \ref{lem:mr03} in the spirit of Proposition \ref{prop01}: We will derive an SDE for the residuum
\begin{equation*}
	r_{L} := \tilde{u}_{L}(x) - \tilde{u}_{L}(0) - F_{ \tau_{*}, \tau }^{\dagger} x .
\end{equation*}
In line with the definitions of $\tau_{*}$, $L_{*}$, and $
\lambda_{*}$, we will henceforth write
\begin{equation*}
r_{*} := r_{L_{*}}.
\end{equation*}
As for $ \tilde f_L $ in Lemma \ref{lem03}, $ r_L $ will solve the SDE of a stochastic exponential
\begin{align}
	d r &= \nabla d \phi(0)^{\dagger} r_L + d g \quad \text{with initial data} \quad r_{*} \overset{\eqref{E: definition of tilde u},\eqref{E: coupling equation for F}}{=} \tilde{\phi}_{L_{*}}(x) - \tilde{\phi}_{L_{*}}(0)  \label{E: equation of mu}
\end{align}
and forcing term $g$ given by
\begin{equation}
	dg = \nabla ( d \phi(x) - d \phi(0) )^{\dagger} \tilde{\phi}_L (x) + d ( \phi(x) - \phi(0) - \nabla \phi(0)^{\dagger} x ) , \label{E: forcing}
\end{equation}
see Step 2 in the proof below. Hence the same strategy as in Lemma \ref{lem01} can be used to establish estimate \eqref{pmr03}: We rely on the observation that the driver $  \nabla d \phi(0)^{\dagger} r $ in \eqref{E: equation of mu} leads to the differential inequality $ d \E | r |^2 - \E | r_{L} |^2 d \tau \lesssim ... $ in the variable $ \tau = \ln \tilde\lambda $. By this exponential structure, the estimate on the initial data $ \frac{1}{|x|^{2}} \mathbb{E} | r_{*} |^{2} \lesssim \frac{\varepsilon^{2}}{\tilde{\lambda}_{*}^{2}} $ is amplified as stated in \eqref{pmr03}.

\medskip

\begin{proof}[Proof of Lemma \ref{lem:mr03}.]

We will prove the lemma by establishing
\begin{align}\label{pmr03b}
	\frac{1}{|x|^{2}} \E | \tilde u_L ( x ) - \tilde u_L ( 0 ) - F_{ \tau_*, \tau }^{ \dagger } x  |^2
	&\lesssim \left\{ \begin{aligned}
		\frac{ \varepsilon^{2} }{ \tilde{\lambda}^{2} }  & \quad { \rm provided } ~ L \leq L_{*}  , \\
		\frac{ \varepsilon^2 }{ \tilde\lambda_{ * }^2 } \Big( 1 + \frac{ \varepsilon^2 }{ \tilde\lambda_{ * }^2 } \Big) \frac{ \tilde\lambda_L }{ \tilde\lambda_{ * } } & \quad { \rm provided } ~ L_{*} \leq L .
	\end{aligned} \right.
\end{align}
The regime $ L \leq L_{*} $ is the content of Step 1. The rest of the proof is devoted to the opposite regime.

\medskip

\textit{Step 1 (Estimate of $ \mathbb{E} | r |^{2} $ in the regime $ L \leq L_{*} $).}  We claim that 
	\begin{align} \label{initial}
		|x|^{-2} \mathbb{E} | r_{L} |^{2} \lesssim \varepsilon^{2} \tilde{\lambda}^{-2} \quad { \rm provided } \quad L \leq L_{*}.
	\end{align}
In fact, it suffices to prove the apparently weaker estimate:
	\begin{align} \label{init}
		|x|^{-2} \mathbb{E} | r_{L} |^{2} \lesssim \left\{ \begin{array}{r l}
										\varepsilon^{2} \tilde{\lambda}^{-2}, & \text{if} \, \, L \leq \max\{ 2 |x|,L_{*}\}, \\
										\varepsilon^{2} (\ln L_{*} + \tilde{\lambda}_{*} ), & \text{if} \, \, 2 |x| < L \leq L_{*}.
									\end{array} \right.
	\end{align}
To see this, notice that the definition $L_{*}^{2} = |x|^{2} + 1$ implies that $L_{*} \leq 2$ if $2|x| \leq L_{*}$. Since $\varepsilon <1$, this shows that $\tilde{\lambda}_{L} \leq 2$ if $2|x| \leq L \leq L_{*}$, hence \eqref{init} actually implies \eqref{initial}.  

\medskip

Notice that by definition \eqref{E: coupling equation for F} we have
\begin{align*}
	r_L = \tilde\phi_{ L } ( x ) - \tilde\phi_L ( 0 )
	\quad { \rm provided } \quad L \leq L_{*} .
\end{align*}
Thus, consider two cases: If $L \leq 2|x|$, we use the triangle inequality, stationarity, and Lemma \ref{lem02} to obtain
\begin{align*}
	|x|^{-2} \E | r_{L} |^2 
	\leq 2 |x|^{-2} (  \E | \tilde \phi_{L}(x) |^2 + \E | \tilde \phi_{L} ( 0 ) |^2 ) 
	= 4|x|^{-2} \E | \tilde \phi_{L} |^2  \lesssim \varepsilon^{2} \tilde{\lambda}^{-2}.
\end{align*}
On the other hand, if $2 |x| < L$, we Taylor expand the increment to find, again by stationarity,
	\begin{align*}
		|x|^{-2} \E | \tilde r_{*} |^{2} \leq \int_{0}^{1} dt \, \mathbb{E} | \nabla \tilde{\phi}_{L_{*}}(tx) |^{2} = \mathbb{E} | \nabla \tilde{\phi}_{L_{*}} |^{2}.
	\end{align*}
It only remains to observe that $\mathbb{E}|\nabla\tilde{\phi}_{L}|^{2} \lesssim \varepsilon^{2} ( \ln L_{*} + \tilde{\lambda}_{*} )$ if $L \leq L_{*}$.  Indeed, as in \cite[Proof of Lemma 5.3]{CMOW}, we can multiply the identity \eqref{helm-tilde} by $ e^i + \nabla \tilde{\phi}_L^i $ to obtain
	\begin{align*}
		\frac{1}{n} \mathbb{E} | \nabla \tilde{\phi}_{L} |^{2} = \tilde{\lambda}_{L} - 1 + \frac{1}{n} \mathbb{E} \nabla \tilde{\phi}_{L} \cdot \tilde{f}_{L},
	\end{align*}
which implies, by Young's inequality, the estimate of $\tilde{f}_{L}$ from Proposition \ref{prop01}, the definition \eqref{lambda} of $\tilde{\lambda}_{L}$, and the elementary $ \sqrt{ 1 + \varepsilon^2 \ln L } - 1 \leq \frac{ 1 }{ 2 } \varepsilon^2 \ln L $,
	\begin{align*}
		\mathbb{E} | \nabla \tilde{\phi}_{L} |^{2} \lesssim \varepsilon^{2} ( \ln L + \tilde{\lambda}_{L} ) \leq \varepsilon^{2} ( \ln L_{*} + \tilde{\lambda}_{*} ) \quad { \rm provided } \quad L \leq L_{*}.
	\end{align*}

\medskip

\textit{Step 2 (Justification of SDE \eqref{E: equation of mu})} Equation \eqref{E: equation of mu} follows upon recalling the definition \eqref{E: definition of tilde u} of $\tilde{u}_{L}$ and then taking differences in \eqref{phi-tilde}:
\begin{align*}
	& d (\tilde{u}(x) - \tilde{u}(0)) \\ 
	&= d \phi(x) + \tilde{\phi}_L^{i}(x) \partial_{i} d \phi (x) - d \phi ( 0 )- \tilde{\phi}_L^{i} ( 0 ) \partial_{i} d \phi ( 0 ) \\
	&= ( x^i + \tilde{\phi}_L^i(x) - \tilde{\phi}_L^i(0) )  \partial_i d \phi(0)
	+ \tilde\phi_L^i(x) \partial_i ( d \phi(x) - d \phi(0) )
	+ d ( \phi(x) - \phi(0) - x^{i} \partial_{i}\phi(0) ) .
\end{align*}
Regrouping terms and noting that, in general, $f^{i} \partial_{i} g = \nabla g^{\dagger} f$, this says that
\begin{align} \label{E: equation for increment}
	d(\tilde{u}(x) - \tilde{u}(0)) &= \nabla d \phi(0)^{\dagger} ( \tilde{u}_L (x) - \tilde{u}_L (0) )  + d g,
\end{align}
which, in view of the equation \eqref{E: coupling equation for F}, implies \eqref{E: equation of mu}.

\medskip

\textit{Step 3 (Differential inequality for the residuum $ \E | r |^2 $).} We start with the claim
\begin{align}\label{pmr08}
		d\mathbb{E} | r |^{2} - \mathbb{E} | r_{ L } |^{2} d \tau
		&\lesssim \varepsilon |x| \mathbb{E}^{\frac{1}{2}} | r_L |^{2} \frac{ d \tau }{ \tilde\lambda }
		+ |x|^{2} \mathbb{E}^{\frac{1}{2}} | r_L |^{2} \frac{ d \tau }{ L }
		+ \varepsilon^{2} |x|^{2} \frac{ d \tau }{ \tilde{\lambda}^{2} }
		+ |x|^{4} \frac{ d \tau }{ L^{2} } .
	\end{align}
Since by \eqref{E: equation of mu} the residuum $ r $ is a martingale we have $  d \E | r |^2 = d \E [ r \cdot r ] $, so that we may derive \eqref{pmr08} by monitoring $ [ r \cdot r ] $.

\medskip

By using the SDE \eqref{E: equation of mu}, and stationarity of $ \nabla d \phi $, we can write
	\begin{align*}
		d [ r \cdot r] = r_L^{i} r_L^{j} d [ \partial_{i} \phi  \cdot \partial_{j} \phi ]
				+ 2 r_L^i d [ \partial_i \phi ( 0 ) \cdot g ]
				+ d [ g \cdot g ].
	\end{align*}
Upon invoking the second item of \eqref{wr65bis} in Lemma \ref{lemR01} and the definition \eqref{E: forcing} of the forcing $ dg $, this equation becomes
	\begin{align*}
		d [ r \cdot r ]
		= | r_L |^{2} d \tau
		&+2  r_L^i \tilde{\phi}_L^{j}(x) d [ \partial_{i} \phi (0) \cdot (\partial_{j} \phi (x) - \partial_{j} \phi (0) ) ] \\
		&+ 2 r_L^i d [ \partial_i \phi (0) \cdot ( \phi (x) - \phi (0) - \nabla \phi (0)^{\dagger}x ) ] + d [ g \cdot g ] .
	\end{align*}
Taking the estimates
\begin{align}
	| d [\partial_{i} \phi (0) \cdot (\partial_{j} \phi (x) - \partial_{j} \phi (0) )] | &\lesssim | x | \frac{d \tau }{L}, \label{E: covariation estimate 1} \\
	| d [\partial_{i} \phi (0) \cdot ( \phi (x) - \phi (0) - \nabla \phi (0)^{\dagger} x) ] | &\lesssim | x |^{2} \frac{d \tau}{L}, \label{E: covariation estimate 2}
\end{align}
for granted for now (we give the proof in Step 5 below), we bound the first two error terms
	\begin{align*}
		\mathbb{E} r_L^i \tilde{\phi}_L^j(x) d [ \partial_i \phi (0) \cdot (\partial_{j} \phi(x) - \partial_{j} \phi (0) ) ]
		&\lesssim |x| \big( \mathbb{E}| r_L |^{2} \mathbb{E} |\tilde{\phi}_L |^{2} \big)^{\frac{1}{2}} \frac{d \tau }{L}, \\
		\mathbb{E} r_L^i d [ \partial_{i} \phi (0) \cdot ( \phi (x) - \phi (0) - \nabla \phi (0)^{\dagger} x ) ]
		&\lesssim | x |^{2}  \mathbb{E}^{\frac{1}{2}} | r_L |^{2} \frac{d \tau}{L}.
	\end{align*}
The quadratic variation $ d [ g \cdot g ] $ can be estimated by first invoking the definition \eqref{E: forcing} of $ dg $ and then applying the estimates
\begin{align}
	| d [  (\partial_{i} \phi (x) - \partial_i \phi(0) ) \cdot (\partial_{j} \phi (x) - \partial_{j}d\phi (0) )]| &\lesssim | x |^2 \frac{d \tau}{L^2}, \label{E: covariation estimate 1 b}  \\
	d [ (\phi (x) - \phi (0) - \nabla \phi (0)^{\dagger} x) \cdot (\phi (x) - \phi (0) - \nabla \phi (0) ^{  \dagger } x)] &\lesssim | x |^{4} \frac{d \tau}{L^2}, \label{E: covariation estimate 2 b}
\end{align}
which we also also prove in the auxiliary Step 5. These last two bounds immediately yield the estimate
\begin{align*}
		\mathbb{E} d [ g \cdot g ] &\lesssim \mathbb{E}\tilde{\phi}_L^{i}(x) \tilde{\phi}_L^{j}(x) d [ (\partial_i \phi (x) - \partial_{i} \phi (0) ) \cdot (\partial_{j} \phi (x) - \partial_{j} \phi (0) ) ]  \phantom {  \frac{d \tilde{\lambda}^{2}}{L^2 \tilde{\lambda}^{2}} }
\\
			&+ d [( \phi (x) - \phi (0) - \nabla \phi (0)^{\dagger} x ) \cdot ( \phi (x) - \phi (0) -\nabla \phi (0)^{\dagger} x )] \\
			& \lesssim |x|^{2} \mathbb{E} |\tilde{\phi}_L |^{2} \frac{d \tau}{L^{2}}
			+ | x |^{4} \frac{d \tau}{L^2} .
	\end{align*}
Recalling that $ d \E | r |^2 = d \E [ r \cdot r ] $ and putting it all together, we find
	\begin{align*}
		d\mathbb{E} | r |^{2}  - \mathbb{E} | r_L |^{2} d \tau
		&\lesssim |x| \big( \mathbb{E}| r_L |^{2} \mathbb{E}| \tilde{\phi}_L |^{2} \big)^{\frac{1}{2}} \frac{d \tau}{L}
		+  |x|^{2} \mathbb{E}^{\frac{1}{2}} | r_L |^{2} \frac{d \tau}{L}
		+ |x|^{2} \E | \tilde\phi_L |^2 \frac{ d \tau }{ L^2 }
		+ |x|^{4} \frac{ d \tau }{L^{2}}.
	\end{align*}
Appealing again to the bound on $\tilde{\phi}_L $ in Lemma \ref{lem02}, we can rearrange this to obtain \eqref{pmr08}.

\medskip

\textit{Step 4 (Estimate of $ \mathbb{E} | r |^{2} $ in the regime $ L_{*} \leq L $).} Next, we use an ODE argument to infer
\begin{align}\label{pmr09}
	\frac{1}{ | x |^2 } \E | r_{L} |^2
	\lesssim \frac{ \tilde\lambda }{ \tilde\lambda_* } \bigg( \frac{1}{ | x |^2 } \E | r_* |^2
	+ \frac{ \varepsilon^2 } { \tilde\lambda_*^2 } \Big( 1
	+ \frac{ \varepsilon^2  }{ \tilde\lambda_*^{ 2 } } \Big) \bigg) \quad \text{provided} \, \, L_{*} \leq L.
\end{align}
By combining this estimate with \eqref{initial}, we obtain \eqref{pmr03b} in the regime $ | x | \leq L $.

\medskip

Upon dividing the differential inequality \eqref{pmr08} by the integrating factor $ \tilde\lambda_L^{ - 1 } = e^{ -  \tau } $, we obtain
\begin{align*}
\frac{ d }{ d \tau } \frac{ \E | r_L |^2 }{ \tilde\lambda_L }
= \frac{ 1 }{ \tilde\lambda_L } \frac{ d \E | r_L |^2 }{ d \tau } - \frac{ E | r_L |^2 }{ \tilde\lambda_L }
\lesssim \varepsilon \frac{ | x | \E^{ \frac{1}{2} } | r_L |^2 }{ \tilde\lambda_L^2 }
+ \frac{ | x |^2 \E^{ \frac{1}{2} } | r_L |^2 }{ L \tilde\lambda_L }
+ \varepsilon^2 \frac{ | x |^2 }{ \tilde\lambda_L^3 }
+ \frac{ | x |^4 }{ L^2 \tilde\lambda_L } ,
\end{align*}
which implies after integration that
\begin{align*}
	\frac{ \E | r_L |^2 }{ \tilde\lambda_L }  - \frac{ \E | r_* |^2 }{ \tilde\lambda_* }
	& \lesssim \Big( \varepsilon | x |  \int_{ \tau( L_* ) }^{ \tau ( L ) } d\tau \frac{1}{ \tilde\lambda_{ L ( \tau ) } ^{ \frac{3}{2 } } }
	+ \frac{ | x |^2 }{ \tilde\lambda_*^{ \frac{5}{2} } }  \int_{ \tau_* }^{ \tau ( L ) } d\tau \frac{ \tilde\lambda_{ L( \tau ) }^2 }{ L ( \tau )  } \Big) \sup_{ L_* \leq \ell \leq L } \Big( \frac{ \E | r_{ \ell } |^2 }{ \tilde\lambda_{ \ell }  } \Big)^{ \frac{1}{2} } \\
	&+ \varepsilon^2 | x |^2  \int_{ \tau_* }^{ \tau ( L ) } d\tau \frac{1}{ \tilde\lambda_{ L ( \tau ) } ^{ 3 } }
	+ \frac{ | x |^4 }{ \tilde\lambda_*^3 } \int_{ \tau_* }^{ \tau ( L ) } d\tau \frac{ \tilde\lambda_{ L(\tau) } ^2 }{ L( \tau )^2 } .
\end{align*}
To evaluate the integrals, we invoke \eqref{int02} and \eqref{int03} from Appendix \ref{AppIntegrals}
\begin{align*}
\int_{ \tau_* }^{ \tau ( L ) } d\tau\, \frac{ \tilde\lambda_{ L ( \tau ) } ^2 }{L( \tau )^p} \stackrel{ (\ref{int02} ) }{ \lesssim } \frac{ \varepsilon^2 }{ L_*^p }
\quad { \rm and } \quad
\int_{ \tau_* }^{ \tau ( L ) } d\tau \, \frac{1}{ \tilde\lambda_{ L ( \tau ) } ^p } \stackrel{ (\ref{int03} ) }{ \lesssim } \frac{ 1 }{ \tilde\lambda_*^{p} }
\quad { \rm for } ~ p \in \{ 1 , 2 \} ,
\end{align*}
so that we may post-process the above estimate to
\begin{align*}
	\frac{ \E | r_L |^2 }{ \tilde\lambda }  - \frac{ \E | r_* |^2 }{ \tilde\lambda_* }
	& \lesssim \varepsilon \Big( \frac{ | x | }{ \tilde\lambda_*^{ \frac{3}{2 } } }
	+ \varepsilon \frac{ | x |^2 }{ L_* \tilde\lambda_*^{ \frac{5}{2} } }
	\Big) \sup_{ L_* \leq \ell \leq L } \Big( \frac{ \E | r_{ \ell } |^2 }{ \tilde\lambda_{ \ell }  } \Big)^{ \frac{1}{2} }
	+ \varepsilon^2 \Big( \frac{  | x |^2 }{ \tilde\lambda_*^{ 3 } }
	+  \frac{ | x |^4 }{ L_*^2 \tilde\lambda_*^3 } \Big) .
\end{align*}
Using Young's inequality and the fact that $L_{*}^{2} = 1 + |x|^{2} $, we infer from this last estimate that \eqref{pmr09} holds.

\medskip

\textit{Step 5 (Estimates on quadratic variations)}
In this auxiliary step we prove the covariation estimates \eqref{E: covariation estimate 1},  \eqref{E: covariation estimate 1 b}, \eqref{E: covariation estimate 2} and \eqref{E: covariation estimate 2 b} that were needed above. We will first derive \eqref{E: covariation estimate 1 b}. It is then a routine computation that we include for the readers convenience to derive \eqref{E: covariation estimate 1}, \eqref{E: covariation estimate 2} and \eqref{E: covariation estimate 2 b}.

\medskip

\textit{Argument for \eqref{E: covariation estimate 1 b}.} Upon applying Cauchy-Schwarz-inequality, we may restrict to the case $ i = j $. To compute the quadratic variation, we first exploit the independence of the increments $ \nabla d \phi $, and then appeal to the mean-value inequality:
	\begin{align*}
		& \Big[ \int_{ \tau_- }^{ \tau_+  } (\partial_{i} d \phi (x) - \partial_{i} d \phi (0) ) \cdot \int_{ \tau_- }^{ \tau_+  } (\partial_{i} d \phi (x) - \partial_{i} d \phi ( 0 ) ) \Big] \\
		& = \mathbb{E} \Big| \int_{ \tau_- }^{ \tau_+  } (\partial_{i} d \phi (x) - \partial_{i} d \phi (0) ) \Big|^2 
		= \mathbb{E}  \Big| \int_{0}^{1} ds  \int_{ \tau_- }^{ \tau_+  } \nabla \partial_{i} d \phi (s x)^{\dagger} x \Big|^{2}
		\leq |x|^{2} \mathbb{E} \Big|  \int_{ \tau_- }^{ \tau_+  } \nabla \partial_{i} d \phi  \Big|^{2} .
	\end{align*}
The last integral we post-process using \eqref{dertrick2}, which yields
\begin{align*}
	|x|^{2} \mathbb{E} \Big|  \int_{ \tau_- }^{ \tau_+  } \nabla \partial_{i} d \phi  \Big|^{2}
		\lesssim |x|^{2}  \int_{ \tau_- }^{ \tau_+  }\frac{ d \tau }{ L^2 } .
\end{align*}
In view of the arbitrariness of the scales $ \tau_{-} $ and $ \tau_+ $ this implies \eqref{E: covariation estimate 1 b}.

\medskip

\textit{Argument for \eqref{E: covariation estimate 1}.} By the Kunita-Watanabe inequality we have
	\begin{align*}
		& \Big|  \int_{ \tau_- }^{ \tau_+  } d [ \partial_i \phi (0) \cdot (\partial_j \phi (x) - \partial_j \phi ( 0 ) )] \Big| \\
		&\leq \Big( \int_{ \tau_- }^{ \tau_+  } \frac{1}{L} d [ \partial_i \phi \cdot \partial_i \phi ] \Big)^{\frac{1}{2}}
		\Big(  \int_{ \tau_- }^{ \tau_+  } L d [ (\partial_{j} \phi (x) - \partial_{j} \phi (0) ) \cdot ( \partial_j \phi (x) - \partial_{j} \phi (0 )) ] \Big)^{\frac{1}{2}}
	\end{align*}
	for any two scales $ \tau_{-}, \tau_{+} $. Hence we may deduce from \eqref{E: covariation estimate 1 b} and the bound $d [ {\rm tr} \, \nabla \phi (\nabla \phi)^{*} ] \lesssim d \tau$ that
	\begin{align*}
		\Big|  \int_{ \tau_- }^{ \tau_+  }  d[ \partial_{j} \phi^{i} (0) \, (\partial_{\ell} \phi^{i}(x) - \partial_{\ell} \phi^{i}(0) )] \Big| \lesssim | x |  \int_{ \tau_- }^{ \tau_+  } \frac{ d\tau }{ L }.
	\end{align*}
Again, since the scales $ \tau_{-} $ and $ \tau_{+} $ are arbitrary this implies \eqref{E: covariation estimate 1}.

\medskip

\textit{Argument for \eqref{E: covariation estimate 2}.} Again, we proceed by Taylor, this time in the form of
\begin{align*}
\int_{ \tau_- }^{ \tau_+  } ( d \phi  (x) - d \phi (0) - \nabla \phi (0)^{\dagger} x )
=  \int_0^1 ds  \int_{ \tau_- }^{ \tau_+  } \nabla ( d \phi ( s x ) - d \phi (0) )^{\dagger} x .
\end{align*}
By virtue of Jensen's inequality, followed by \eqref{E: covariation estimate 1 b}, this identity implies
\begin{align*}
	& \Big[  \int_{ \tau_- }^{ \tau_+  } ( d \phi (x) - d \phi (0) - \nabla d \phi (0)^{\dagger} x ) \cdot  \int_{ \tau_- }^{ \tau_+  } ( d \phi (x) - d \phi (0) - \nabla d \phi (0)^{\dagger} x )  \Big] \\
	& \leq \int_0^1 ds \Big[  \int_{ \tau_- }^{ \tau_+  } \nabla ( d \phi ( s x ) - d \phi (0) )^{\dagger} x \cdot   \int_{ \tau_- }^{ \tau_+  } \nabla ( d \phi ( s x ) - d \phi (0) )^{\dagger} x \Big]
	\lesssim |x|^{4}  \int_{ \tau_- }^{ \tau_+  } \frac{ d \tau }{ L^{2} },
	\end{align*}
as claimed.

\medskip

\textit{Argument for \eqref{E: covariation estimate 2 b}.} Similar to the argument for \eqref{E: covariation estimate 1}, the Kunita-Watanabe inequality and \eqref{E: covariation estimate 2} imply \eqref{E: covariation estimate 2 b}.
\end{proof}

\subsection{Proof of the main theorems}\label{sec:mr2}

\begin{proof}[Proof of Theorem \ref{thm:mr01}.] With the notation of Lemma \ref{lem:mr03}, we claim that, if $ L^{2} = T + 1 $, then
\begin{align}\label{pmr06cont}
\E \fint_0^T dt \, | u ( x, t ) - u ( 0, t ) - F_{ \tau_*, \tau }^{ \dagger } x |^2
\lesssim \varepsilon^2 | x |^2 \Big(  \max \{ 1, \frac{ \tilde\lambda_L }{ \tilde\lambda_{ * } } \}  + \tilde\lambda_L \Big).
\end{align}
Since $\tilde{\lambda}_{L}, \tilde{\lambda}_{*} \geq 1$, the desired conclusion \eqref{ao03} then follows from Lemma \ref{L: properties of F} (iii).  To see that \eqref{pmr06cont} holds, we recall from \cite[Lemma 9.1 and (98)]{CMOW} that we have\footnote{Here the proof only requires the energy estimate, hence the fact that it is stated for dimension $n = 2$ is no issue.}
	\begin{equation} \label{E: CMOW bound}
		\E \fint_{0}^{T} dt \, | \nabla u(0,t) - \nabla \tilde{u}_L (0) |^{2} \lesssim \varepsilon^{2} \tilde{\lambda}_{L}.
	\end{equation}
Combining this with the result of Lemma \ref{lem:mr03} and the elementary estimate $\E |v(x) - v(0)|^{2}$ $\leq |x|^{2} \E |\nabla v(0)|^{2}$ for a stationary field $v$, we obtain \eqref{pmr06cont}.
\end{proof}

We now use Theorem \ref{thm:mr01} and the bound \eqref{E: CMOW bound} to establish intermittency of the expected position $u$.  To do so, we will utilize the non-equi-integrability result on the process $F$ from Section \ref{S: intermittency}, cf.~Lemma \ref{lem-weakinter}, and its higher-dimensional analogue Lemma \ref{up02}, which will be proved in the upcoming work \cite{KMOW25}.



\begin{proof}[Proof of Theorem \ref{thm:mr02}.]

The first statement is obtained as a corollary of  Theorem \ref{thm:mr01}: We postprocess \eqref{ao03} by using the elementary $ | a^2 - b^2 | \leq ( 2 | b | + | a - b |  ) | a - b | $ and Cauchy-Schwarz to obtain
\begin{align*}
	& \Big| \mathbb{E}\fint_0^Tdt \, \frac{1}{|x|^2}|u(x,t)-u(0,t) |^2 - \mathbb{E} \frac{ 1 }{|x|^2}| F_{\tau(|x|^2),\tau(T)}^\dagger x |^2 \Big| \\
	& \stackrel{ \eqref{ao03} }{ \lesssim } \Big( \sqrt{ \mathbb{E} \frac{ 1 }{|x|^2}| F_{\tau(|x|^2),\tau(T)}^\dagger x |^2 } + \sqrt{ \varepsilon^{2} \mathbb{E}  |F_{0,\tau(T)}|^2 } \, \Big) \sqrt{ \varepsilon^{2} \mathbb{E} |F_{0,\tau(T)}|^2 } .
\end{align*}
By isotropy the term $ \mathbb{E} | F_{\tau(|x|^2),\tau(T)}^\dagger x |^2 $ is independent of the direction $ \frac{ x }{ | x | } $ and thus equal to
\begin{align*}
	\mathbb{E} | F_{\tau(|x|^2),\tau(T)}^\dagger x |^2
	= | x |^2 \fint_{ \mathbb{S}^{ n-1 } } d \hat y \,  \mathbb{E} | F_{\tau(|x|^2),\tau(T)}^\dagger \hat y |^2
	= \frac{ 1 }{ n }  | x |^2 \mathbb{E} | F_{\tau(|x|^2),\tau(T)} |^2 .
\end{align*}
Therefore the normalization in Lemma \ref{L: properties of F} (iii) (recall that $ \tau = \ln \lambda $ by \eqref{ao12}) and the identity $  \mathbb{E} | F_{\tau(|x|^2),\tau(T)} |^2 $ $= \mathbb{E} | F_{ 0 , \tau(T) - \tau(|x|^2) } |^2 $ provided $ \tau(T) - \tau(|x|^2) \geq 0 $ (recall that $ F $ has stationary increments) imply
\begin{align*}
	\mathbb{E} \frac{1}{|x|^{2}}| F_{\tau(|x|^2),\tau(T)}^\dagger x |^2
	= \max \{ 1 ,  \frac{ \lambda ( T )}{ \lambda ( | x |^{2} )} \} .
\end{align*}
On the other hand, by the assumption on the regime \eqref{vregime}, 
	\begin{equation*}
		\varepsilon^{2} \mathbb{E} | F_{0,\tau(T)} |^{2} = \varepsilon^{2} n \lambda(T) \ll \max \{ 1, \frac{ \lambda(T) }{ \lambda(|x|^{2}) } \}
	\end{equation*} 
and thus \eqref{ao06} holds.

\medskip

\medskip

Next, we will use Lemma \ref{lem-weakinter} \& \ref{up02} in form of the statement
\begin{align}\label{inter01}
	\mathbb{E} | F_{ \tau } |^2 I( | F_{ \tau } |^2 \ge \frac{ 1 }{ c }  \, \mathbb{E}^{ \frac{ n + 4 }{ n + 2 } } | F_{ \tau } |^2 )
	\geq \frac{1}{c} \E | F_{ \tau } |^2 
	\quad { \rm provided } ~ \tau \gg_{ n } 1 ,
\end{align}
and for some constant $ c > 1 $
to prove the lower bound \eqref{ao77}. To this end let us write $ X = \fint_{ \mathbb{S}^{n-1}} d \hat{y} \, \fint_{0}^{T} dt \, \frac{1}{|x|^{2}} |u(|x| \hat{y} ,t) - u(0,t)|^{2}$ and $Y = \frac{ 1 }{ n }  |F_{\tau ( | x |^2 ) , \tau ( T ) }|^{2}$, where here $ d \hat{y} $ is the surface measure on the unit circle $ \mathbb{S}^{n-1}$.  It will be convenient to let $X(\hat{y}) =  \fint_{0}^{T} dt \, \frac{1}{|x|^{2}} |u(|x| \hat{y},t) - u(0,t)|^{2}$ and $Y(\hat{y}) =  | F_{\tau ( | x |^2 ) , \tau ( T ) }^{\dagger} \hat{y} |^{2} $ so that $X = \fint_{ \mathbb{S}^{n -1}} d \hat{y} \, X(\hat{y})$ and likewise $Y =  \fint_{ \mathbb{S}^{ n - 1 } } d\hat{y} \,  Y(\hat{y}) $.   Our first goal is to prove the inequality
\begin{align}\label{pmr02}
\E X I ( X < { \textstyle \frac{1}{2} \frac{1}{c} } r^{ \frac{n+4}{n+2} }  )
\leq \Big( n - \frac{2}{c} \Big) r
\quad { \rm with } \quad
r = \frac{ 1 }{ n } \frac{ \lambda( T ) }{ \lambda( | x |^2 ) } ,
\end{align}
whenever $ \varepsilon^2 \lambda( | x | ) \ll 1 $ and $\tau - \tau ( | x |^2 )  \gg 1 $. With the present notation, by isotropy, \eqref{ao06} says that asymptotically $ nr  \approx \E X $. Next, we post-process this statement into our claim.

\medskip

We start from the observation
\begin{align}\label{pmr02prev}
X I ( X < { \textstyle \frac{1}{2} \frac{ 1 }{ c } } r^{\frac{n+4}{n+2}} )
\leq Y I ( Y < { \textstyle \frac{1}{c} } r^{\frac{n+4}{n+2}} ) + | X - Y | .
\end{align}
On the error term $ | X - Y | $ we proceed as follows. We apply Jensen's inequality and the elementary inequality $ || a |^2 - | b |^2|  \leq | a - b | ( 2 | a | + | a - b | )  $ (see above) to obtain
\begin{align*}
&\E | X(\hat{y}) - Y(\hat{y}) |
 \lesssim \Big( \E \fint_0^T dt \, | F_{ \tau ( | x |^2 ), \tau ( T ) } ^{ \dagger } \hat{y}  - |x|^{-1} ( u( |x| \hat{y}, t ) - u ( 0, t ) ) |^2 \Big)^{ \frac{1}{2} } \\
&\qquad \times \bigg( \Big( \E \fint_0^T dt \,  | F_{ \tau ( | x |^2 ) , \tau ( T ) } |^2 \Big)^{ \frac{1}{2} } + \Big( \E \fint_0^T dt \,  |  F_{ \tau ( | x |^2 ), \tau ( T ) } ^{ \dagger } \hat{y} - |x|^{-1} ( u( |x| \hat{y} , t ) - u ( 0, t ) ) |^2 \Big)^{ \frac{1}{2} } \bigg).
\end{align*}
Applying Theorem \ref{thm:mr01}, \eqref{ao06}, and the bound $\mathbb{E}|X - Y| \leq \fint_{ \mathbb{S}^{n-1}} \, d \hat{y} \, \mathbb{E}|X( \hat{y} ) - Y( \hat{y} )|$, we find
\begin{align*}
\E | X - Y |
\lesssim \Big( \varepsilon^2 \frac{\lambda^2(T)}{\lambda(|x|^2)} \Big)^{ \frac{1}{2} } + \varepsilon^{2} \lambda(T)
\ll  \frac{1}{n} \frac{\lambda(T)}{\lambda(|x|^2)} = r
\quad { \rm provided} \quad
\varepsilon^2 \lambda( | x |^2 ) \ll 1 .
\end{align*}
Therefore, \eqref{pmr02prev} and \eqref{inter01} together imply
\begin{align*}
\E X I ( X < { \textstyle \frac{1}{2} \frac{1}{c} } r^{ \frac{ n + 4 }{ n + 2 } } )
\leq \E Y I ( Y < { \textstyle \frac{1}{c} } r^{ \frac{n + 4 }{ n + 2 } } ) + \E | X - Y |
\leq \Big( n - \frac{ 2 }{ c } \Big) r
\quad { \rm with } \quad
r =  \frac{1}{n} \frac{ \lambda( T ) }{ \lambda( | x |^2 ) } ,
\end{align*}
whenever $ \varepsilon^2 \lambda( | x | ) \ll 1 $ and $\tau - \tau ( | x |^2 ) \gg 1$, as claimed in \eqref{pmr02}.

\medskip

To conclude the proof in case $\tau - \tau( | x |^2 ) \gg 1$, let us note that \eqref{ao06} and \eqref{pmr02} together imply
\begin{align}\label{meet01}
\E X I ( X \geq { \textstyle \frac{1}{ 2 c }  } r^{ \frac{ n + 4 }{ n + 2 } } )
\geq \frac{2}{c} r ,
\end{align}
which by Markov's inequality (recall that $ n r \approx \E X $) shows
\begin{align*}
\E X^p
\geq \frac{2}{c} \Big( \frac{1}{ 2c } r^{ \frac{ n + 4 }{ n + 2 }  } \Big)^{p - 1} r.
\end{align*}
After applying Jensen's inequality and isotropy in law, this yields
	\begin{align*}
		\mathbb{E} \Big( \fint_{0}^{T} dt \, \frac{1}{|x|^{2}} |u(x,t) - u(0,t)|^{2} \Big)^{p} \geq \E X^{p} \gtrsim r^{1 + \frac{ n + 4 }{ n + 2 } (p - 1)}.
	\end{align*}
Recalling again that by \eqref{ao06} we have $ nr \approx \E \fint_{0}^{T} dt \, \frac{1}{|x|^{2}} |u(x,t) - u(0,t)|^{2}$, this implies the claim in the regime $\tau - \tau ( | x |^2 ) \gg 1 $.

\medskip

If $ \tau - \tau ( | x |^2 ) \gg 1 $ does not hold, we may apply Jensen's inequality to \eqref{ao06} to obtain
\begin{align*}
	1 \lesssim \Big( \mathbb{E}\fint_0^Tdt\frac{1}{|x|^2}|u(x,t)-u(0,t)|^2 \Big)^p
	\leq   \mathbb{E} \Big( \fint_0^Tdt\frac{1}{|x|^2}|u(x,t)-u(0,t)|^2 \Big)^p,
\end{align*}
so that the proposition holds trivially in this case.
\end{proof}


\section*{Acknowledgements}

We thank B\'{a}lint T\'{o}th, Fabio Toninelli, Giuseppe Cannizzaro, and Ofer Zeitouni for helpful discussions.  We also benefitted from conversations with Tomasz Komorowski.  Finally, we acknowledge Anna Wienhard, Corentin Le Bars, and \c{S}efika Kuzgun for discussions concerning random walks and diffusion on groups.

\appendix

\section{Representation theory}\label{reptheory}

Throughout this work, we need some classical statements from representation theory of $  \textbf{SO} ( n ) $ and $ \textbf{O} ( n ) $ that we summarize in this appendix. We start with the following lemma classifying invariant two-tensors.

\begin{lemma}\label{lemRep02}
    (i) We have
	\begin{align*}
		{ \rm dim } \big\{ \text{symmetric} ~ \textbf{SO} ( n ) \text{-invariant} ~ \text{bilinear} ~ \text{forms} ~ \text{on} ~ { \rm tangent ~ space } \big\} = 1 ;
	\end{align*}
	more precisely, this space is spanned by the tensor $ ( \dot x , \dot y ) \mapsto \dot x \cdot \dot y $. In particular, every $ \textbf{SO} ( n ) $-invariant symmetric bilinear form is also $ \textbf{O}(n) $-invariant.
	
    \medskip
    
    (ii) We have
	\begin{align*}
		{ \rm dim } \big\{ \textbf{O} ( n ) \text{-invariant} ~ \text{endomorphisms} ~ \text{of} ~ { \rm co tangent ~ space } \big\} = 1 ,
	\end{align*}
and consequently the space is spanned by the identity $ { \rm id } $.
\end{lemma}

For the reader's convenience we provide a proof of Lemma \ref{lemRep02} further below. Of importance for us is the following corollary of point (i).

\begin{corollary}\label{corRep01}
	We have 
	\begin{align*}
		{ \rm dim } \, \big\{  \mathcal{F} c : \mathbb{S}^{ n - 1 } \rightarrow  { \rm tangent ~ space } \otimes_{ \rm sym }  { \rm tangent ~ space } ~ { \rm satisfying } ~ \eqref{ao02c} ~ \& ~ \eqref{ao24repeat} \} = 1 .
	\end{align*}
	More precisely, the space is spanned by the function $ k \mapsto \delta^{ i j } e_i \otimes e_j - k^* \otimes k^* $.
\end{corollary}

\begin{proof}[Proof of Corollary \ref{corRep01}]
	Recall that we can identify the tensor space $ { \rm tangent ~ space } \otimes_{ \rm sym } { \rm tangent ~ space } $ with the symmetric bilinear forms on $ { \rm co tangent ~ space } $. In this vein, given $ k \in \mathbb{S}^{ n - 1 }$, \eqref{ao02c} states that
	\begin{align*}
		\mathcal{F} c ( k )
		~ { \rm is ~ an ~ isotropic ~ symmetric ~ bilinear ~ form ~ on ~ the ~ orthogonal ~ complement ~ of } ~ k .
	\end{align*}
	Thus by Lemma \ref{lemRep02} (i)\footnote{technically speaking applied to forms on a subspace of $ { \rm co tangent ~ space} $} the tensor $ \mathcal{F} c ( k ) $ is a multiple of the Euclidean inner product on the orthogonal complement of $ k $:
	\begin{align*}
		\mathcal{F} c ( k ) . \xi \otimes \xi'
		= { \rm constant } ( k ) \, \xi \cdot \xi'
		\quad { \rm provided } ~ \xi \cdot k = \xi' \cdot k = 0 .
	\end{align*}
	By \eqref{ao24repeat} we know that $ \mathcal{F} c ( k ) $ vanishes whenever $ \xi $ or $ \xi' $ are parallel to $ k $. In particular the form only depends on the orthogonal projection $ \xi - ( \xi \cdot k ) k $. Using the identity $ ( \xi - ( \xi \cdot k ) k ) \cdot ( \xi' - ( \xi' \cdot k ) k )= \xi \cdot \xi' - ( \xi \cdot k ) ( \xi' \cdot k ) $, we find
	\begin{align*}
		( \xi , \xi' ) \mapsto \mathcal{F} c ( k ) . \xi \otimes \xi' 
		\quad { \rm is ~ a ~ multiple } ~ { \rm constant } ( k ) ~ { \rm of } \quad
		( \xi , \xi' ) \mapsto \xi \cdot \xi' - ( \xi \cdot k ) ( \xi' \cdot k ) .
	\end{align*}
Since the latter form already shares the invariance \eqref{ao02c}, we learn that the potentially $ k $-dependent constant is isotropic and thus independent of $ k $. This proves the desired one-dimensionality.
\end{proof}

Next, we provide a proof of Lemma \ref{lemRep02}. Note that in point (ii) we do not assume symmetry; thus it does not immediately follow from point (i).

\begin{proof}[Proof of Lemma \ref{lemRep02}]
	The first part is a standard consequence of Schur's lemma, which we recall for the readers convenience: Given a bilinear form $f$ on tangent space, there is a symmetric endomorphism $T$ such that $f(\dot{x}, \dot{y}) = \dot{x} \cdot T \dot{y}$.  By the spectral theorem, we can (orthogonally) diagonalize $T$. By the $ \textbf{SO} ( n ) $-invariance we know that all the eigenspaces are $ \textbf{SO} ( n ) $-invariant. Since any two unit vectors may be transformed into each other by an element of $ \textbf{SO} ( n ) $, we learn that the eigenspaces of $ T $ are either trivial or equal to the full space. Therefore, $ T $ has only one eigenvalue and is thus a multiple of the identity, which proves $f$ is a multiple of the Euclidean inner product.

	\medskip
	
	For the second item, let $ T $ be an $ \textbf{O} ( n ) $-invariant endomorphism. For some vector $ e $ we consider the reflection sending $ e $ to $ - e $ and fixing orthogonal complement of $ e $. By the $ \textbf{O} ( n ) $-invariance of $ T $ we conclude $ v \cdot Te = 0 $ for any $ v $ orthogonal to $ e $. Thus $ Te $ is parallel to $ e $,~i.e.~$ e $ is an eigenvector. Since $ e $ was arbitrary, we learn that any vector is an eigenvector for $ T $. Since any two unit vectors are mapped into each other by a orthogonal transformation, $ T $ has exactly one eigenvalue and is thus a multiple of the identity.
\end{proof}

The next lemma deals with the classification of $ \textbf{O}(n) $-invariant symmetric bilinear forms on $ { \rm End } ( { \rm tangent ~ space } ) $.  By $\textbf{O}(n)$-invariant, we mean a form $\bullet$ such that 
	\begin{align*}
		(O E O^{-1}) \bullet ( O E' O^{-1} ) = E \bullet E' \quad \text{for each} \quad E \in \text{End}( { \rm tangent ~ space }), \, \, O \in \textbf{O}(n).
	\end{align*}

\begin{lemma}\label{lemRep03}

(i) Every $\textbf{O}(n)$-invariant symmetric bilinear form on $ { \rm End } ( { \rm tangent ~ space } ) $ factors $ { \rm End } ( { \rm tangent ~ space } ) $ orthogonally into the space of isotropic, trace-free and symmetric, and skew-symmetric endomorphisms.

\medskip
	
	(ii) We have
	\begin{align*}
		{ \rm dim } \big\{ \text{symmetric} ~\textbf{O} ( n ) \text{-invariant} ~ \text{bilinear} ~ \text{forms} ~ \text{on} ~ \text{symmetric} ~ \text{trace-free} ~ \text{endomorphisms} \big\} = 1.
	\end{align*}
	%
	
	(iii) We have
	\begin{align*}
		{ \rm dim } \big\{ \text{symmetric} ~\textbf{O} ( n ) \text{-invariant} ~ \text{bilinear} ~ \text{forms} ~ \text{on} ~ \text{skew-symmetric} ~ \text{endomorphisms} \big\} = 1.
	\end{align*}
	%

\end{lemma}

\subsection{Self-contained argument for Lemma \ref{lemRep03}} 

We refer to \cite[Theorem 5.3.3]{GoodmanWallach} for a proof of Lemma \ref{lemRep03} in the case of complex vector spaces. Since we are not aware of a reference covering the real case, we provide a self-contained argument in this subsection.

\medskip

The argument involves elementary representation theory.  As throughout the paper, we use the fact that the group $\textbf{O}(n)$ acts on $\text{End}( { \rm tangent ~ space })$ via conjugation $E \mapsto O E O^{-1}$.  Recall that an $\textbf{O}(n)$-invariant linear subspace $V$ of $\text{End}( { \rm tangent ~ space } )$ is called \emph{irreducible} if its only $\textbf{O}(n)$-invariant subspaces are the zero subspace $\{0\}$ and $V$ itself.  As stated in the next proposition, in this setting, the irreducible subspaces are known.

\medskip

	\begin{proposition} \label{P: irreducibility} The $\textbf{O}(n)$-invariant subspaces $\text{span}\{\text{id}\}$, $\text{Sym}_{0}$, and $\text{Skew}$ of $\text{End}( { \rm tangent ~ space } )$ are irreducible. \end{proposition}
	
We will show that items (i)-(iii) of Lemma \ref{lemRep03} are all consequences of irreducibility.  In fact, (ii) and (iii) are equivalent to the irreducibility of $\text{span}\{\text{id}\}$ and $\text{Sym}_{0}$, respectively.  We start with the proof of (i) next.
	
	
		\begin{proof}[Proof of Lemma \ref{lemRep03} (i)] Give a symmetric bilinear form $\bullet$ on the endomorphisms of tangent space, there is a symmetric linear map $\mathbf{T}$ such that $E \bullet E' = {\rm tr} \, E' (\mathbf{T} E)$ for any pair of endomorphisms $E,E'$.  Assume that $\bullet$ is $\textbf{O}(n)$-invariant.  It then follows that $\mathbf{T}$ commutes with the action of $\textbf{O}(n)$, that is, $O (\mathbf{T} E) O^{-1} = \mathbf{T}(O E O^{-1})$ for any $O \in \textbf{O}(n)$ and endomorphism $E$.  To see that $\text{span}\{\text{id}\}$, $\text{Sym}_{0}$, and $\text{Skew}$ are orthogonal under $\bullet$, it only remains to show that if $V, W \in \{\text{span}\{\text{id}\},\text{Sym}_{0},\text{Skew}\}$ and $\pi_{V}$ and $\pi_{W}$ denote the orthogonal projections onto $V$ and $W$, respectively, then $\pi_{W} \circ \mathbf{T} \circ \pi_{V} = 0$ whenever $V \neq W$.
		 		
		\medskip
		
	To see this, note that $\pi_{W} \circ \mathbf{T} \circ \pi_{V}$ commutes with the action of $\textbf{O}(n)$ since $\mathbf{T}$, $\pi_{W}$, and $\pi_{V}$ all do.  Thus, the kernel and range of $\pi_{W} \circ \mathbf{T} \circ \pi_{V}$ are $\textbf{O}(n)$-invariant subspaces of $V$ and $W$, respectively.  Therefore, by irreducibility (Proposition \ref{P: irreducibility}), either $\pi_{W} \circ \mathbf{T} \circ \pi_{V} = 0$ or $\pi_{W} \circ \mathbf{T} \circ \pi_{V}$ is an isomorphism between $V$ and $W$.  However, this second option is impossible: If $n > 2$, then the dimensions of $\text{span}\{\text{id}\}$, $\text{Sym}_{0}$, and $\text{Skew}$ all differ so they cannot be isomorphic.  If $n = 2$, then this argument fails since $\text{span}\{\text{id}\}$ and $\text{Skew}$ are both one-dimensional.  Nonetheless, an isomorphism between them cannot commute with the $\mathbf{O}(n)$ action.  This is because while every $O \in \mathbf{O}(n)$ acts as the identity transformation on $\text{span}\{\text{id}\}$, this is not true of the action of $\mathbf{O}(n)$ on $\text{Skew}$ (e.g., by Lemma \ref{lemRep02} (ii)). \end{proof}

In the proof of Proposition \ref{P: irreducibility}, we proceed by treating the symmetric and skew-symmetric subspaces separately.

	\begin{lemma} \label{L: symmetric irreducibility} Let $\text{Sym} = \text{span} \{\text{id}\} \oplus \text{Sym}_{0}$ denote the space of symmetric endomorphisms of tangent space.  Let $\pi_{\text{span}\{\text{id}\}}$ and $\pi_{\text{Sym}_{0}}$ denote the orthogonal projections of $\text{Sym}$ onto $\text{span}\{\text{id}\}$ and $\text{Sym}_{0}$, respectively.  If $\mathbf{T}$ is a symmetric  linear map on $\text{Sym}$ that commutes with the action of $\textbf{O}(n)$, then there are constants $\lambda_{\text{id}}, \lambda_{0} \in \mathbb{R}$ such that 
		\begin{equation*}
			\mathbf{T} = \lambda_{\text{id}} \pi_{\text{span}\{\text{id}\}} + \lambda_{0} \pi_{\text{Sym}_{0}}.
		\end{equation*}
	\end{lemma}
	
		\begin{proof} By the spectral theorem and linearity, it suffices to prove that there are constants $\lambda,\mu \in \mathbb{R}$ such that $\mathbf{T} \dot{x}^{*} \otimes \dot{x} = \lambda \text{id} + \mu \dot{x}^{*} \otimes \dot{x}^{*}$ for any unit vector $\dot{x}$ in tangent space.  Since $\mathbf{T}$ commutes with the action of $\textbf{O}(n)$, it suffices to consider the case when $\dot{x} = e_{1}$.  
		
%
		\medskip
		
		Again, by $\textbf{O}(n)$-invariance, $\mathbf{T} e^{1} \otimes e_{1}$ is invariant under conjugation by elements of $\textbf{O}(n)$ that fix $e_{1}$ itself.  That is, if $O \in \textbf{O}(n)$ and $Oe_{1} = e_{1}$, then $O (\mathbf{T} e^{1} \otimes e_{1})O^{-1} = \mathbf{T} e^{1} \otimes e_{1}$.  From this, we deduce that any eigenspace of $\mathbf{T} e^{1} \otimes e_{1}$ is invariant under the action of such an $O$, from which it follows that eigenspaces belong to the set $\{\{0\}, \text{span} \{e_{1}\}, \{e_{1}\}^{\perp}, { \rm tangent ~ space } \}$.  Since the eigenspaces of $\mathbf{T} e^{1} \otimes e_{1}$ span ${ \rm tangent ~ space }$,  we conclude that $\mathbf{T} e^{1} \otimes e_{1} = \lambda \text{id} + \mu e^{1} \otimes e_{1}$ for some $\lambda, \mu \in \mathbb{R}$. \end{proof}

    Due to the equivalence between symmetric bilinear forms and self-adjoint endomorphisms, the irreducibility of $\text{Sym}_{0}$ is equivalent to statement (ii) in Lemma \ref{lemRep03}, as shown next.

    \begin{proof}[Proof of Lemma \ref{lemRep03} (ii)] If $\bullet$ is a symmetric bilinear form on $\text{Sym}_{0}$, then there is a symmetric $\mathbf{T} \in \text{End}(\text{Sym}_{0})$ such that $E \bullet E' = {\rm tr} \, E^{*} (\mathbf{T}E')$ for any $E,E' \in \text{Sym}_{0}$.  If $\bullet$ is $\textbf{O}(n)$-invariant, then $\mathbf{T}$ commutes with the action of $\textbf{O}(n)$.  Extending $\mathbf{T}$ to the entire space of symmetric endomorphisms by setting $\mathbf{T} \text{id} = 0$, we invoke Lemma \ref{L: symmetric irreducibility} to deduce that $\mathbf{T}$ is equal to some constant $\lambda$ times the orthogonal projection onto $\text{Sym}_{0}$.  This implies that $E \bullet E' = \lambda {\rm tr} \, E^{*} E'$.  \end{proof}

    It only remains to treat $\text{Skew}$.  Here it is convenient to work with bilinear forms $\bullet$ rather than maps $\mathbf{T}$.
		
	\begin{lemma} \label{L: skew symmetric irreducibility} If $\bullet$ is an $\textbf{O}(n)$-invariant symmetric bilinear form on $\text{Skew}$, then there is a constant $\lambda \in \mathbb{R}$ such that $E \bullet E' = \lambda {\rm tr} \, E^{*} E'$. \end{lemma}

    In particular, this lemma directly implies Lemma \ref{lemRep03} (iii).
	
		\begin{proof}  Let $\bullet$ be an $\textbf{O}(n)$-invariant symmetric bilinear form on $\text{Skew}$.  Given tangent vectors $\dot{x},\dot{y}$, denote by $\dot{x} \otimes_{\text{skew}} \dot{y}$ the antisymmetric tensor product 
			\begin{equation*}
				\dot{x} \otimes_{\text{skew}} \dot{y} = \dot{x}^{*} \otimes \dot{y} - \dot{y}^{*} \otimes \dot{x}.
			\end{equation*}
		Since the endomorphisms $e_{i} \otimes_{\text{skew}} e_{j}$ with $1 \leq i < j \leq n$ form a basis of $\text{Skew}$, we only need to prove that there is a constant $\lambda \in \mathbb{R}$ such that 
        \begin{equation*}
            ( e_{i} \otimes_{\text{skew}} e_{j} ) \bullet ( e_{m} \otimes_{\text{skew}} e_{k} ) = \lambda \delta_{i m} \delta_{j k} \quad \text{for any} \, \, i < j, m < k.
        \end{equation*}
        This is straightforward: We use the $\textbf{O}(n)$-invariance of $\bullet$ together with the fact that any permutation of $\sigma$ of $\{1,2,\dots,n\}$ gives rise to an element $O_{\sigma} \in \textbf{O}(n)$ with $O_{\sigma} e_{\ell} = e_{\sigma(\ell)}$ for all $\ell$.  Since any pair $i < j$ can be mapped to another pair $m < k$ by some permutation $\sigma$, the diagonal entry $( e_{i} \otimes_{\text{skew}} e_{j}) \bullet (e_{i} \otimes_{\text{skew}} e_{j})$ does not depend on $\{i,j\}$.  Concerning the off-diagonal entries, if $\{i,j\} \neq \{m,k\}$, then, assuming without loss of generality that $m \notin \{i,j\}$, we fix the $O \in \textbf{O}(n)$ such that $O e_{m} = - e_{m}$ and $O e_{\ell} = e_{\ell}$ for $\ell \neq m$ to find
            \begin{equation*}
                ( e_{i} \otimes_{\text{skew}} e_{j} ) \bullet ( e_{m} \otimes_{\text{skew}} e_{k} ) = ( Oe_{i} \otimes_{\text{skew}} Oe_{j} ) \bullet ( Oe_{m} \otimes_{\text{skew}} Oe_{k} ) = -  ( e_{i} \otimes_{\text{skew}} e_{j} ) \bullet ( e_{m} \otimes_{\text{skew}} e_{k} ).
            \end{equation*}
		This proves that the off-diagonal entries all vanish.\end{proof}
        
        \begin{proof}[Proof of Proposition \ref{P: irreducibility}] The irreducibility of $\text{Sym}_{0}$ follows from Lemma \ref{L: symmetric irreducibility}:  If $V \subseteq \text{Sym}_{0}$ is a linear subspace invariant under the action of $\textbf{O}(n)$, then the orthogonal projection $\pi_{V}$ of $\text{Sym}$ onto $V$ commutes with the action of $\textbf{O}(n)$.  Thus, Lemma \ref{L: symmetric irreducibility} implies that $\pi_{V} = \pi_{\text{Sym}_{0}}$ or $\pi_{V} = 0$, that is, either $V = \text{Sym}_{0}$ or $V = \{0\}$.

        \medskip
        
         The irreducibility of $\text{Skew}$ follows from an analogous argument: Given an $\textbf{O}(n)$-invariant subspace $V \subseteq \text{Skew}$, consider the orthogonal projection $\pi_{V}$ of $\text{Skew}$ onto $V$.  If we define the symmetric bilinear form $\bullet$ by $E \bullet E' = {\rm tr} \, (\pi_{V} E)^{*} (\pi_{V} E')$, then $\bullet$ is $\textbf{O}(n)$-invariant since $V$ is.  Therefore, Lemma \ref{L: skew symmetric irreducibility} implies $E \bullet E' = \lambda {\rm tr} \, E^{*} E'$ for some $\lambda \in \mathbb{R}$, which can only hold if $V = \{0\}$ or $V =\text{Skew}$.
         
         \medskip
         
         Finally, we observe that $\text{span}\{\text{id}\}$ is irreducible since it is one-dimensional. \end{proof}

\section{Proof of Remark \ref{rmk02} -- anisotropic case}\label{pfrmk02}

\medskip

We start from (\ref{tao17y}) in the form of
\begin{align}\label{tao17z}
d\tilde a=d[\phi\otimes b],
\end{align}
where (\ref{tao17x}) is replaced by $-\nabla.\tilde a_L\nabla d\phi=db$, which is solved by
\begin{align*}
{\mathcal F}d\phi=\frac{1}{k.\tilde a_Lk}{\mathcal F}db;
\end{align*}
note that the latter makes sense as a pathwise definition since $\frac{1}{k.\tilde a_Lk}$ is
smooth on the support $\{L|k|=1\}$ of ${\mathcal F}b$, see the remarks after definition
(\ref{cov04}). Hence $\phi_L$ inherits from $b_L$ the properties of being a stationary
centered Gaussian with independent increments. In terms of the covariance tensor
$c$ of $b$, cf.~Subsection \ref{ss:X}, and definition (\ref{cov04}), we have
\begin{align*}
\mathbb{E}(\phi_{L_+}-\phi_L)\otimes(b_{L_+}-b_L)
= \frac{ 1 }{ ( 2 \pi )^{ \frac{ n }{ 2 } } } \int_{\mathbb{R}^n}dk \, I(L_+^{-1}\le|k|<L^{-1})(k.\tilde a_Lk)^{-1}
{\mathcal F}c(k).
\end{align*}
In view of (\ref{ao02d}), the integrand $(k.\tilde a_Lk)^{-1}
{\mathcal F}c(k)$  has homogeneity $-n$ in $|k|$, so that we obtain for
the quadratic variation
\begin{align*}
d[\phi\otimes b]
= \frac{ 1 }{ ( 2 \pi )^{ \frac{ n }{ 2 } } } \big(\int_{\mathbb{S}^{n-1}}dk \, (k.\tilde a_Lk)^{-1} {\mathcal F}c(k)\big)\frac{dL}{L}.
\end{align*}
Hence in view of the form (\ref{ao08}) of ${\mathcal F}c$ as a covariance tensor
an isotropic ensemble of divergence-free fields, and in view of
the definition (\ref{lambda}) of $\tilde\lambda$, (\ref{tao17z}) turns into the 
deterministic evolution equation
\begin{align}\label{blaao16}
d\tilde a=f(\tilde a_L)\tilde\lambda_Ld\tilde\lambda\quad\mbox{where}\quad
f(a) := c ( n ) \int_{\mathbb{S}^{n-1}}dk \, (k.ak)^{-1}( I-k^*\otimes k^* )
\end{align}
for some positive constant $c(n)$ only depending on $n$.
Since the defining vector field $f=f(a)$ on positive-definite symmetric bilinear forms $a$ 
obviously satisfies the homogeneity
\begin{align}\label{blaao02bis}
f(sa)=s^{-1}f(a)\quad\mbox{for all}\;s\in(0,\infty),
\end{align}
we obtain in terms of the new variables
\begin{align*}
\hat a_L=\frac{\tilde a_L}{\tilde\lambda_L}\quad\mbox{and}\quad\tau=\ln\tilde\lambda_L,
\end{align*}
cf.~(\ref{forFelix06}), the autonomous dynamical system
\begin{align}\label{blaao05}
\frac{d\hat a}{d\tau}=f(\hat a)-\hat a.
\end{align}
Since in view of Lemma \ref{L: construction of proxies}, $\hat a\equiv I $ must be a solution of (\ref{blaao05}),
we learn that the normalization of $f$ is such that
\begin{align}\label{blaao03}
f( I )= I,
\end{align}
which we argue to be the global attractor, see Lemma \ref{blaL:aniso}

\medskip

Before doing so, we show that the stationary point (\ref{blaao03}) is stable. More precisely,
we will show that ${\rm id}-Df(I)$ is strictly monotone\footnote{with respect to the standard
inner product given by ${\rm tr}a^*a$}, 
where $Df(I)$ denotes the Fr\'echet derivative
of $f$ in $I$, which is an endomorphism on the space of symmetric bilinear forms $\dot a$. 
This is a consequence of its explicit identification
\begin{align}\label{lo01}
Df(I)\dot a=\frac{2}{(n-1)(n+2)}\dot a-\frac{n+1}{(n-1)(n+2)}(I^*.\dot a)I,
\end{align}
noting that $\frac{2}{(n-1)(n+2)}<1$. 

\medskip

It thus remains to establish (\ref{lo01}).
Appealing to\footnote{the l.~h.~s.~defines a symmetric form that is invariant
under the action of $\textbf{O}(n)$ and thus is a multiple of $I$;
the constant is determined by contraction with $I^*$} 
$\int_{\mathbb{S}^{n-1}}dk \, k^*\otimes k^*$ $=\frac{1}{n}I$, 
we learn from the normalization (\ref{blaao03}) that the implicit constant 
in definition (\ref{blaao16}) is given by $c(n)=\frac{n}{n-1}$.
Hence we obtain from deriving (\ref{blaao16}) at $I$ in direction of $\dot a$ 
\begin{align*}
Df(I)\dot a=-\frac{n}{n-1}\int_{\mathbb{S}^{n-1}}dk \, k.\dot a k(I-k^*\otimes k^*)
=\frac{n}{n-1}\big(\int_{\mathbb{S}^{n-1}}dk \, (k.\dot a k)k^*\otimes k^*
-\frac{1}{n}(I^*.\dot a)I\big).
\end{align*}
It thus remains to identify the quartic integral. Indeed, since the space of
totally symmetric $\textbf{O}(n)$-invariant quartic forms on tangent space
is one-dimensional, we must have
%
\begin{align*}
\int_{\mathbb{S}^{n-1}}dk\prod_{i=1}^4(k.\dot x_i)
=c(n)\sum_{\text{permutations}\;\sigma}(\dot x_{\sigma(1)}\cdot\dot x_{\sigma(2)})
(\dot x_{\sigma(3)}\cdot\dot x_{\sigma(4)})
\end{align*}
for some (new) constant $c(n)$ only depending on dimension $n$.
This specifies to
\begin{align*}
\int_{\mathbb{S}^{n-1}}dk \, (k.\dot x)^2(k.\dot y)^2
=c(n)\big(2(\dot x\cdot \dot y)^2+|\dot x|^2|\dot y|^2\big),
\end{align*}
from which we learn that $c(n)=\frac{1}{n(n+2)}$ by taking $\dot x$, $\dot y$ from
an two copies of an orthonormal basis and summing over both. Since any $\dot a$
can be written as a linear combination of $\dot x\otimes\dot x$'s, this translates into
the desired
\begin{align*}
\int_{\mathbb{S}^{n-1}}dk \, (k.\dot a k)k^*\otimes k^*
=\frac{1}{n(n+2)}\big(2\dot a+(I^*.\dot a)I\big).
\end{align*}

\begin{lemma}\label{blaL:aniso}
For any solution $\hat a$ of (\ref{blaao05}) we have
\begin{align}\label{blaao17}
\lim_{\tau\uparrow\infty}\hat a = I .
\end{align}
\end{lemma}

\begin{proof} 
We start by collecting structural properties of the vector field $f$ on
the space of symmetric and positive-definite bilinear form $a$ (of cotangent space), cf.~
(\ref{blaao16}). On the one hand, as a consequence of the isotropy of the ensemble
of $b$'s in the sense of (\ref{ao02c}), $f$ must be $\textbf{SO}(n)$-equivariant in the sense that, for a given $O \in \textbf{SO}(n)$,
	\begin{align} \label{blaao02}
		\text{if} \quad \tilde{a} ( \xi, \xi' ) = a ( O^{\dagger} \xi, O^{\dagger} \xi' ), \quad \text{then} \quad 
		f(\tilde{a}) (\xi, \xi') = f(a) (O^{\dagger} \xi, O^{\dagger} \xi').
	\end{align}

The equivariance (\ref{blaao02}) can also be directly derived from the representation
(\ref{blaao16}) of $f$.
On the other hand, we claim that
$f$ is strict monotone decreasing w.~r.~t.~the ordering on symmetric forms:
\begin{align}\label{blaao04}
a'\le a\;\mbox{and}\;a'\not=a\quad\Longrightarrow\quad
f(a')>f(a),
\end{align}
where the latter means that $f(a)-f(a')$ is positive definite. The strictness can
be seen as follows: Suppose that there exists a non-zero $\xi$ such that
$\xi.(f(a')-f(a))\xi=0$; by the representation (\ref{blaao16}) of $f$ this implies
\begin{align*}
0=\int_{\mathbb{S}^{n-1}}dk \, \big((k.a k)^{-1}-(k.a'k)^{-1}\big)\,\xi.( I -k^*\otimes k^*)\xi.
\end{align*}
Since the second factor of the integrand is positive for almost every $ k \in \mathbb{S}^{ n - 1 } $ (w.r.t.~the uniform measure on $ \mathbb{S}^{ n - 1 } $); and the first factor is non-positive for all $k$,
the latter must be zero for almost every -- and by continuity every -- $k$,
which by homogeneity implies $a'=a$.

\medskip

In view of (\ref{blaao02}),
we may restrict to forms $a$ that are diagonal with respect to the same orthonormal basis.
To this end, let us start by arguing that in view of (\ref{blaao02}), the symmetric bilinear forms
$a$ and $f(a)$ have the same (orthogonal) principal axes: If (the cotangent vectors)
$e^1,e^2$ are two distinct principal axes of $a$, we claim that $f(a)(e^1,e^2)=0$.
This can be seen by considering the reflection $O$ with $Oe^2=-e^2$ (and thus $Oe^1=e^1$), 
which since $e^2$ is a principal axes leaves $a$ invariant, 
so that by (\ref{blaao02}) it also leaves $f(a)$ invariant,
which by linearity implies $f(a)(e^1,e^2)$ $=f(a)(Oe^1,Oe^2)$ $=-f(a)(e^1,e^2)$.

\medskip

We translate this insight from bilinear forms on cotangent space
to elements of the tensor product of the tangent space: For any orthonormal basis
$\{e_1, \hdots, e_n\}$ of tangent space,
\begin{align*}
a=\sum_{i=1}^n\mu_ie_i\otimes e_i\quad\Longrightarrow\quad
f(a)=\sum_{i=1}^nf(\mu_i,\{\mu_j\}_{j\not=i})e_i\otimes e_i,
\end{align*}
where by an abuse of notation we use the same symbol $f$ to denote the equivariant
vector field $f=f(a)$ and its representation as a function $f=f(\mu_1,\mu_2,\hdots,\mu_n)$ which, 
as the above notation is meant to suggest,
is invariant under permuting the last $n-1$ entries.
Then (\ref{blaao05}) reduces
to an evolution of the unordered set $\{\mu_1,\hdots,\mu_n\}$, which is of the form 
\begin{align}\label{blaao06}
\frac{d\mu_i}{d\tau}=f(\mu_i,\{\mu_j\}_{j\not=i})-\mu_i.
\end{align}
The strict monotonicity (\ref{blaao04}) translates into
\begin{align}\label{blaao07}
\mu_i '\le\mu_i \;\mbox{for all}\;i \;\;\mbox{and}\;\;\mu_i '\not=\mu_i \;\mbox{for some}\; i
\quad\Longrightarrow\quad
f(\mu_i ',\{\mu_j '\}_{j \not= i })>f(\mu_i ,\{\mu_j \}_{j \not = i }).
\end{align}

\medskip

We start by establishing that
\begin{align}\label{blaao09}
\frac{\mu_j }{\mu_i }\left\{\begin{array}{c}>\\<\end{array}\right\}1\quad\Longrightarrow\quad
\frac{d}{d\tau}\frac{\mu_j }{\mu_i }\left\{\begin{array}{c}<\\>\end{array}\right\}0.
\end{align}
Indeed, for notational simplicity let us replace $(j , i )$ with $ j \neq i $ by $( i ,1)$ with $ i \neq 1 $, treat
$\mu_1$ as the num\'eraire, and consider the ratio $\sigma_i :=\frac{\mu_i }{\mu_1}$:
%
\begin{align*}
\mu_1^2\frac{d\sigma_i}{d\tau}=\mu_1\frac{d\mu_i}{d\tau}-\mu_i \frac{d\mu_1}{d\tau}
\stackrel{(\ref{blaao06})}{=}\mu_1f(\mu_i ,\{\mu_j\}_{j\not=i})-\mu_i f(\mu_1,\{\mu_j\}_{j\not=1}).
\end{align*}
By homogeneity, cf.~(\ref{blaao02bis}), we obtain
\begin{align}\label{blaao08bis}
\mu_1^2\frac{d\sigma_i}{d\tau}
=f(\sigma_i,\{\sigma_j\}_{j\not=1,i}\cup\{1\})
-f(\frac{1}{\sigma_i},\{\frac{\sigma_j}{\sigma_i}\}_{j\not=1,i}\cup\{1\}).
\end{align}
It follows from (\ref{blaao07}) that
$\sigma_i \mapsto f(\sigma_i ,\{\sigma_j\}_{j\not=1,i}\cup\{1\})$ is strictly monotone decreasing,
and that $\sigma_i \mapsto f(\frac{1}{\sigma_i },\{\frac{\sigma_j}{\sigma_i }\}_{j\not=1,i }\cup\{1\})$
is strictly monotone increasing. Hence the r.~h.~s.~of (\ref{blaao08bis}) is strictly
monotone decreasing; it obviously vanishes for $\sigma_i=1$. This establishes (\ref{blaao09}).

\medskip

We now establish the qualitative 
\begin{align}\label{blaao11}
\mu:=(\mu_1,\hdots,\mu_n)\quad\mbox{remains in a compact subset of}\;(0,\infty)^n.
\end{align}
Let $\theta\in(0,1)$ be such that the initial data are contained in
the set $\{\mu|\mu_j \ge\theta\mu_\ell \;\mbox{for all}\; j , \ell \}$.
It follows from (\ref{blaao09}) that this set 
is forward invariant under the dynamical system (\ref{blaao06}) for any fixed $\theta\in(0,1)$:
\begin{align}\label{blaao10}
\mu_j \ge\theta\mu_\ell \quad\mbox{for all}\; j,\ell .
\end{align}
Let us now consider 
\begin{align*}
\mu_*=\min_{j=1,\hdots,n}\mu_j \quad\mbox{so that by (\ref{blaao10})}\quad\mu_j\le\frac{1}{\theta}\mu_*.
\end{align*}
Hence it follows from (\ref{blaao10}) in conjunction with the non-strict form of 
monotonicity (\ref{blaao07}) 
\begin{align*}
f(\mu_i ,\{\mu_j \}_{j\not=i })\ge f(\frac{\mu_*}{\theta},\{\frac{\mu_*}{\theta}\}).
\end{align*}
By the homogeneity (\ref{blaao02bis}) and the normalization (\ref{blaao03}) this implies
\begin{align*}
f(\mu_i ,\{\mu_j\}_{j \not=i })\ge \frac{\theta}{\mu_*},
\end{align*}
so that by (\ref{blaao06}) 
\begin{align}\label{blaao12}
\frac{d\mu_*}{d\tau}\ge\frac{\theta}{\mu_*}-\mu_*.
\end{align}
Likewise, 
\begin{align}\label{blaao13}
\mu^*:=\max_{ j =1,\hdots,n}\mu_j \quad\mbox{satisfies}\quad
\frac{d\mu^*}{d\tau}\le\frac{1}{\theta\mu^*}-\mu^*.
\end{align}

Both differential inequalities combine to (\ref{blaao11}).

\medskip

We now are in a position to establish (\ref{blaao17}) in the form of
\begin{align}\label{blaao14}
\lim_{\tau\uparrow\infty}\mu_i =1.
\end{align}
Define $\theta^{*}$ by 
\begin{align*}
\theta^*(\tau):=\max\{\theta ~ | ~ \mu_j (\tau)\ge\theta\mu_\ell (\tau)\;\mbox{for all}\;j,\ell\}.
\end{align*}
We claim that $\theta^{*}(\tau) \to 1$ as $\tau \to \infty$.  To see this, for any indices $i,j$ such that $\mu_{j}(\tau = 0 ) < \mu_{i}(\tau = 0)$ and any $\theta \in (0,1)$, it is convenient to define $T_{ji}(\theta)$ by
\begin{equation*}
T_{ji}(\theta) = \sup \{ \tau > 0 \, \mid \, \mu_{j}(\tau) \leq \theta \mu_{i}(\tau) \}.
    \end{equation*}
Since $\mu$ solves the ODE \eqref{blaao06}, there is a continuous function $g_{ji}$ in $(0,\infty)^{n}$ such that
\begin{align*}
\frac{d}{d\tau} \frac{\mu_{j}}{\mu_{i}} = g_{ji}(\mu_{1},\dots,\mu_{n}) ,
    \end{align*}
     see also \eqref{blaao08bis}.
Thus, by compactness (see \eqref{blaao11}) and continuity, the merely qualitative positivity in \eqref{blaao09} can be upgraded to the existence of a constant $c_{ji}(\theta) > 0$ such that
\begin{align*}
\frac{d}{d\tau} \frac{\mu_{j}}{\mu_{i}} \geq c_{ji}(\theta) \quad \text{in} \, \, (0,T_{ji}(\theta)).
\end{align*}
From this, it follows that, for any given $\theta \in (0,1)$, $T_{ji}(\theta) < \infty$ holds, hence
\begin{align*}
\limsup_{\tau \to \infty} \theta^{*} \geq \theta.
\end{align*}
This proves $\theta^{*}(\tau) \to 1$ as claimed.

\medskip

%
%

By the same argument as for (\ref{blaao12}) \& (\ref{blaao13}) we obtain
\begin{align*}
\frac{d\mu_*}{d\tau}\ge\frac{\theta^*}{\mu_*}-\mu_*\quad\mbox{and}\quad
\frac{d\mu^*}{d\tau}\le\frac{1}{\theta^*\mu^*}-\mu^*
\end{align*}
Rewriting the above equations in the form $\frac{d}{d\tau} e^{2 \tau} \mu_{*}^{2} \geq 2\theta^{*} e^{2 \tau}$ and $\frac{d}{d \tau} e^{2 \tau} \mu^{*2} \leq \frac{2}{\theta^{*}} e^{2 \tau}$, we can integrate from an arbitrarily large time $\tau_{0}$ to time $\tau = +\infty$ to deduce that $\liminf_{\tau\uparrow\infty}\mu_*\ge 1\ge\limsup_{\tau\uparrow\infty}\mu^*$,
which amounts to (\ref{blaao14}).
\end{proof}

\section{Computing Integrals}\label{AppIntegrals}

Recall from Lemma \ref{lemR01} \& \eqref{lambda} that the three variables $ \tau $, $ \tilde\lambda $ and $ L $ are related by $ \tau = \ln \tilde\lambda $, $ \tilde\lambda^2 = 1 + \varepsilon^2 \ln L $. For our analysis, it is important to know the asymptotics of the following integrals: 
\begin{align}
	e^{ p \tau_* } \int_{ 0 } ^{ \tau_* }  d \tau \, \frac{ L^2 ( \tau ) }{ e^{ p \tau } } 
	& \lesssim_{ p } \varepsilon^2  \Big( \frac{  L_{ * }  }{ \tilde\lambda_{ * } } \Big)^ 2
	\quad { \rm for ~ any } ~ 0 < \varepsilon \leq 1 , ~ p > 0 ,  \label{int01} \\
	\int_{ \tau_* }^{ \infty } d\tau\, \frac{ \tilde\lambda_{ L( \tau ) } ^2 }{ L^p ( \tau ) }
	& \lesssim_{ p } \frac{ \varepsilon^2 }{ L_{ * }^{ p } } 
	\quad { \rm for ~ any } ~ p > 0 , \label{int02} \\
	\int_{ \tau_* }^{ \infty } d\tau \, \frac{1}{ \tilde\lambda^p_{ L ( \tau ) } } &\lesssim_{ p } \frac{ 1 }{ \tilde\lambda_*^{p} }
	\quad { \rm for ~ any } ~ p > 0 , \label{int03}
\end{align}
or equivalently due to the relation $  d \tau = \frac{ 1 }{ 2 } \frac{ d \tilde\lambda ^2 } { \tilde\lambda^2 }  $
\begin{align}
	\tilde\lambda_{ * }^{ p } \int_{ 1 } ^{ \tilde\lambda_*^{2} }  d \tilde\lambda ^2 \, \frac{ L^2 ( \tilde\lambda ) }{ \tilde\lambda^{ p + 2 } } 
	& \lesssim_{ p } \varepsilon^2  \Big( \frac{  L_{ * }  }{ \tilde\lambda_{ * } } \Big)^ 2
	\quad { \rm for ~ any } ~ 0 < \varepsilon \leq 1 , p > 0 ,  \label{int01a} \\
	\int_{ \tilde\lambda_*^2 }^{ \infty } d\tilde\lambda^2 \, \frac{ 1 }{ L^p ( \tilde\lambda ) }
	& \lesssim_{ p } \frac{ \varepsilon^2 }{ L_{ * }^{ p } } 
	\quad { \rm for ~ any } ~ p > 0 , \label{int02a} \\
	\int_{ \tilde\lambda_*^2 }^{ \infty } d\tilde\lambda^2 \, \frac{1}{ \tilde\lambda^{ p + 2 } } &\lesssim_{ p } \frac{ 1 }{ \tilde\lambda_*^{p} }
	\quad { \rm for ~ any } ~ p > 0 . \label{int03a}
\end{align}

\medskip

Integral \eqref{int03a} (and thus \eqref{int03}) follows from $ \frac{ 2 }{ p } \frac{ d }{ d \tilde\lambda^2 } \frac{ 1 }{ \tilde\lambda^p } =  - \frac{ 1 }{ \tilde\lambda^{ p + 2 } } $. In the other integrals it is crucial to capture the constant of order $ \varepsilon^2 $. In \eqref{int02a} (and hence \eqref{int02}) this follows from the relation $ L^{ - p } ( \tilde\lambda )= e^{ - \frac{ p }{ \varepsilon^2 } ( \tilde\lambda^2 - 1 ) } $, which implies $ - \frac{ d }{ d \tilde\lambda^2 }  \frac{ \varepsilon^2 }{ 2 p }  L^{ - p } ( \tilde\lambda ) =  L^{ - p } ( \tilde\lambda ) $. We use the same relation on \eqref{int01a} to integrate by parts: Since the integrand is positive, we have
\begin{align*}
	\int_{ 1 } ^{ \tilde\lambda_*^{2} }  d \tilde\lambda ^2 \, \frac{ L^2 ( \tilde\lambda ) }{ \tilde\lambda^{ p + 2 } } 
	&= \frac{ \varepsilon^2 }{ 2 } \int_{ 1 } ^{ \tilde\lambda_*^{2} }  d \tilde\lambda ^2 \, \frac{ 1 }{ \tilde\lambda^{ p + 2 } } \frac{ d }{ d \tilde\lambda^2 } L^2 ( \tilde\lambda ) \\
	&\leq \frac{ \varepsilon^2 }{ 2 } \frac{ L^2 ( \tilde\lambda_* ) }{ \tilde\lambda_*^{ p + 2 } } 
	+ ( \frac{ p }{ 2 } + 1 ) \frac{ \varepsilon^2 }{ 2 }   \int_{ 1 } ^{ \tilde\lambda_*^{2} }  d \tilde\lambda^2 \, \frac{  1  }{ \tilde\lambda^{ 2 } } 
 \frac{  L^2 ( \tilde\lambda )  }{ \tilde\lambda^{ p + 2 } }  .
\end{align*}
For $ \varepsilon^2  \leq ( \frac{ p }{ 2 } + 1 )^{ - 1 }  $ we can absorb the latter integral to the l.~h.~s., which implies the claim. In the opposite regime $ ( \frac{ p }{ 2 } + 1 )^{ - 1 } \leq \varepsilon^2  \leq 1 $ we argue as follows: We split the latter integral at $ \tilde\lambda^2 = \frac{p}{2} + 1 $. On the interval $ \frac{ p }{ 2 } + 1 \leq \tilde\lambda^2 \leq \tilde\lambda_{*}^2 $ we can use the additional factor $ \frac{ 1 }{ \tilde\lambda^2 } \leq ( \frac{p}{2} + 1 )^{ - 1 } $ to bound the latter integral by $ \leq \frac{ 1 }{ 2 } \int_{ 1 } ^{ \tilde\lambda_*^{2} }  d \tilde\lambda^2 \,  \frac{  L^2 ( \tilde\lambda )  }{ \tilde\lambda^{ p + 2 } }  $. On the remaining interval $ 1 \leq \tilde\lambda^2 \leq \frac{p}{2} + 1 $, due to the lower bound on $ \varepsilon $, the function $ \tilde\lambda \mapsto \frac{ L^2 ( \tilde\lambda ) }{ \tilde\lambda^{ p + 2 } } $ may be estimated solely in terms of $ p $. The same holds for its integral, which thus implies for the union of the intervals
\begin{align*}
( \frac{ p }{ 2 } + 1 )  \int_{ 1 } ^{ \tilde\lambda_*^{2} }  d \tilde\lambda^2 \, \frac{  1  }{ \tilde\lambda^{ 2 } }  \frac{  L^2 ( \tilde\lambda )  }{ \tilde\lambda^{ p + 2 } } -
\frac{ 1 }{ 2 } \int_{ 1 } ^{ \tilde\lambda_*^{2} }  d \tilde\lambda^2 \,  \frac{  L^2 ( \tilde\lambda )  }{ \tilde\lambda^{ p + 2 } } 
\lesssim_{ p } 1
\lesssim_{ p } \frac{  L^2 ( \tilde\lambda_* )  }{ \tilde\lambda_*^{ p + 2 } } ,
\end{align*}
where all the constants are uniform in $ ( \frac{ p }{ 2 } + 1 )^{ - 1 } \leq \varepsilon^2  \leq 1 $ but depend on $ p $.

\section{Proof of Lemma \ref{L: R SDE and lyapunov exponent}} \label{A: stopping}


\begin{proof}[Proof of Lemma \ref{L: R SDE and lyapunov exponent}] We only need to prove the first statement, namely, that $R_{\tau} > 1$ for all $\tau > 0$ with probability one.  Indeed, once this is established, the identity \eqref{E: R quadratic variation} implies that we can define a Brownian motion $w$ via $dw = (R_{\tau}^{2} - 1)^{-\frac{1}{2}} {\rm tr} \, F_{\tau}^{*} F_{\tau} dB$, and then $dR = R_{\tau} d \tau + \sqrt{R^{2}_{\tau} - 1} dw$ follows.  As noted already in the discussion in Section \ref{S: intermittency}, the Stratonovich form of this SDE is precisely \eqref{E: R SDE}.

\medskip

It remains to show that $R_{\tau} > 1$ for all $\tau > 0$. 
 To this end, we claim that there is no loss of generality in assuming that the initial condition $R_{0} > 1$.  Otherwise, if $R_{0} = 1$, then $R$ immediately takes excursions away from one.  Indeed, given an arbitrary $\epsilon > 0$, on the event that $R_{\tau} = 1$ for each $\tau \in [0,\epsilon]$, we know that the quadratic variation $[R R] \equiv 0$ in $[0,\epsilon]$, hence \eqref{E: decomposition} implies $1 = R_{\epsilon} = 1 + \int_{0}^{\epsilon} d \tau R_{\tau} = 1 + \epsilon$, which is absurd.  This proves that $R$ instantaneously ventures away from one.  Therefore, up to suitably stopping $F$ and invoking the strong Markov property, we can assume that $R_{0} > 1$.

 \medskip

 Finally, fix an initial condition $R_{0} > 1$.  Following Feller (see \cite[Section 5.5.C]{karatzas-shreve}), we work with the so-called scale function associated with the SDE $dR = R_{\tau} d\tau + \sqrt{R^{2}_{\tau} - 1} dw$.  To motivate the scale function, observe that if $u$ is harmonic in the sense that it is annihilated by the generator, that is, $r u'(r) + \frac{1}{2} (r^{2} - 1) u''(r) = 0$, in some interval $(a,b) \subseteq (1,\infty)$ containing $R_{0}$, and if $T_{a,b}$ denotes the first time that $R$ hits the boundary $\{a,b\}$, then, by \eqref{E: decomposition} and \eqref{E: R quadratic variation}, the process $u(R_{\tau \wedge T_{a,b}})$ is a local martingale and a standard argument shows that 
  \begin{equation} \label{E: hitting prob}
\mathbb{P}\{R_{T_{a,b}} = a\} = \frac{u(b) - u(R_{0})}{u(b) - u(a)}.
 \end{equation}
At the same time, we can easily derive an explicit representation for $u'$:
 \begin{equation*}
\frac{d}{dr} \{ \ln u' \} = \frac{2 r}{r^{2} - 1}, \quad \text{hence} \quad u' = \frac{\text{const}}{r^{2} - 1}.
 \end{equation*}
Integration of this last expression motivates the introduction of the scale function $u(r) = \ln ( \frac{r - 1}{r + 1} )$, a particular solution of the ODE.
Since $u(1^{+}) = - \infty$ and $u(+\infty) < +\infty$, we easily read off from \eqref{E: hitting prob} that $R$ will escape to infinity without ever hitting one.
 \end{proof}


\end{document}